\documentclass[3p,times,sort&compress]{elsarticle}

\usepackage{amssymb}
\usepackage{subcaption}

\usepackage{booktabs}

\newcommand{\e}{\mathrm{e}}

\journal{}

\usepackage{mathtools}
\DeclareMathOperator{\sgn}{sgn}

\usepackage{hyperref}
\renewcommand{\i}{\mathrm{i}}

\makeatletter
\def\ps@pprintTitle{%
  \let\@oddhead\@empty
  \let\@evenhead\@empty
  \def\@oddfoot{%
    \hbox to \textwidth{%
      \@myfooterfont%
      \ifx\@elsarticlemyfooteralign\@elsarticlemyfooteraligncenter%
        \hfil\@elsarticlemyfooter\hfil%
      \else%
      \ifx\@elsarticlemyfooteralign\@elsarticlemyfooteralignleft%
        \@elsarticlemyfooter\hfill{}%
      \else%
      \ifx\@elsarticlemyfooteralign\@elsarticlemyfooteralignright%
        {}\hfill\@elsarticlemyfooter%
      \else%
        \ifx\@journal\@empty%
          \hfill \@date%
        \else%
          \hfill \@journal \hfill \@date%
        \fi%
      \fi%
      \fi%
      \fi%
    }%
  }%
  \let\@evenfoot\@oddfoot
}
\makeatother % This removes "Preprint submitted to" while retaining the date. To remove the footer entire, just do 

\DeclarePairedDelimiterX{\inp}[2]{\langle}{\rangle}{#1, #2}

\newcommand{\tran}{^{\mkern-1.5mu\mathsf T}}

\DeclareMathOperator{\diag}{diag}

\DeclareMathOperator{\arccsc}{arccsc}

\usepackage{bm}

\numberwithin{equation}{section}

\usepackage{amsthm}
\newtheorem{definition}{Definition}
\newtheorem{theorem}{Theorem}
\newtheorem{proposition}{Proposition}
\newtheorem{lemma}{Lemma}
\newtheorem{corollary}{Corollary}
\newtheorem{remark}{Remark}

\usepackage{tikz}

\usepackage{physics2}
\makeatletter
\newcommand\vb{\protect\@ifstar\protect\boldsymbol\protect\mathbf}
\makeatother
\usepackage{fixdif}

\usepackage{todonotes}

\newcommand{\dx}{\d{x}}

\newcommand{\dthe}{\d{\theta}}

\usepackage{derivative}
\derivset{\odv}[switch-*=false]
\derivset{\pdv}[switch-*=false]

\def\bfF{\vb F}
\def\bfP{\vb P}

\def\bftheta{\bm \theta}
\def\bfW{\vb W}

\newcommand{\colr}[1]{#1}

\newcommand{\rev}[1]{#1}

\begin{document}

\begin{frontmatter}

\title{A sparse $hp$-finite element method for piecewise-smooth differential equations with periodic boundary conditions}
\author[imp]{Daniel VandenHeuvel\corref{cor1}}
\author[imp]{Sheehan Olver }
\address[imp]{Department of Mathematics, Imperial College, London, England}
\cortext[cor1]{Corresponding author: d.vandenheuvel24@imperial.ac.uk}

\begin{abstract}
We develop an efficient $hp$-finite element method for piecewise-smooth differential equations with periodic boundary conditions, using orthogonal polynomials defined on circular arcs. The operators derived from this basis are banded and achieve optimal complexity regardless of $h$ or $p$, both for building the discretisation and solving the resulting linear system in the case where the operator is symmetric positive definite. The basis serves as a useful alternative to other bases such as the Fourier or integrated Legendre bases, especially for problems with discontinuities. \colr{We relate the convergence properties of these bases to regions of analyticity in the complex plane, and further use several differential equation examples to demonstrate these properties}. The basis spans the low order eigenfunctions of constant coefficient differential operators, thereby achieving better smoothness properties for time-evolution partial differential equations.
\end{abstract}


\end{frontmatter}

%\linenumbers


\section{Introduction}

In this work, we develop a periodic piecewise smooth basis for solving differential equations with periodic boundary conditions in one dimension. The basis can either be viewed as a re-orthogonalised variant of Fourier series, or via the change of variables $x = \cos \theta$, $y = \sin \theta$, as  bivariate orthogonal polynomials defined over arcs of the unit circle. The construction leads to sparse banded operators for the weak Laplacian and mass matrices, with bandwidth independent of the truncation of the basis, hence the complexity of numerical solutions matches that of Fourier series. Since our basis is periodic, the numerical solution remains periodic and, unlike a standard Fourier approach \cite{boyd2001chebyshev, wright2015extension}, we can easily handle discontinuities by  placing element nodes at the points of discontinuity. While in this work we only consider one-dimensional problems, the basis we develop can be  extended to two-dimensional problems such as radial problems, for example differential equations in a sector.

Differential equations with periodic boundary conditions arise in a variety of applications, such as molecular dynamics \cite{frenkel2002understanding}, quantum mechanics \cite{fausset2007periodic}, and fluid mechanics \cite{combined2006dong}. Applying orthogonal polynomials for solving these differential equations with a spectral element method can be traced back to the hierarchical $p$-finite element method (FEM) basis using integrated Legendre polynomials as introduced by Babu{\v{s}}ka and Szab{\'o} \cite{babuvska1983lecture}, where high-order orthogonal polynomials can be utilised together with sparse operators for the efficient numerical solution of such problems. While other classical bases such as the Fourier or the integrated Legendre basis could be applied to these periodic problems \cite{boyd2001chebyshev, wright2015extension, knook2024quasi, raju2014spectral}, they all have certain limitations. The Fourier basis is not able to handle discontinuities, and other orthogonal bases do not have periodicity as a natural property. For example, the integrated Legendre basis can be used for the spectral element method by making the shape functions themselves periodic \cite{los2015dealing}, \colr{but the periodicity in the resulting solution's derivatives may be lost numerically at the element junctions, in particular at $\theta = \pm T$ for $2T$-periodic problems, when solving problems over time}. \colr{We call this numerical difference in periodicity the \textit{periodic drift} and find that solutions using bases such as the integrated Legendre basis may display growing periodic drift over time, particularly in the solution's derivatives.} The basis we develop in this work gets around this property by, instead of defining the polynomials on the interval  in the $\theta$ coordinate, we define the polynomials on circular arcs instead, meaning in the $x$ and $y$ coordinates. Thus, each individual basis polynomial is itself periodic \rev{when extended to the whole interval} and we obtain a sparse and efficient solver. 

\colr{We find that this basis helps to reduce the effects of periodic drift, an advantage over bases such as the integrated Legendre basis.} A key contribution of this paper is the analysis and comparison of these different bases, that is, the Fourier basis (which we denote $\bfF$), the integrated Legendre basis (which we denote $\bfW^{(b),\bftheta}$ and is polynomial in $\theta$) and the new arc polynomial basis (which we denote $\bfP^{(b),\bftheta}$ and is polynomial in $x$ and $y$). A similar basis \rev{was} previously introduced in \cite{huybrechs2010fourier}\rev{, which defines} \emph{half-range Chebyshev polynomials} for the purpose of Fourier extensions, which are in fact a special case of our arc polynomials \rev{corresponding to} a semicircular arc. \rev{Orthogonal polynomials on arcs were also considered by \cite{olver2021orthogonal}; however an important difference --- and a primary contribution --- of our work is the embedding of such a basis within an $hp$-FEM framework, together with a subsequent convergence analysis.} \colr{The bases we consider in this work are summarised in Table \ref{tab:basis_tab}.}

% \usepackage{booktabs}
\begin{table}[h!]
\centering
\caption{Bases used in this work. The primary contribution of this work is the introduction of the $\vb P^{(b)}$ and $\vb P^{(b), \bm\theta}$ bases, although the periodic forms of the $\vb W^{(b)}$ and $\vb W^{(b), \bm\theta}$ bases from \cite{babuvska1983lecture, knook2024quasi} are also new.}\label{tab:basis_tab}
\begin{tabular}{@{}lll@{}}
\toprule
Name                          & Notation                 & Definition                                                  \\ \midrule
Fourier                       & $\vb F$                  & Trigonometric polynomials. \\
Piecewise Integrated Legendre & $\vb W^{(b), \bm\theta}$ & Periodic hierarchical $p$-FEM basis \cite{babuvska1983lecture}; see \ref{app:integ_legen_def}. \\
Arc Polynomial                & $\vb P^{(b)}$            & Orthogonal trigonometric polynomials on an arc; see Section \ref{sec:ortho_arc_p}. \\
Piecewise Arc Polynomial      & $\vb P^{(b), \bm\theta}$ & Hierarchical $p$-FEM basis built from arc polynomials; see Section \ref{sec:sem2_}. \\ \bottomrule
\end{tabular}
\end{table}

Our introduction of the basis $\bfP^{(b),\bftheta}$ as a replacement for the Fourier basis $\bfF$ is \rev{similar in spirit} to the work in \cite{wright2015extension}. There, the authors extend the Chebfun software \cite{Driscoll2014} by replacing the Chebyshev basis with a Fourier basis and develop the corresponding approximation theory and operator framework, with applications to \rev{areas such as} differential equations and eigenvalue problems. \rev{In this sense, our work can be viewed as having a similar goal: we replace the classical Fourier basis by a more flexible piecewise polynomial basis $\bfP^{(b),\bftheta}$, which may be interpreted as a natural piecewise extension of $\bfF$.} While \rev{the underlying motivation is similar}, the operators required to work with $\bfP^{(b),\bftheta}$ are substantially more delicate to construct.

The paper is structured as follows. In Section \ref{sec:semiclassical_jacobi_basis} we introduce the semiclassical Jacobi polynomials from \cite{papadopoulos2024building, magnus1995painleve} and generalise them to the non-integrable weight $x^a(t-x)^c/(1-x)$, serving as a fundamental component of our spectral element method for their use as bubble functions, basis elements which are supported on a single element. Section \ref{sec:ortho_arc_p} defines orthogonal polynomials on an arc $\{(x, y) : x^2 + y^2 = 1,\, x \geq h\}$, $|h| < 1$, defined in terms of the polynomials from Section \ref{sec:semiclassical_jacobi_basis}. We then use these polynomials in Section \ref{sec:sem2_} to derive our spectral element method, taking the polynomials from Section \ref{sec:ortho_arc_p} and defining them piecewise over an entire circle. Section \ref{sec:analysis__} gives some analysis of our piecewise basis, showing that we can indeed \rev{represent} trigonometric polynomials exactly, and that the convergence rate for expansions in this basis are efficient compared to a periodic version of the integrated Legendre polynomial basis from \cite{knook2024quasi}\colr{, relating this convergence to regions of analyticity in the complex plane}. Finally, we give some examples of solving differential equations in Section \ref{sec:solve_de_} as well as a short analysis of the computational performance of our basis, and give concluding remarks in Section \ref{sec:conc}. 

\section{Semiclassical Jacobi polynomials with $b = -1$}\label{sec:semiclassical_jacobi_basis}

A fundamental building block for the orthogonal polynomials constructed in this work are the \textit{semiclassical Jacobi polynomials}.

\begin{definition}[\cite{papadopoulos2024building, magnus1995painleve}]\label{def1}
The semiclassical Jacobi polynomials with parameters $t$, $a$, $b$, and $c$, are orthogonal polynomials associated with the weight $w^{t, (a, b, c)}(x) := x^a(1-x)^b(t-x)^c$, where $t > 1$ and $a, b > -1$. We denote by $P_n^{t, (a, b, c)}$ the $n$th degree semiclassical Jacobi polynomial. Using quasimatrix notation \cite{townsend2015continuous}, we define
\begin{equation}\label{eq:semiclassical_jacobi}
\vb P^{t, (a, b, c)}(x) := \begin{bmatrix} P_0^{t, (a, b, c)}(x) & P_1^{t, (a, b, c)}(x) & P_2^{t, (a, b, c)}(x) & \cdots \end{bmatrix}.
\end{equation}
For the linear polynomials, it will also be useful to define the notation $P_1^{t, (a, b, c)} := \beta^{t, (a, b, c)}(x - \alpha^{t, (a, b, c)})$. The associated inner product is denoted
\begin{equation}\label{eq:def:inp}
\inp{f}{g}^{t, (a, b, c)} := \int_0^1 f(x)g(x)w^{t,(a,b,c)}(x) \dx.
\end{equation}
We similarly define $\|f\|_{t, (a, b, c)}^2 := \inp{f}{f}^{t, (a, b, c)}$.
\end{definition}

The coefficients $\beta^{t, (a, b, c)}$ and $\alpha^{t, (a, b, c)}$ are derived in \ref{app:linear_coefficients}. In this section, we are interested in extending Definition \ref{def1} to the case $b = -1$, mimicking the Jacobi polynomials
$$
P_k^{(-1,0)}(x) = \begin{cases} 1 & k = 0 \\
					-(1-x) P_{k-1}^{(1,0)}(x)/2 & \hbox{otherwise}
					\end{cases}
$$
as these can then be used as \rev{bubble functions in our} spectral element method developed later. The weight in this case is not integrable and so the inner product \eqref{eq:def:inp} only gives a formal definition of orthogonality for these polynomials, but we can still make a definition that gives us sparse operators. Similarly to the negative parameter Jacobi polynomials in \cite{guo2009generalized}, we make the following definition.

\begin{definition}
The semiclassical Jacobi polynomials with $b = -1$ are defined in terms of the $b = 1$ polynomials via
\begin{equation}\label{eq:semiclassical_jacobi_bneg1_def}
\vb P^{t, (a, -1, c)}(x) := \begin{bmatrix} 1 & (1 - x)\vb P^{t, (a, 1, c)}(x) \end{bmatrix}.
\end{equation}
\end{definition}

The cases for $b > -1$ were already developed in \cite{papadopoulos2024building}, and their software implementation is available from \cite{SemiPoly.jl2024}. The first five polynomials with $t = 2$, $a = \pm 1/2$, and $c = \pm 1/2$ are shown in Figure \ref{fig:semijac_polys}.

\begin{figure}[h!]
\centering
\includegraphics[width=\textwidth]{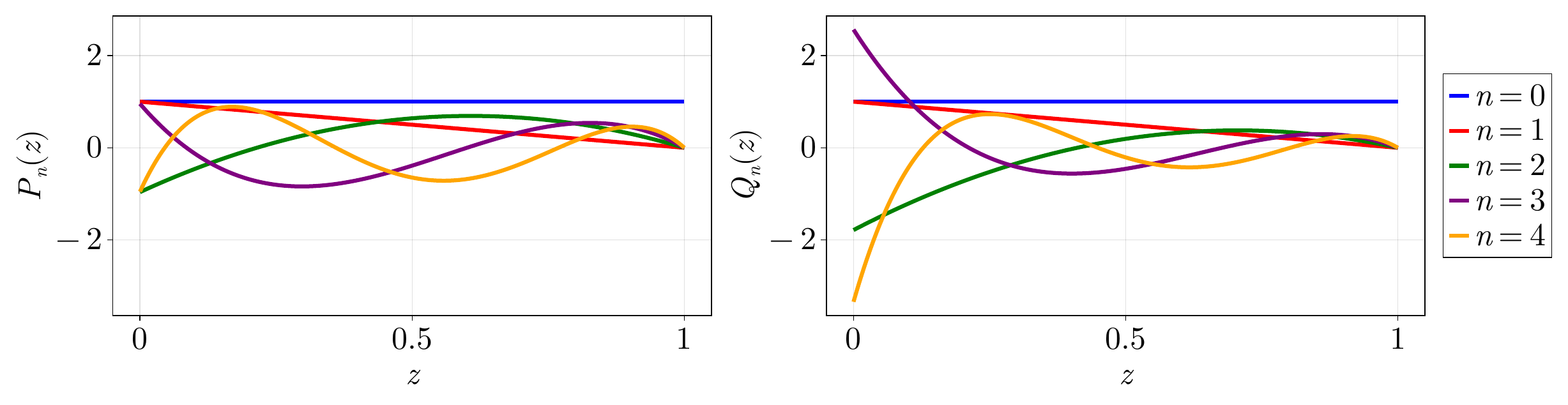}
\caption{The first five semiclassical Jacobi polynomials with $b=-1$ for $P_n^{2, (-1/2, -1, -1/2)}$ (left) and $P_n^{2, (1/2, -1, 1/2)}$ (right).}\label{fig:semijac_polys}
\end{figure}

\begin{proposition}\label{prop1}
The $\vb P^{t, (a, -1, c)}$ polynomials satisfy the three-term recurrence 
\begin{equation}
xP_n^{t, (a, -1, c)}(x) = c_n^{t, (a, -1, c)}P_{n-1}^{t, (a, -1, c)}(x) + a_n^{t, (a, -1, c)}P_n^{t, (a, -1, c)}(x) + b_n^{t, (a, -1, c)}P_{n+1}^{t, (a, -1, c)}(x), \quad n = 0, 1, 2, \ldots,
\end{equation}
where
\begin{equation}\label{eq:semiclassical_jacobi_3tr_bneg1}
a_n^{t, (a, -1, c)} = \begin{cases} 1 & n = 0, \\ a_{n-1}^{t, (a, 1, c)} & n \geq 1, \end{cases} \quad b_n^{t, (a, -1, c)} = \begin{cases} -1 & n = 0, \\ b_{n-1}^{t, (a, 1, c)} & n \geq 1, \end{cases} \quad c_n^{t, (a, -1, c)} = \begin{cases} 0 & n = 1, \\ c_{n-1}^{t, (a, 1, c)} & n \geq 2, \end{cases}
\end{equation}
using the notation $xP_n^{t, (a, b, c)}(x) = c_n^{t, (a, b, c)}P_{n-1}^{t, (a, b, c)}(x) + a_n^{t, (a, b, c)}P_n^{t,(a,b,c)} + b_n^{t, (a, b, c)}P_{n+1}^{t, (a, b, c)}(x)$ for the three-term recurrence in general.
\end{proposition}

\begin{proof}
The $n = 0$ case follows immediately by using $P_0^{t, (a, b, c)} = 1$ and $P_1^{t, (a, -1, c)}(x) = 1 - x$. For $n > 0$, the result is derived directly by substituting in the definition \eqref{eq:semiclassical_jacobi_bneg1_def}.
\end{proof}

Proposition \ref{prop1} allows us to efficiently evaluate $P_n^{t, (a, -1, c)}(x)$ using Clenshaw's algorithm \cite{olver2020fast}. Now let us build operators for converting between bases.

\begin{definition}\label{def:connmat}
The upper triangular \emph{connection matrix} $R_{(a, b, c)}^{t, (\alpha, \beta, \gamma)}$ is defined by 
\[
\vb P^{t, (a, b, c)} = \vb P^{t, (\alpha, \beta, \gamma)}R_{(a, b, c)}^{t, (\alpha, \beta, \gamma)},
\] 
or equivalently $P_n^{t, (a, b, c)} = \sum_{j=0}^n r_{jn}P_j^{t, (\alpha, \beta, \gamma)}$ for $n=0,1,2,\ldots$, where $R_{(a, b, c)}^{t, (\alpha, \beta, \gamma)} = \{r_{jn}\}_{j,n=0,1,2,\ldots}$. In particular, $R_{(a, b, c)}^{t, (\alpha, \beta, \gamma)}$ gives the coefficients of $\vb P^{t, (a, b, c)}$ in the $\vb P^{t, (\alpha, \beta, \gamma)}$ basis. Similarly, we define the \emph{weighted connection matrix} $L_{\mathrm a\mathrm b\mathrm c, (a, b, c)}^{t, (\alpha, \beta, \gamma)}$ by 
\[
w^{t, (a, b, c)}\vb P^{t, (a, b, c)} = w^{t, (\alpha, \beta, \gamma)}\vb P^{t, (\alpha, \beta, \gamma)}L_{\mathrm a\mathrm b\mathrm c, (\alpha, \beta, \gamma)}^{t, (\alpha, \beta, \gamma)}.
\]
The Roman subscripts $\mathrm a$, $\mathrm b$, $\mathrm c$ denote the terms present in the weight functions, and may be omitted otherwise. For example, $(1-x)^b \vb P^{t, (a, b, c)} = (1-x)^\beta \vb P^{t, (\alpha, \beta, \gamma)}L_{\mathrm b, (a, b, c)}^{t,(\alpha,\beta,\gamma)}$. 
\end{definition}

\begin{proposition}\label{prop:ra1cmalphbeta_mat_r}
The matrix $R_{(a, -1, c)}^{t, (\alpha, \beta, \gamma)}$ is given by
\begin{equation}
R_{(a, -1, c)}^{t, (\alpha, \beta, \gamma)} = \begin{cases} \diag\left(1, R_{(a, 1, c)}^{t, (\alpha, 1, \gamma)}\right) & \beta = -1, \\[0.3em]
\begin{bmatrix} \vb e_1 & L_{\mathrm b, (a, 1, c)}^{t, (a, 0, c)} \end{bmatrix} & (a, c) = (\alpha, \gamma),\, \beta = 0, \\[0.3em]
R_{(a, 0, c)}^{t, (\alpha, \beta, \gamma)}R_{(a, -1, c)}^{t, (a, 0, c)} & otherwise,
\end{cases}
\end{equation}
where $\vb e_1 = (1, 0, \ldots)\tran$. Moreover, $R_{(a, -1, c)}^{t, (a, 0, c)}$ is upper bidiagonal.
\end{proposition}

\begin{proof}
The case $\beta = -1$ follows immediately from the definition \eqref{eq:semiclassical_jacobi_bneg1_def}. To consider the case where $(a, c) = (\alpha, \gamma)$ and $\beta = 0$, write
\begin{equation}
R_{(a, -1, c)}^{t, (a, 0, c)} = \begin{bmatrix} 
r_{11} & \vb r_{12}\tran \\ \vb 0 & R_{22} \end{bmatrix}.
\end{equation}
The condition $\vb P^{t, (a, -1, c)} = \vb P^{t, (a, 0, c)}R_{(a, -1, c)}^{t, (a, 0, c)}$ then gives $r_{11} = 1$ and, using $P_0^{t, (a, 1, c)} \equiv 1$,
\[
(1 - x)\vb P^{t, (a, 1, c)} = \vb r_{12}\tran + \begin{bmatrix} P_1^{t, (a, 0, c)} & P_2^{t, (a, 0, c)} & \cdots \end{bmatrix}R_{22} = \vb P^{t, (a, 0, c)}\begin{bmatrix} \vb r_{12}\tran \\ R_{22} \end{bmatrix}.
\]
We thus see that the last matrix is $L_{\mathrm b, (a, 1, c)}^{t, (a, 0, c)}$, giving us this case. The fact that $R_{(a, -1, c)}^{t, (a, 0, c)}$ is upper bidiagonal follows from $L_{\mathrm b, (a, 1, c)}^{t, (a, 0, c)}$ being lower bidiagonal \cite{papadopoulos2024building}. Lastly, for the general case, this expression follows from the definition of the connection matrix, where we first increment from $b = -1$ up to $b = 0$ and then, from the $\vb P^{t, (a, 0, c)}$ basis, convert to $\vb P^{t, (\alpha,\beta,\gamma)}$ using $R_{(a, 0, c)}^{t, (\alpha,\beta,\gamma)}$.
\end{proof}

The last component we need for our basis is the differentiation matrix.

\begin{definition}
The matrix $D_{(a, b, c)}^{t, (\alpha, \beta, \gamma)}$ is the matrix defined by 
\begin{equation}
\odv*{\vb P^{t, (a, b, c)}}{x} = \vb P^{t, (\alpha, \beta, \gamma)}D_{(a, b, c)}^{t, (\alpha, \beta, \gamma)}.
\end{equation}
Similar to Definition \ref{def:connmat}, we include Roman subscripts to indicate weighted derivatives, for example the matrix $D_{\mathrm b, (a, b, c)}^{t, (\alpha, \beta, \gamma)}$ is defined by $\mathrm d[(1-x)^b\vb P^{t, (a, b, c)}]/\mathrm dx = (1-x)^\beta\vb P^{t, (\alpha,\beta,\gamma)}D_{\mathrm b, (a, b, c)}^{t, (\alpha,\beta,\gamma)}$.
\end{definition}

\begin{proposition}
The matrix $D_{(a, -1, c)}^{t, (a+1, 0, c+1)}$ is a $(-1, 2)$ banded matrix, meaning a banded matrix with bandwidth $-1$ and upper bandwidth $2$, given by 
\begin{equation}\label{eq:diff_mat_thm}
D_{(a, -1, c)}^{t, (a+1, 0, c+1)} = \begin{bmatrix} \vb 0 & D_{\mathrm b, (a, 1, c)}^{t, (a+1, 0, c+1)} \end{bmatrix}.
\end{equation}
\end{proposition}

\begin{proof}
The first column is zero since $\mathrm d1/\mathrm dx = 0$. For the remaining columns, we have 
\[
\odv*{\left[(1 - x)\vb P^{t, (a, 1, c)}\right]}{x} = \vb P^{t, (a+1, 0, c+1)}D_{\mathrm b, (a, 1, c)}^{t, (a+1, 0, c+1)},
\]
giving \eqref{eq:diff_mat_thm}. Since $D_{\mathrm b, (a, 1, c)}^{t, (a+1, 0, c+1)}$ is upper bidiagonal \cite{papadopoulos2024building}, $D_{(a,-1,c)}^{t,(a+1,0,c+1)}$ is a $(-1, 2)$ banded matrix.
\end{proof}

We now have all the operators that we need for working with this $b = -1$ basis.

\section{Orthogonal polynomials on an arc}\label{sec:ortho_arc_p}

In this section we define orthogonal polynomials on the arc
\begin{equation}\label{eq:arc_def}
\Omega_h = \{(x, y) : x^2 + y^2 = 1, x \geq h\},
\end{equation}
where $|h| < 1$. These polynomials will underpin the basis we use for solving differential equations on a complete interval $-\pi \leq \theta < \pi$, splitting the interval into segments that individually represent arcs of the form $\Omega_h$. Following \cite{olver2021orthogonal}, we consider polynomials orthogonal with respect to the inner product
\begin{align}\label{eq:inp_arc}
\inp{f}{g}^{(b, h)} &:= \int_{-\varphi}^{\varphi} f(\cos\theta, \sin\theta)g(\cos\theta, \sin\theta)w(\cos\theta) \dthe \\
&= \int_h^1 \left[f(x, y)g(x, y) + f(x, -y)g(x, -y)\right]\frac{w(x)}{y} \dx,\label{eq:inp_arc_cartesian}
\end{align}
where $x = \cos\theta$, $y = \sin\theta = \sqrt{1 - x^2}$, $\cos\varphi = h$, and $w(x) = (x - h)^b$ for $b \geq -1$; note that this is only an inner product for $f, g \in \mathbb R[x, y] \setminus \langle x^2 + y^2 - 1 \rangle$, i.e. bivariate polynomials modulo those with a factor of $x^2 + y^2 - 1$. We may leave the dependence on $h$ implicit, writing $\inp{f}{g}^{(b)}$ instead of $\inp{f}{g}^{(b, h)}$, and similarly when convenient we may omit the superscript entirely. We will also use the notation $\|f\|_{(b, h)}^2 := \inp{f}{f}^{(b, h)}$, similarly omitting the subscripts when convenient.

\begin{theorem}
An orthogonal basis for the space of bivariate polynomials of degree $n$ with respect to \eqref{eq:inp_arc}, denoted $\mathcal H_n$, is given by $\{p_n^{(b, h)}\}_{n=0}^\infty \cup \{q_n^{(b, h)}\}_{n=1}^\infty$, where
\begin{align}
p_n^{(b, h)}(x, y) &= P_n^{\tau, (-1/2, b, -1/2)}\left(\sigma\right), \quad n = 0, 1, 2, \ldots, \label{eq:arcp} \\
q_n^{(b, h)}(x, y) &= yP_{n-1}^{\tau, (1/2, b, 1/2)}\left(\sigma\right), \quad n = 1, 2, \ldots, \label{eq:arcq}
\end{align}
defining $\tau := 2/(1-h)$ and $\sigma := (x-1)/(h-1)$. We call these polynomials the \emph{arc polynomials}. Note that $x = (h-1)\sigma + 1$.
\end{theorem}

\begin{proof}
This result can be obtained by applying a similar argument to the proof in Theorem 3.1 of \cite{olver2021orthogonal}, noting that with the given definitions of $\tau$ and $\sigma$ we have
\begin{equation}\label{eq:inp_sj}
\inp{f}{g} = (1-h)^b\int_0^1 \left[f(x, y)g(x, y) + f(x, -y)g(x, -y)\right]w^{\tau, (-1/2, b, -1/2)}(\sigma) \d\sigma. \qedhere
\end{equation}
\end{proof}

\begin{definition}
We denote the arc polynomial basis by 
\begin{align}\label{eq:quasimatrix_p}
\vb P^{(b, h)}(x, y) = \begin{bmatrix} \vb P_0^{(b, h)}(x, y) & \vb P_1^{(b, h)}(x, y) & \vb P_2^{(b, h)}(x, y) & \cdots \end{bmatrix},
\end{align}
where $\vb P_0^{(b, h)}(x, y) = p_0^{(b, h)}(x, y) = 1$ and $\vb P_n^{(b, h)}(x, y) = (q_n^{(b, h)}(x, y), p_n^{(b, h)}(x, y))$ for $n \geq 1$. \rev{We further define}
\begin{align}
\vb p_n^{(b, h)}(x, y) &= \begin{bmatrix} p_0^{(b, h)}(x, y) & p_1^{(b, h)}(x, y) & p_2^{(b, h)}(x, y) & \cdots \end{bmatrix}, \\
\vb q_n^{(b, h)}(x, y) &= \begin{bmatrix} q_1^{(b, h)}(x, y) & q_2^{(b, h)}(x, y) & q_3^{(b, h)}(x, y) & \cdots \end{bmatrix},
\end{align}
When convenient we will omit the dependence on either $h$, writing $\vb P^{(b)}$ for example, or on both $b$ and $h$. We may also write $p_n^{(b, h)}(\theta) := p_n^{(b, h)}(\cos\theta,\sin\theta), q_n^{(b, h)}(\theta) := q_n^{(b, h)}(\cos\theta,\sin\theta)$, and $\vb P^{(b, h)}(\theta) := \vb P^{(b, h)}(\cos\theta,\sin\theta)$.
\end{definition}

Note that these arc polynomials are the same as the half-range Chebyshev polynomials developed in \cite{huybrechs2010fourier} except there the polynomials are defined with $h=0$ while we allow any $|h| < 1$.  We show the first five arc polynomials for each of $p_n$ and $q_n$ in Figure \ref{fig:arcpoly} where $h = 0.2$ and $b = -1$.

\begin{figure}[h!]
\centering
\includegraphics[width=\textwidth]{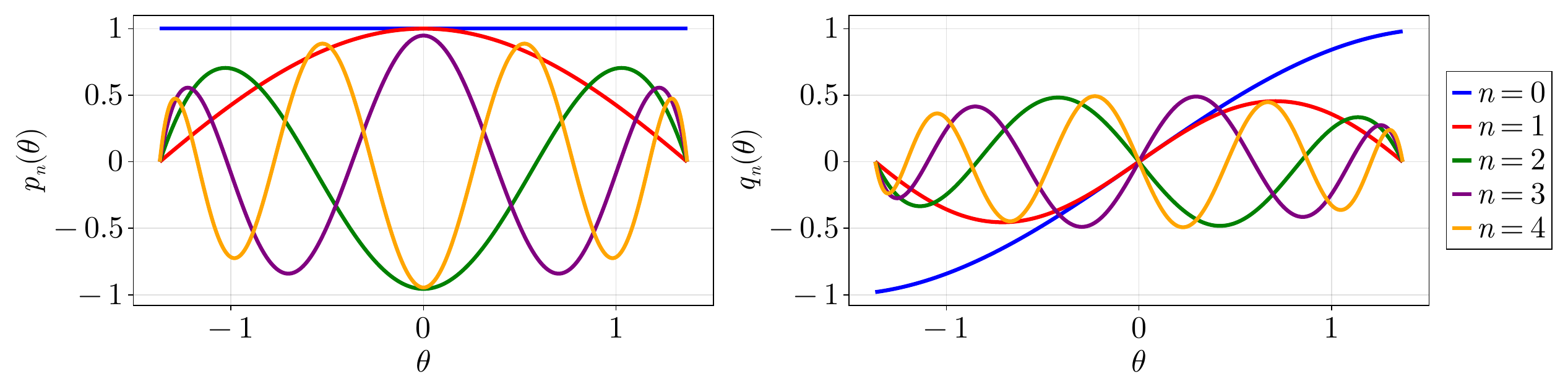}
\caption{The first five arc polynomials with  $b=-1$ for both $p_n^{(0.2, -1)}$ (left) and $q_{n+1}^{(0.2, -1)}$ (right).}\label{fig:arcpoly}
\end{figure}

\subsection{Connection matrix}

The first problem we consider is converting from $\vb P^{(b)}$ to $\vb P^{(b+1)}$, assuming that the value for $h$ remains the same.

\begin{definition}
The matrix $R_{(b_1)}^{(b_2, h)}$ is the matrix such that $\vb P^{(b_1)} = \vb P^{(b_2)}R_{(b_1)}^{(b_2, h)}$. It is assumed that the value of $h$ remains the same. When convenient, we will omit the dependence on $h$, writing $R_{(b_1)}^{(b_2)}$.
\end{definition}

\rev{In the next result, we only consider the case $b_1=b$ and $b_2=b_1+1$. For the case $|b_2-b_1| > 1$, note that the} connection matrix can be written as a product of connection matrices of this form \cite{olver2020fast}, for example $R_{(0)}^{(3)} = R_{(2)}^{(3)}R_{(1)}^{(2)}R_{(0)}^{(1)}$.

\begin{proposition}
The matrix $R_{(b)}^{(b+1)}$ is given by the $(0, 2)$ banded matrix
\begin{equation}\label{eq:conn_mat_rbb+1_02bm_arc_prop4}
R_{(b)}^{(b+1)} = \begin{bmatrix}
a_{00}^{(b)} & 0            & a_{01}^{(b)} &                &                &          \\
               & b_{00}^{(b)} & 0            & b_{01}^{(b)} &                &          \\
               &                & a_{11}^{(b)} & 0            & a_{12}^{(b)} &          \\
               &                &                & b_{11}^{(b)} & 0            & \ddots \\
               &                &                &                & a_{22}^{(b)} & \ddots \\
               &                &                &                &                & \ddots
\end{bmatrix},
\end{equation}
where
\begin{align}\label{eq:arc_connection_p1}
p_n^{(b)}(x, y) &= a_{n-1,n}^{(b)}p_{n-1}^{(b+1)}(x, y) + a_{nn}^{(b)}p_n^{(b+1)}(x, y), \quad n=0,1,2,\ldots, \\ \label{eq:arc_connection_p2}
q_n^{(b)}(x, y) &= b_{n-2,n-1}^{(b)}q_{n-1}^{(b+1)}(x, y) + b_{n-1,n-1}^{(b)}q_n^{(b+1)}(x, y), \quad n=1,2,\ldots.
\end{align}
\end{proposition}

\begin{proof}
The recurrence relationships \eqref{eq:arc_connection_p1}--\eqref{eq:arc_connection_p2} come directly from the definition of the polynomials together with the connection matrices $R_{(-1/2, b, -1/2)}^{\tau, (-1/2,b+1,-1/2)} := \{a_{ij}^{(b)}\}_{i,j=0,1,\ldots}$ and $R_{(1/2,b,1/2)}^{\tau,(1/2,b+1,1/2)} := \{b_{i-1,j-1}^{(b)}\}_{i,j=1,2,\ldots}$, noting that these connection matrices are upper bidiagonal. Interlacing these entries gives the form of $R_{(b)}^{(b+1)}$ shown.
\end{proof}

\subsection{Differentiation matrix}

Now we consider differentiation. Since the differential equations we will be considering are in terms of $\theta$ rather than $x$ and $y$, we are interested in differentiation matrices for the operator
\begin{equation}
\odv*{}{\theta} = x\pdv*{}{y} - y\pdv*{}{x}.
\end{equation}

\begin{definition}
The matrices $D_{(b), \vb p}^{(b+1, h), \vb q}$, $D_{(b), \vb q}^{(b+1, h), \vb p}$, and $D_{(b)}^{(b+1, h)}$ are defined by 
\begin{equation}
\odv*{\vb p^{(b, h)}}{{\theta}} = \vb q^{(b+1, h)}D_{(b), \vb p}^{(b+1, h), \vb q}, \quad \odv*{\vb q^{(b, h)}}{{\theta}} = \vb p^{(b+1, h)}D_{(b), \vb q}^{(b+1, h), \vb p}, \quad \text{and} \quad \odv*{\vb P^{(b, h)}}{\theta} = \vb P^{(b+1, h)}D_{(b)}^{(b+1, h)}.
\end{equation}
When convenient, we will omit the dependence on $h$ in these superscripts.
\end{definition}

\begin{proposition}\label{prop:diffmatpqqp}
The matrices $D_{(b), \vb p}^{(b+1), \vb q}$ and $D_{(b), \vb q}^{(b+1), \vb p}$ are given by
\begin{align}
D_{(b), \vb p}^{(b+1), \vb q} &= \frac{1}{1-h}D_{(-1/2, h, -1/2)}^{\tau, (1/2, b+1, 1/2)}, \label{eq:diff_mat_arcp} \\
D_{(b), \vb q}^{(b+1), \vb p} &= R_{(1/2, b, 1/2)}^{\tau, (-1/2, b+1, -1/2)}\left[\left(h-1\right)J^{\tau, (1/2, b, 1/2)} + I\right] + \left(1-h\right)R_{(1/2,b+1,1/2)}^{\tau,(-1/2,b+1,-1/2)}L_{\mathrm a\mathrm c, (3/2,b+1,3/2)}^{\tau,(1/2,b+1,1/2)}D_{(1/2,b,1/2)}^{\tau,(3/2,b+1,3/2)},\label{eq:diff_mat_arcq}
\end{align}
where $\sigma\vb P^{\tau, (1/2, b, 1/2)}(\sigma) = \vb P^{\tau, (1/2, b, 1/2)}(\sigma)J^{\tau, (1/2, b, 1/2)}$. Moreover, $D_{(b), \vb p}^{(b+1), \vb q}$ is a $(-1, 2)$ banded matrix and $D_{(b), \vb q}^{(b+1), \vb p}$ is a lower bidiagonal matrix, so that we can also write these relationships as
\begin{align}\label{eq:diff_mat_arcp_relation}
\odv*{p_n^{(b)}(\theta)}{\theta} &= d_{n-1, n+1, \vb p}^{(b+1), \vb q}q_{n-1}^{(b+1)}(\theta) + d_{n, n+1, \vb p}^{(b+1), \vb q}q_n^{(b+1)}(\theta), \\ \label{eq:diff_mat_arcq_relation}
\odv*{q_{n}^{(b)}(\theta)}{\theta} &= d_{n,n,\vb q}^{(b+1), \vb p}p_{n-1}^{(b+1)}(\theta) + d_{n+1,n,\vb q}^{(b+1), \vb p}p_n^{(b+1)}(\theta),
\end{align}
with the coefficients coming from the differentiation matrices. The matrix $D_{(b)}^{(b+1)}$ is defined from interlacing $D_{(b), \vb p}^{(b+1), \vb q}$ and $D_{(b), \vb q}^{(b+1), \vb p}$, meaning placing the associated entries from $D_{(b), \vb p}^{(b+1), \vb q}$ in each odd column and those from $D_{(b), \vb q}^{(b+1), \vb p}$ in each even column, and is a $(1, 3)$ banded matrix.
\end{proposition}

\begin{proof}
Differentiating $\vb p^{(b)}$, we obtain
\begin{equation}
\odv*{\vb p^{(b)}}{\theta} = x\pdv*{\vb P^{\tau, (-1/2, b, -1/2)}(\sigma)}{y} - y\pdv*{\vb P^{\tau, (-1/2, b, -1/2)}(\sigma)}{x} = \vb q^{(b+1)}\frac{1}{1-h}D_{(-1/2,b,-1/2)}^{\tau,(1/2,b+1,1/2)},
\end{equation}
and we let $D_{(b),\vb p}^{(b+1),\vb q} := (1-h)^{-1}D_{(-1/2,b,-1/2)}^{\tau,(1/2,b+1,1/2)}$. Since $D_{(-1/2,b,-1/2)}^{\tau,(1/2,b+1,1/2)}$ is a $(-1, 2)$ banded matrix, $D_{(b),\vb p}^{(b+1),\vb q}$ is as well. The $\vb q^{(b)}$ case is more involved. Using the definitions of $J^{\tau, (1/2, b, 1/2)}$ and $R_{(1/2, b, 1/2)}^{\tau, (-1/2, b+1,-1/2)}$, we obtain
\begin{equation}\label{eq:diff_mat_stepA}
x\pdv*{\vb q^{(b)}}{y} = x\pdv*{\left[y\vb P^{\tau, (1/2, b, 1/2)}(\sigma)\right]}{y} = x\vb P^{\tau, (1/2, b, 1/2)}(\sigma) = \vb p^{(b+1)}R_{(1/2, b, 1/2)}^{\tau, (-1/2,b+1,-1/2)}\left[(h-1)J^{\tau,(1/2,b,1/2)}+I\right].
\end{equation}
The $y\partial_x$ operator is obtained by using $y^2 = \sigma(1-h)^2(\tau-\sigma)$ to allow for use of the weight connected matrix, giving
\begin{equation}\label{eq:diff_mat_stepB}
y\pdv*{\vb q^{(b)}}{x} = y\pdv*{\left[y\vb P^{\tau, (1/2, b, 1/2)}(\sigma)\right]}{x} = (h-1)\vb p^{(b+1)}R_{(1/2,b+1,1/2)}^{\tau,(-1/2,b+1,-1/2)}L_{\mathrm a\mathrm c, (3/2,b+1,3/2)}^{\tau,(1/2,b+1,1/2)}D_{(1/2,b,1/2)}^{\tau,(3/2,b+1,3/2)}.
\end{equation}
Subtracting \eqref{eq:diff_mat_stepB} from \eqref{eq:diff_mat_stepA} gives the form of $D_{(b), \vb q}^{(b+1), \vb p}$ stated. Interlacing these entries give the $(1, 3)$ banded matrix $D_{(b)}^{(b+1)}$. We prove that $D_{(b), \vb q}^{(b+1), \vb p}$ is lower bidiagonal in \ref{app:dbqbpqb_lowerbidiag}. 
\end{proof}

\subsection{Mass matrix}

We now consider computing the mass matrix $M^{(b)} := [\vb P^{(b)}]\tran\vb P^{(b)}$. We can assume here that $b=0$ since, for $b \neq 0$, we can write $\vb P^{(b)} = \vb P^{(0)}R_{(0)}^{(b)}$ so that
\begin{align}\label{eq:arc_mass_matrix_bnot0}
M^{(b)} &= \left[R_{(0)}^{(b)}\right]\tran M^{(0)}R_{(0)}^{(b)},
\end{align}
thus all we need is the ability to compute $M^{(0)}$. The entries of $M^{(0)}$ are defined by 
\begin{align*}
M_{ij}^{(0)} = \int_{-\varphi}^{\varphi} P_i^{(0)}(\theta)P_j^{(0)}(\theta) \dthe, 
\end{align*}
where $P_k^{(0)}$ denotes the $k$th polynomial in $\vb P^{(0)}$. This is exactly the inner product defining orthogonality for the $\vb P^{(0)}$ polynomials, hence $M^{(0)}$ is diagonal. The entries come from interlacing the mass matrices for the $p_n^{(0)}$ and $q_n^{(0)}$ polynomials, in particular their squared norms.  Using the formula for the norms of the semiclassical Jacobi polynomials from \eqref{eqapp:weight_integral} in \ref{app:linear_coefficients}, we can show that 
\begin{align}\label{eq:arc_polynomial_mass_matrix_formulae_pq}
\|p_n^{(0)}\|^2  = 4\arccsc(\!\!\sqrt \tau), \quad \|q_n^{(0)}\|^2 = \frac{(1-h)^2}{2}\left[\tau^2\arccsc(\!\!\sqrt \tau) + (2 - \tau)\sqrt{\tau-1}\right].
\end{align}
Interlacing these expressions gives us $M^{(0)}$. We have therefore proven the following.

\begin{proposition}
Define the mass matrix $M^{(b, h)} := [\vb P^{(b, h)}]\tran\vb P^{(b, h)}$. The matrix $M^{(0, h)}$ is given by 
\begin{equation}
M^{(0, h)} = \diag(m_p, m_q, m_p, m_q, \ldots), 
\end{equation}
where $m_p = 4\arccsc(\!\!\sqrt\tau)$ and $m_q = (1-h)^2[\tau^2\arccsc(\!\!\sqrt\tau)+(2-\tau)\sqrt{\tau-1}]/2$. The general case $M^{(b, h)}$ is given by 
\begin{equation}
M^{(b, h)} = [R_{(0)}^{(b, h)}]\tran M^{(0, h)}R_{(0)}^{(b, h)}, \quad b \geq -1.
\end{equation}
When convenient, we will omit the dependence on $h$, writing $M^{(b)}$.
\end{proposition}

\subsection{Transforms}

To finish this section, we are interested in computing transforms that approximate a given function $f(x, y)$ by a polynomial in the arc polynomial basis. In particular, define the polynomial $f_n(x, y)$ by 
\begin{equation}\label{eq:interpolating_polynomial}
f_n(x, y) := \hat f_0 p_0(x, y) + \sum_{j=1}^{n-1} \left[ \hat f_{j1}q_j(x, y) + \hat f_{j2}p_j(x, y)\right].
\end{equation}
The aim is to find coefficients $\hat f_0$, $\hat f_{j1}$, and $\hat f_{j2}$ such that $f_n(x, y)$ interpolates $f(x, y)$ at some set of nodes. The following lemma tells us that these nodes should be the Gauss--Radau quadrature nodes.

\begin{lemma}\label{lem:gaussradaurule}
Assume $b > -1$. Let $\{(\xi_j, w_j)\}_{j=1}^n$ denote the Gauss--Radau quadrature nodes and weights with respect to the weight $w^{\tau, (-1/2, b, -1/2)}$, with $\xi_1 = 1$. The polynomials $\{p_j\}_{j=0}^{n-1}$ and $\{q_j\}_{j=1}^{n-1}$ are orthogonal with respect to the discrete inner product
\begin{equation}\label{eq:discrete_orthogonality}
\inp{f}{g}_n := \sum_{j=1}^n w_j\left[f(x_j, y_j) g(x_j, y_j) + f(x_j, -y_j)g(x_j, -y_j)\right],
\end{equation}
where $x_j = (h-1)\xi_j + 1$ and $y_j = \sqrt{1 - x_j^2}$. 
\end{lemma}

\begin{proof}
This is a simple extension of Theorems 1.22 and 1.23 in \cite{gautschi2004orthogonal} adapted to Gauss--Radau quadrature and is similar to the quadrature rules developed in \cite{olver2021orthogonal}. The Gauss--Radau quadrature nodes and weights can be computed using Theorem 3.2 in \cite{gautschi2004orthogonal}. Note that the summand for $j = 1$ is $2w_1f(1, 0)g(1, 0)$.
\end{proof}

Lemma \ref{lem:gaussradaurule} can be used to compute the coefficients in \eqref{eq:interpolating_polynomial} using inner products. Similar to Proposition 8.9 in \cite{olver2021orthogonal}, we can derive Proposition \ref{prop:transformcoeffarc}. Proposition \ref{prop:transformcoeffarc} is key to computing expansions in our bases throughout this work, and we make use of it adaptively, increasing $n$ until the interpolating polynomial approximates $f(x, y)$ up to machine precision \cite{ClassicalOrthogonalPolynomials.jl2024,aurentz2017chopping}.

\begin{proposition}\label{prop:transformcoeffarc}
Assume $b > -1$. Given a function $f(x, y)$, the polynomial $f_n(x, y)$ from \eqref{eq:interpolating_polynomial} interpolates $f(x, y)$ at the Gauss--Radau quadrature nodes defined in Lemma \ref{lem:gaussradaurule}, and the coefficients are given by $\hat{\vb f} = P\vb f_n$, where $P =  N^{-1}E\tran W$ and
\begin{align*}
\hat{\vb f}_n\tran &= \begin{bmatrix} \hat f_0 & \hat f_{11} & \hat f_{12} & \cdots & \hat f_{n-1, 1} & \hat f_{n-1, 2} \end{bmatrix}, \\
\hat{\vb f}_n\tran &= \begin{bmatrix} f(1, 0) & f(x_2, y_2) & f(x_2, -y_2) & \cdots & f(x_{n-1}, y_{n-1}) & f(x_{n-1}, -y_{n-1}) \end{bmatrix}, \\
N &= \diag\left(\|p_0\|_n^2, \|q_1\|_n^2, \|p_1\|_n^2, \ldots, \|q_{n-1}\|_n^2, \|p_{n-1}\|_n^2\right), \\
E\tran &=  \begin{bmatrix}
p_0(1, 0) & p_0(x_2, y_2) & p_0(x_2, -y_2) & \cdots & p_0(x_{n-1}, y_{n-1}) & p_0(x_{n-1}, -y_{n-1}) \\
q_1(1, 0) & q_1(x_2, y_2) & q_1(x_2, -y_2) & \cdots & q_1(x_{n-1}, y_{n-1}) & q_1(x_{n-1}, -y_{n-1}) \\
p_1(1, 0) & p_1(x_2, y_2) & p_1(x_2, -y_2) & \cdots & p_1(x_{n-1}, y_{n-1}) & p_1(x_{n-1}, -y_{n-1}) \\
\vdots & \vdots & \vdots & \ddots & \vdots & \vdots \\
q_{n-1}(1, 0) & q_{n-1}(x_2, y_2) & q_{n-1}(x_2, -y_2) & \cdots & q_{n-1}(x_{n-1}, y_{n-1}) & q_{n-1}(x_{n-1}, -y_{n-1}) \\
p_{n-1}(1, 0) & p_{n-1}(x_2, y_2) & p_{n-1}(x_2, -y_2) & \cdots & p_{n-1}(x_{n-1}, y_{n-1}) & p_{n-1}(x_{n-1}, -y_{n-1}) \end{bmatrix}, \\
W &= \diag\left(2w_1, w_2, w_2, \ldots, w_n, w_n\right),
\end{align*}
where $\|g\|_n^2 := \inp{g}{g}_n$. Note that the dependence on $(b, h)$ has been omitted in these expressions.
\end{proposition}

We will also be interested in the case $b = -1$, which is problematic as Lemma \ref{lem:gaussradaurule} no longer applies since $w^{\tau, (-1/2, -1, -1/2)}$ is not an integrable weight function.  To handle this case, we simply compute the expansion coefficients in the $\vb P^{(0)}$ basis and then transform back into the $\vb P^{(-1)}$ basis using the inverse of $R_{(b)}^{(b+1)}$.

\section{Spectral element method}\label{sec:sem2_}

We now use the arc polynomials to derive a spectral element method.  This will require that we first extend the arc polynomials in a piecewise basis over multiple intervals, and then we need to derive formulas for the mass matrix and the weak Laplacian. The bases we define are analogous to the bases defined in terms of the integrated Legendre functions introduced in \cite{babuvska1983lecture, knook2024quasi}. In what follows, we need the following definitions. 

\begin{definition}
We define the interval $I := [-\pi, \pi]$ and associate with $\bm\theta := (\theta_1, \ldots, \theta_{n+1})\tran$ the partition $-\pi = \theta_1 < \cdots < \theta_n < \theta_{n+1} = \pi$, where $n \geq 2$. The $i$th element is $E_i := [\theta_i, \theta_{i+1}]$ with length $\ell_i := \theta_{i+1} - \theta_i$, $i=1,\ldots,n$. We further define $E_0 = E_n$, $\ell_0 = \ell_n$, and $\theta_0 = \theta_{n+1}$.
\end{definition} 

\begin{definition}\label{eq:piecewise_arc_element_mapping}
The mapping $a_i \colon E_i \to E_i'$ is defined by 
\begin{equation}
a_i(\theta) := \theta - \frac{\theta_i + \theta_{i+1}}{2}, \quad i = 1, \ldots,n,
\end{equation}
where $E_i' := [-\varphi_i, \varphi_i]$ and $\varphi_i := \ell_i/2$. We define $h_i := \cos\varphi_i$ and $\tau_i := 2/(1-h_i)$.
\end{definition}

It is a crucial property that the mapping \eqref{eq:piecewise_arc_element_mapping} is defined only as a translation, instead of using, for example, $b_i(\theta) = (\pi/2\ell_i)(2\theta - \theta_i - \theta_{i+1})$ which maps $E_i$ into $[-\pi/2, \pi/2]$, since scaling implies that trigonometric polynomials on a single element will no longer be trigonometric polynomials with the same period on the entire interval $I$. This property is what will allow our bases to preserve periodicity, as we discuss later. As an example of this issue with scaling, consider $f(\theta) = \cos\theta$ on $E_i$. We find $f(a_i(\theta)) = \sin[(\theta_i+\theta_{i+1})/2]\sin(\theta) + \cos[(\theta_i+\theta_{i+1})/2]\cos\theta$ which is still $2\pi$-periodic, while $f(b_i(\theta)) = \sin[\pi(\theta_{i+1}-\theta)/\ell_i]$ is $2\ell_i$-periodic.

The first polynomials we define are related to the arc polynomials where $b = 0$. These are simply translated versions of $\vb P^{(0)}$ that are repeated over each interval. Formally, we have the following definition.

\begin{definition}[Piecewise arc polynomials with $b=0$]
The polynomial $P_{mj;i}^{(b), \bm\theta} \colon I \to \mathbb R$ is given by 
\begin{equation}\label{eq:piecewise_arc_base_polynomial_b0}
P_{mj; i}^{(b), \bm\theta}(\theta) := \begin{cases} P_{mj;i}^{(b)}\left(a_i(\theta)\right) & \theta \in E_i, \\ 0 & \text{otherwise}, \end{cases}
\end{equation}
\rev{where $m=0,1,\ldots$ denotes the polynomial degree, $i=1,2,\ldots,n$ indexes the element, and $j=1,2$ distinguishes the two families of arc polynomials.} 
The polynomials $P_{mj;i}^{(b)}$ are defined by 
\begin{align}
P_{m1;i}^{(b)}(\theta) &:= q_m^{(b, h_i)}(\theta), \quad m=1,2,3,\ldots,\quad i=1,\ldots,n, \\
P_{m2;i}^{(b)}(\theta) &:= p_m^{(b, h_i)}(\theta), \quad m=0,1,2,\ldots,\quad i=1,\ldots,n.
\end{align}
With these polynomials, we define the basis $\vb P^{(0), \bm\theta}$ by 
\begin{equation}\label{eq:piecewise_arc_beq0_defn}
\vb P^{(0), \bm\theta} := \begin{bmatrix} \vb P_{02}^{(0), \bm\theta} & \vb P_{11}^{(0), \bm\theta} & \vb P_{12}^{(0), \bm\theta} & \vb P_{21}^{(0), \bm\theta} & \vb P_{22}^{(0), \bm\theta} & \cdots \end{bmatrix},
\end{equation}
where
\begin{equation}\label{eq:piecewise_arc_beq0_blocks}
\vb P_{mj}^{(b), \bm\theta} := \begin{bmatrix} P_{mj; 1}^{(b), \bm\theta} & P_{mj; 2}^{(b), \bm\theta} & \cdots & P_{mj; n}^{(b), \bm\theta} \end{bmatrix}.
\end{equation}
\end{definition}

The $\vb P^{(0), \bm\theta}$ basis is not sufficient on its own for our spectral element method as to obtain sparsity we need a basis comprised of hat and bubble functions, also known as internal and external shape functions, respectively \cite{knook2024quasi, schwab1998theory}. In order to define an analogous basis $\vb P^{(-1), \bm\theta}$, let us first determine what the hat functions should look like in such a basis. Since we will be working with trigonometric polynomials, it would not be appropriate to consider linear hat functions in $\theta$. In particular, we should assume that the hat functions are represented as trigonometric polynomials, and that we have $n$ of them instead of $n+1$ since $\theta_1$ and $\theta_{n+1}$ represent the same point on a periodic interval.

\begin{lemma}\label{lem:hat_functions}
Let $\phi_i$ be a hat function satisfying $\phi_i(\theta_j) = \delta_{ij}$ where $\delta_{ij} = 1$ if $i=j$ and $\delta_{ij} = 0$ otherwise, and $i \in \{1, \ldots, n\}$. Then $\phi_i$ can be written as
\begin{equation}\label{eq:hat_function_psi0}
\phi_i(\theta) = \begin{cases} \psi_i^{(1)}(\theta) &\theta \in E_{i-1}, \\ \psi_i^{(2)}(\theta) & \theta \in E_i, \\ 0 & \text{otherwise}, \end{cases} \quad i = 1, \ldots, n,
\end{equation}
where $\psi_i^{(1)}(\theta)$ and $\psi_i^{(2)}(\theta)$ are linear trigonometric polynomials of the form
\begin{align}\label{eq:hat_function_psi1}
\psi_i^{(1)}(\theta) &= \frac12\left\{1 - \csc(\ell_{i-1})\left[\left(\sin \theta_{i-1} + \sin\theta_i\right)\cos\theta - \left(\cos\theta_{i-1}+\cos\theta_i\right)\sin\theta\right]\right\}, \\
\psi_i^{(2)}(\theta) &= \frac12\left\{1 + \csc(\ell_i)\left[\left(\sin\theta_i + \sin\theta_{i+1}\right)\cos\theta - \left(\cos\theta_i + \cos\theta_{i+1}\right)\sin\theta\right]\right\}. \label{eq:hat_function_psi2}
\end{align}
\end{lemma}

\begin{proof}
This can be easily proven by simply substituting $\theta = \theta_j$ into \eqref{eq:hat_function_psi1}--\eqref{eq:hat_function_psi2} and considering the cases $i=j$ and $i \neq j$ separately.
\end{proof}

\begin{figure}[h!]
\centering
\includegraphics[width=0.8\textwidth]{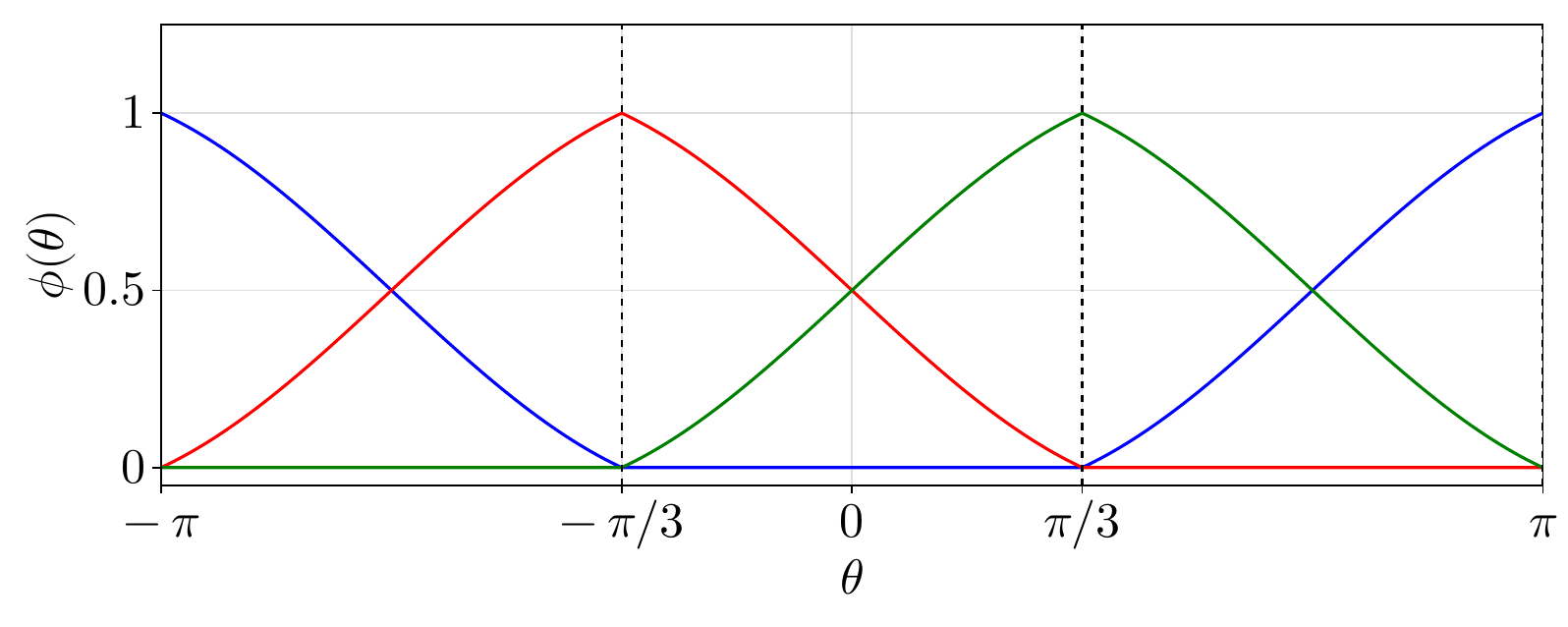}
\caption{The hat functions $\phi_1(\theta)$ (blue), $\phi_2(\theta)$ (red), and $\phi_3(\theta)$ (green) for the equally spaced grid $\bm\theta = (-\pi, -\pi/3, \pi/3, \pi)\tran$. The vertical lines show the grids. These hat functions are piecewise affine functions in $x$ and $y$.}\label{fig:hat_functions}	
\end{figure}

This lemma gives us the definition that we need for our hat functions $\phi_i$, and thus for our $b = -1$ basis. An example of what these hat functions look like in the case $n = 3$ is shown in Figure \ref{fig:hat_functions}. This basis will be what we use to define our trial and test functions for our spectral element method.

\begin{definition}[Piecewise arc polynomials with $b = -1$]\label{def:piecewisearcbneg1hat}
The $\vb P^{(-1), \bm\theta}$ basis is defined by 
\begin{equation}
\vb P^{(-1), \bm\theta} := \begin{bmatrix} \vb H^{\bm\theta} & \vb B^{\bm\theta} \end{bmatrix},
\end{equation}
where $\vb H^{\bm\theta}$ are the hat functions $\vb H^{\bm\theta} = (\phi_1, \ldots, \phi_n)$ from Lemma \ref{lem:hat_functions}, and $B$ are the bubble functions 
\begin{align}
\vb B^{\bm\theta} := \begin{bmatrix}
\vb P_{12}^{(-1), \bm\theta} & \vb P_{21}^{(-1), \bm\theta} & \vb P_{22}^{(-1), \bm\theta} & \vb P_{31}^{(-1), \bm\theta} & \cdots \end{bmatrix}.
\end{align} 
We will typically suppress the dependence of $\vb H^{\bm\theta}$ and $\vb B^{\bm\theta}$ on $\bm\theta$. 
\end{definition}

\rev{We note that a key difference between the $b=0$ and $b=-1$ bases is that the $b=0$ basis can represent discontinuous functions while the $b=-1$ basis cannot, with both bases giving the same convergence rate as we eventually show in Theorem \ref{thm:rate_of_conv_main_rho:all__}. We see the difference between these two bases later in Section \ref{sec:screened_poisson}, where we expand the solution to a differential equation in the $b=-1$ basis while expanding a discontinuous right-hand side in the $b=0$ basis.}

\subsection{Mass matrix}\label{sec:piecewise_arc:mass_matrix}

The mass matrix for these piecewise bases can be easily computed in terms of the arc polynomials, where we define the mass matrix $M^{(b), \bm\theta}$ by $M^{(b), \bm\theta} := [\vb P^{(b), \bm\theta}]\tran\vb P^{(b), \bm\theta}$. Let us start by considering $M^{(0), \bm\theta}$. To compute the entries, note that since the $b = 0$ arc polynomials are orthogonal with respect to the unit weight, $M^{(0), \bm\theta}$ must be a diagonal matrix. These entries can be found by noting that, for a given function $f(\theta) = g(a_i(\theta))$ supported over an element $E_i$, 
\begin{equation}
\int_I f(\theta)^2 \dthe = \int_{E_i} g(a_i(\theta))^2 \dthe = \int_{-\varphi_i}^{\varphi_i} g(\theta)^2 \dthe.
\end{equation}
Thus, $M^{(0), \bm\theta}$ can be obtained by simply taking the corresponding entries of the mass matrix for the associated arc polynomials. The $b = -1$ case can be obtained using the connection matrix defined in the next section, similarly to \eqref{eq:arc_mass_matrix_bnot0}, so that $M^{(-1), \bm\theta} = [R_{(-1)}^{(0), \bm\theta}]\tran M^{(0), \bm\theta}R_{(-1)}^{(0), \bm\theta}$. These findings are summarised below.

\begin{proposition}
The mass matrix $M^{(0), \bm\theta}$ is given by 
\begin{equation}\label{eq:piecewise_arc_mass_matrix_formula_repr}
M^{(0), \bm\theta} = \diag\left(\|\vb p_0^{(0), \bm\theta}\|^2, \|\vb q_1^{(0), \bm\theta}\|^2, \|\vb p_1^{(0), \bm\theta}\|^2, \ldots\right),
\end{equation}
where
\begin{equation}\label{eq:piecewise_arc_mass_matrix_formula_repr_pm_def}
\|\vb p_m^{(0), \bm\theta}\|^2 := \diag\left(\|p_m^{(0, h_1)}\|_{(0, h_1)}^2, \ldots, \|p_m^{(0, h_n)}\|_{(0, h_n)}^2\right) \in \mathbb R^{n\times n}
\end{equation}
and similarly for $\|\vb q_m^{(0), \bm\theta}\|^2$. The mass matrix $M^{(-1), \bm\theta}$ is given by $M^{(-1), \bm\theta} = [R_{(-1)}^{(0), \bm\theta}]\tran M^{(0), \bm\theta}R_{(-1)}^{(0), \bm\theta}$, where $\vb P^{(-1), \bm\theta} = \vb P^{(0), \bm\theta}R_{(-1)}^{(0), \bm\theta}$.
\end{proposition}

\subsection{Connection matrix}

We consider in this section the connection matrix $R_{(-1)}^{(0), \bm\theta}$, defined by $\vb P^{(-1), \bm\theta} = \vb P^{(0), \bm\theta}R_{(-1)}^{(0), \bm\theta}$. The form of $R_{(-1)}^{(0), \bm\theta}$ is given in the following proposition.

\begin{proposition}\label{prop:pwa_connmat}
Define $R_{(-1)}^{(0), \bm\theta}$ by $\vb P^{(-1), \bm\theta} = \vb P^{(0), \bm\theta}R_{(-1)}^{(0), \bm\theta}$. This $R_{(-1)}^{(0), \bm\theta}$ is a block-banded matrix with block bandwidths $(1, 1)$ of the form
\begin{equation}\label{eq:connection_matrix_formula_piecewise_arc}
R_{(-1)}^{(0), \bm\theta} =\begin{bmatrix}
\|\vb p_0^{(0), \bm\theta}\|^{-2}\vb P_{02}\tran\vb H & A_{01}^{(-1),\bm\theta} &                           &                           &                           &        \\
\|\vb q_1^{(0), \bm\theta}\|^{-2}\vb P_{11}\tran\vb H &      \vb 0                     & B_{01}^{(-1), \bm\theta} &                           &                           &        \\
                     & A_{11}^{(-1),\bm\theta} &       \vb 0                    & A_{12}^{(-1),\bm\theta} &                           &        \\
                     &                           & B_{11}^{(-1),\bm\theta} &                          \vb 0 & B_{12}^{(-1),\bm\theta} &        \\
                     &                           &                           & A_{22}^{(-1),\bm\theta} &                   \vb 0        & \ddots \\
                     &                           &                           &                           & B_{22}^{(-1),\bm\theta} & \ddots \\
                     &                           &                           &                           &                           & \ddots
\end{bmatrix},
\end{equation}
where 
\begin{align}
B_{m -1, m}^{(-1), \bm\theta} := \diag\left(b_{m-1,m}^{(-1,h_1)}, \ldots, b_{m-1,m}^{(-1,h_n)}\right) \in \mathbb R^{n \times n},
\end{align}
with similar definitions for $B_{m-1,m-1}^{(-1), \bm\theta}$, $A_{m,m+1}^{(-1), \bm\theta}$, and $A_{mm}^{(-1), \bm\theta}$. The coefficients in these $A_{ij}^{(-1), \bm\theta}$ and $B_{ij}^{(-1), \bm\theta}$ matrices are the coefficients in \eqref{eq:arc_connection_p1}--\eqref{eq:arc_connection_p2}. The forms of $\vb P_{02}\tran\vb H$ and $\vb P_{11}\tran\vb H$ are
\begin{align}\label{eq:p02h_p11h_mat_v}
\vb P_{02}\tran\vb H = \frac12
\begin{bmatrix}
\ell_1 & \ell_1 &        &        &             &    \\
 & \ell_2 & \ell_2 &        &             &           \\
       &  & \ell_3 & \ddots &             &           \\
       &        & & \ddots & \ell_{n-2} &           \\
       &        &        & \ddots & \ell_{n-1} & \ell_{n-1} \\
\ell_n &        &        &        &  & \ell_n   
\end{bmatrix}, \quad \vb P_{11}\tran\vb H = \frac14\begin{bmatrix} 
\xi_1 & -\xi_1 \\
& \xi_2 & -\xi_2 \\
& & \xi_3 & \ddots \\
&&&\ddots & -\xi_{n-2} \\
&&&\ddots & \xi_{n-1} & -\xi_{n-1} \\
-\xi_n &&&&&\xi_n 
\end{bmatrix},
\end{align}
where $\xi_i := (\sin\ell_i - \ell_i)/\sin(\ell_i/2)$.
\end{proposition}

\begin{proof}
Computing this matrix $R_{(-1)}^{(0), \bm\theta}$ relies on the formula
\begin{equation}\label{eq:piecewise_arc_connection_matrix_inp_form}
R_{(-1)}^{(0), \bm\theta} = \left[M^{(0), \bm\theta}\right]^{-1}\left(\left[\vb P^{(0), \bm\theta}\right]\tran\vb P^{(-1), \bm\theta}\right).
\end{equation}
Thus, to compute $R_{(-1)}^{(0), \bm\theta}$ first requires the computation of $[\vb P^{(0), \bm\theta}]\tran\vb P^{(-1), \bm\theta}$. The complete details for this computation are left to \ref{app:computing_connection_matrix_piecewise}.
\end{proof}

An important observation to make is that the form of \eqref{eq:connection_matrix_formula_piecewise_arc} has a similar structure to the banded-block-banded arrowhead ($B^3$-arrowhead) structure defined in \cite{knook2024quasi}. The only difference is that the blocks in the first row and first column are not banded since they have non-zero values in their corners, meaning they are \emph{cyclically banded} \cite{strang2011groups}. This observation leads us to the following definitions analogous to the definition of a $B^3$-arrowhead matrix from \cite{knook2024quasi}.

\begin{definition}[Adapted from \cite{strang2011groups}]
A matrix $A \in \mathbb R^{n \times n}$ is a \emph{cyclic banded matrix} with sub-bandwidths $(\ell, u)$ if $a_{ij} = 0$ when $\ell < i - j$ or $j - i > u$, with the exception of the cases $n - i + j \leq u$, or $n + i - j \leq \ell$. In particular, $A$ has the form
\begin{center}
\begin{tikzpicture}
\draw[very thick] (2,2)--(-2,2)--(-2,-2)--(2,-2)--cycle;
\fill[color=white,draw=black] (-1,2)--(-2,2)--(-2,1)--(1,-2)--(2,-2)--(2,-1)--cycle;
\draw[dashed](-2,2)--(2,-2);
\draw[<->] (-0.5,-0.5)--node[below]{$\ell$}(-0.02,-0.02);
\draw[<->] (0.02,0.02)--node[above]{$u$}(0.5,0.5);
\draw[color=white,draw=black] (1.0,2.0)--(2.0,1.0);
\draw[<->] (1.5,1.5)--node[left]{$\ell$}(1.98, 1.98);
\draw[color=white,draw=black] (-1.0,-2.0)--(-2.0,-1.0);
\draw[<->] (-1.5,-1.5)--node[right]{$u$}(-1.98,-1.98);
\end{tikzpicture}
\end{center}
\end{definition}

\begin{definition}\label{def:cb3ar}
A cyclic banded-block-banded-arrowhead ($CB^3$-arrowhead) matrix $A \in \mathbb R^{(m + pm) \times (m + pm)}$ with block-bandwidths $(\ell, u)$ and sub-block-bandwidth $\lambda + \mu$ has the following properties:
\begin{enumerate}
\item It is a block-banded matrix with block-bandwidths $(\ell, u)$.
\item The top-left block $A_0 \in \mathbb R^{m \times m}$ is a cyclic banded matrix with sub-bandwidths $(\lambda+\mu, \lambda+\mu)$.
\item The remaining blocks in the first row $B_k \in \mathbb R^{
m \times m}$ are cyclic banded matrices with sub-bandwidths $(\lambda, \mu)$.
\item The remaining blocks in the first column $C_k \in \mathbb R^{m \times m}$ are cyclic banded matrices with sub-bandwidths $(\mu, \lambda)$.
\item All other blocks $D_{kj} \in \mathbb R^{m \times m}$ are diagonal.
\end{enumerate}
Such a matrix takes the form
\begin{equation}\label{eq:cb3mat_}
A = \begin{bmatrix} A_0 & B \\ C & D \end{bmatrix} = \left[\begin{array}{c|c|c|c|c|c|c}
A_0 & B_1 & \cdots & B_u&&& \\ \hline
C_1 & D_{11} & \cdots & D_{1u} & D_{1,1+u}&& \\\hline
\vdots&\vdots&\ddots&\ddots&\ddots&\ddots\\\hline
C_\ell&D_{\ell 1}&\ddots&\ddots&\ddots&\ddots&D_{p-u, p} \\\hline
&D_{\ell+1,1} &\ddots&\ddots&\ddots&\ddots&D_{p-u+1,p} \\\hline
&&\ddots&\ddots&\ddots&\ddots&\vdots\\\hline
&&&D_{p,p-\ell}&D_{p,p-\ell+1}&\cdots&D_{p,p}
\end{array}\right].
\end{equation}
We may also use the notation $CB^3(\ell, u; \lambda, \mu)$ for such a matrix. To store the diagonal blocks in $D$ we represent it as an interlaced matrix of the form
\begin{equation}
D = \bigoplus_{i=1}^m D_i,
\end{equation}
where $D_i$ are banded matrices with bandwidths $(\ell, u)$ and, as in \cite{knook2024quasi} we use a direct sum to denote interlacing the entries of the matrices, in particular
\begin{equation}
\vb e_{\ell}\tran D_{kj} \vb e_\ell = \vb e_k\tran D_\ell \vb e_j.
\end{equation}
\end{definition}

The $CB^3$-arrowhead structure has many properties in common with the $B^3$-arrowhead structure of \cite{knook2024quasi}, for example it can be shown that multiplication of two $CB^3$-arrowhead matrices is another $CB^3$-arrowhead matrix. We leverage this structure to efficiently multiply matrices of these forms and solve linear systems throughout this work. The mass matrix $M^{(-1), \bm\theta}$ is also of this form.

\begin{lemma}\label{lem:product_cyclic_matrices}
Let $A_1, A_2 \in \mathbb R^{n \times n}$ be cyclic banded matrices both with sub-bandwidths $(1, 1)$. The product $A_1A_2$ is another cyclic banded matrix with sub-bandwidths $(2, 2)$.
\end{lemma}

\begin{proof}
This can be proven using the representation
\[
A_i = B_i + \alpha_iE_{1n} + \beta_iE_{n1}, \quad i=1,2,
\]
where $\vb e_j$ denotes the $j$th column of the $n\times n$ identity matrix, $E_{ij} := \vb e_i\vb e_j\tran$, $B_i$ is a $(\ell_i, u_i)$ banded matrix, and $\alpha_i$ and $\beta_i$ are the corner values. Working through the terms in the product in this representation leads to the result.
\end{proof}

\begin{proposition}\label{prop:cb3arrowhead_connection_mass}
The matrices $R_{(-1)}^{(0), \bm\theta}$ and $M^{(-1), \bm\theta}$ are $CB^3(1, 1; 1, 0)$ and $CB^3(2, 2; 1, 0)$ matrices, respectively.
\end{proposition}

\begin{proof}
That $R_{(-1)}^{(0), \bm\theta}$ is a $CB^3(1,1;1,0)$ matrix is obvious from the definition and the form in \eqref{eq:connection_matrix_formula_piecewise_arc}, so we focus on $M^{(-1), \bm\theta}$. It is clear that $M^{(-1), \bm\theta}$ is another $CB^3$-arrowhead matrix since we can show they are closed under multiplication and diagonal matrices are $CB^3(0,0;0,0)$ matrices, so we just need to determine the parameters defining the structure for $M^{(-1), \bm\theta}$. Since the diagonal matrix will not affect the structure, we can simply consider $M := [R_{(-1)}^{(0), \bm\theta}]\tran R_{(-1)}^{(0), \bm\theta}$ which will have the same structure as $M^{(-1), \bm\theta}$. Additionally, we can ignore the $\|\vb p_0^{(0), \bm\theta}\|^{-2}$ and $\|\vb q_1^{(0), \bm\theta}\|^{-2}$ factors in the first column for the same reason, leaving the matrix $\tilde{R}_{(-1)}^{(0), \bm\theta}$. We can compute, omitting the superscripts in the coefficient matrices $A_{ij}^{(-1),\bm\theta}$ and $B_{ij}^{(-1),\bm\theta}$ for convenience,
\begin{align*}
\left[\tilde{R}_{(-1)}^{(0), \bm\theta}\right]\tran \tilde{R}_{(-1)}^{(0), \bm\theta}  &= \begin{bmatrix}
\left(\vb H\tran\vb P_{02}\right)\left(\vb P_{02}\tran\vb H\right) + \left(\vb H\tran\vb P_{11}\right)\left(\vb P_{11}\tran\vb H\right) & \left(\vb H\tran\vb P_{02}\right)A_{01} & \left(\vb H\tran\vb P_{11}\right)B_{01} \\
A_{01}\left(\vb P_{02}\tran\vb H\right) & A_{01}^2 + A_{11}^2 & \vb 0 & A_{11}A_{12} \\
B_{01}\left(\vb P_{02}\tran\vb H\right) &\vb 0& B_{01}^2+B_{11}^2 &\vb 0 & B_{11}B_{12} \\
& A_{11}A_{12} &\vb 0& A_{12}^2 + A_{22}^2 &\vb 0& \ddots \\
&&B_{12}B_{12}&\vb 0&B_{12}^2 + B_{22}^2 & \ddots \\
&&&&&\ddots
\end{bmatrix}
\end{align*}
From Lemma \ref{lem:product_cyclic_matrices} and \eqref{eq:p02h_p11h_mat_v}, $(\vb H\tran\vb P_{02})(\vb P_{02}\tran\vb H)$ is a cyclic banded matrix with sub-bandwidths $(1, 1)$, and similarly for $(\vb H\tran\vb P_{1})(\vb P_{1}\tran\vb H)$. Thus, this first block is a cyclic banded matrix with sub-bandwidths $(1, 1)$. The structure of the other blocks is clear since they are only being multiplied by diagonal matrices, and so we have our result.
\end{proof}

\subsubsection{Factorising $CB^3$-arrowhead matrices}\label{sec:factor_cbbb}

Similar to \cite{knook2024quasi}, it will be useful to know how to compute the reverse Cholesky factorisation of a $CB^3$-arrowhead matrix. In particular, given a $CB^3$-arrowhead matrix $A$, we are interested in computing a lower triangular matrix $L$ such that $A = L\tran L$. The reverse Cholesky factorisation is preferable to the standard Cholesky factorisation because it leads to less fill-in and instead has a similar structure to a $CB^3$-arrowhead matrix, as the following results show.

\begin{lemma}\label{lem:revch2olcbm}
Let $A \in \mathbb R^{n \times n}$ be a symmetric positive definite matrix with sub-bandwidths $(\ell, \ell)$. Then $A = L\tran L$ where 
\begin{equation}
A = \begin{bmatrix} a_{11} & \vb a\tran \\ \vb a & A_0 \end{bmatrix} \quad \textnormal{and} \quad L = \begin{bmatrix} \sqrt{a_{11} - \vb v\tran\vb v} \\ \vb v & L_0 \end{bmatrix},
\end{equation}
where $\vb v = L_0^{-\mkern-1.5mu\mathsf T}\vb a$, $A_0 = L_0\tran L_0$, and $L_0$ is a $(\ell, 0)$ banded matrix. In particular, $L$ is an $(\ell, 0)$ banded matrix except for the first column. This matrix $L$ can be computed in $\mathcal O(n\ell)$ complexity.
\end{lemma}

\begin{proof}
The form of $L$ can be verified directly by multiplying out $L\tran L$. Note that $A_0$ is a $(\ell, \ell)$ banded matrix which gives the banded structure in $L_0$. To obtain the complexity in computing $L$, first see that $L_0$ can be computed from $A_0$ in $\mathcal O(n\ell)$ operations since $A_0$ is banded. Computing $\vb v = L_0^{-\mkern-1.5mu\mathsf T}\vb a$ then takes $\mathcal O(n\ell)$ operations also \cite{Golub2013}, and thus $L$ can be computed in $\mathcal O(n\ell)$ complexity.
\end{proof}

\begin{theorem}\label{thm:revchol_cbbb}
If $A \in \mathbb R^{(m + pm) \times (m + pm)}$ is a symmetric positive definite $CB^3(\ell, \ell; 1, 0)$ or $CB^3(\ell, \ell; 0, 1)$ matrix, then the reverse Cholesky factor $L$ of $A$ can be decomposed as
\begin{equation}
L = L_1 + L_2,
\end{equation}
with $L_1$ a $CB^3(\ell, 0; 1, 1)$ matrix and 
\begin{equation}
L_2 = \begin{bmatrix} \widetilde{L}_2 & \vb 0 \\ \vb 0 & \vb 0 \end{bmatrix},
\end{equation}
where $\widetilde{L}_2 \in \mathbb R^{m \times m}$ is a matrix whose only non-zero entries are in the first column.
\end{theorem}

\begin{proof}
The proof is similar to the proof of Theorem 4.2 of \cite{knook2024quasi}. We only consider the $CB^3(\ell, \ell; 1, 0)$ case since the second case then follows after transposing. Write
\begin{equation}\label{eq:thm_proof_decom_a_cb3a}
A = \begin{bmatrix}
A_0 & B \\ 
B\tran & D \end{bmatrix}
\end{equation}
where, as in Definition \ref{def:cb3ar}, $D = D_1 \oplus \cdots \oplus D_m$. The reverse Cholesky factor of $D$ is given by $\widetilde{L} = \widetilde{L}_1 \oplus \cdots \oplus \widetilde{L}_m$, where $D_i = \widetilde{L}_i\tran\widetilde{L}_i$ for $i=1,\ldots,m$ \cite{knook2024quasi}. Now, write
\begin{equation}
\begin{bmatrix}
A_0 & B \\ B\tran & D \end{bmatrix} = \begin{bmatrix} L_0\tran&L_B\tran \\ & \widetilde{L}\tran\end{bmatrix}\begin{bmatrix}L_0 \\ L_B & \widetilde{L} \end{bmatrix} = \begin{bmatrix} L_0\tran L_0 + L_B\tran L_B & L_B\tran\widetilde{L} \\ \widetilde{L}\tran L_B & \widetilde{L}\tran\widetilde{L} \end{bmatrix}.
\end{equation}
We see that $B = L_B\tran\widetilde{L}$. Letting $L_B\tran = (M_1, M_2, \ldots, M_\ell, \vb 0, \ldots, \vb 0)$ and writing
\begin{equation}
\widetilde{L} = \begin{bmatrix}
\widetilde{L}_{11} \\ \vdots & \ddots \\
\widetilde{L}_{p 1} & \cdots & \widetilde{L}_{pp} 
\end{bmatrix},
\end{equation}
multiplying out $L_B\tran\widetilde{L}$ leads to
\begin{equation}
M_\ell = B_\ell\widetilde{L}_{\ell\ell}^{-1} \quad \textnormal{and} \quad M_k = \left(B_k - \sum_{j=k+1}^\ell M_j\widetilde{L}_{jk}\right)\widetilde{L}_{kk}^{-1}, \quad k=1,\ldots,\ell-1.
\end{equation}
Since the $\widetilde{L}_{jk}$ are all diagonal, the $M_k$ blocks all have the same cyclic banded structure as $B_k$, $k=1,\ldots,\ell$. Next, we see that
\begin{equation}
L_0\tran L_0 = A_0 - L_B\tran L_B = A_0 - \sum_{j=1}^\ell M_jM_j\tran,
\end{equation}
thus $L_0$ is the reverse Cholesky factor of $A_0 - \sum_{j=1}^\ell M_jM_j\tran$. Since $M_j$ is a cyclic banded matrix with sub-bandwidths $(1, 0)$, $M_jM_j\tran$ is a cyclic banded matrix with sub-bandwidths $(1, 1)$, and thus has the same structure as $A_0$.  Therefore, from Lemma \ref{lem:revch2olcbm}, $L_0$ is a $(1, 0)$ banded matrix except for its first column. These results lead to the decomposition given in the theorem.
\end{proof}

We could give a more general result for matrices with larger sub-bandwidths, but the cases covered by Theorem \ref{thm:revchol_cbbb} are sufficient for the differential equations covered in this work. The only difference in the proof would be in the structure of $L_0$. The structure of the reverse Cholesky factor in Theorem \ref{thm:revchol_cbbb} leads to the following result for solving symmetric positive definite linear systems with these matrices in optimal complexity.

\begin{corollary}
Let $A \in \mathbb R^{(m+pm) \times (m+pm)}$ be a symmetric positive definite $CB^3(\ell,\ell;1,0)$ or $CB^3(\ell,\ell;0,1)$ matrix and let $A\vb x = \vb b$, where $\vb x, \vb b \in \mathbb R^{m+pm}$. The solution $\vb x$ can be computed in optimal complexity, namely in $\mathcal O(N)$ operations where $N = m + pm$.
\end{corollary}

\begin{proof}
The proof of Theorem \ref{thm:revchol_cbbb} gave the following algorithm for computing the reverse Cholesky factor $L$. Using the same decomposition \eqref{eq:thm_proof_decom_a_cb3a}, the steps are:
\begin{enumerate}
\item Compute $\widetilde{L}_i\tran$ such that $D_i = \widetilde{L}_i\tran\widetilde{L}_i$ for $i=1,\ldots,m$.
\item Compute $M_k = \left(B_k - \sum_{j=k+1}^\ell M_j\widetilde{L}_{jk}\right)\widetilde{L}_{kk}^{-1}$ for $k=\ell,\ell-1,\ldots,1$.
\item Compute $\widetilde{A}_0 = A_0 - \sum_{j=1}^\ell M_jM_j\tran$.
\item Compute the reverse Cholesky factor $L_0$ of $\widetilde{A}_0$.
\end{enumerate}
The reverse Cholesky factors $\widetilde{L}_i$ can be computed in $\mathcal O(p)$ operations, thus the first step takes $\mathcal O(mp)$ operations. For the second step, each $\widetilde{L}_{jk}$ block is diagonal, and so computing $(B_k - \sum_{j=k+1}^\ell M_j\widetilde{L}_{jk})\widetilde{L}_{kk}^{-1}$ for each $k=1,\ldots,\ell$ takes $\mathcal O(m)$ operations in total. The third step similarly takes $\mathcal O(m)$ operations, as does the final step due to Lemma \ref{lem:revch2olcbm}. Thus, the factor overall takes $\mathcal O(m) + \mathcal O(mp) = \mathcal O(N)$ operations to compute. The remaining complexity comes from applying the inverse of $L$ to vectors for computing $\vb x$. Using a similar argument to the proof of Corollary 4.3 in \cite{knook2024quasi} shows that this complexity is also $\mathcal O(N)$, noting that the additional non-zero column in $L_0$ only adds $\mathcal O(m)$ additional operations.
\end{proof}

We apply these results to the examples discussed in Section \ref{sec:solve_de_}, where any symmetric positive definite system with a $CB^3$-arrowhead matrix will be solved in optimal complexity using the reverse Cholesky factors.

\subsection{Differentiation matrix}

In this section we consider the differentiation matrix and the weak Laplacian. We start by defining these two matrices and then deriving their forms.

\begin{definition}
The differentiation matrix $D_{(-1)}^{(0), \bm\theta}$ is defined by 
\begin{equation}
\odv*{\vb P^{(-1),\bm\theta}}{\theta} = \vb P^{(0), \bm\theta}D_{(-1)}^{(0), \bm\theta}.
\end{equation}
The weak Laplacian $-\Delta^{(-1), \bm\theta}$ is defined by 
\begin{equation}
-\Delta^{(-1), \bm\theta} := \left(\odv*{\vb P^{(-1), \bm\theta}}{\theta}\right)\tran\odv*{\vb P^{(-1), \bm\theta}}{\theta}.
\end{equation}
\end{definition}

\begin{proposition}
The differentiation matrix $D_{(-1)}^{(0), \bm\theta}$ is a $CB^3(2,2;1,0)$ matrix of the form
\begin{equation}\label{eq:differentiation_matrix_formula_piecewise_arc}
D_{(-1)}^{(0), \bm\theta} = \begin{bmatrix} \|\vb p_0^{(0), \bm\theta}\|^{-2}\vb P_{02}\tran\vb H^\prime &                                                                      \vb 0 &  \vb 0                                                                        &                                                                         &                                                                      &                                                                         &        \\\vb 0
                            & D_{1,2,\vb p}^{(0), \bm\theta, \vb q} &                                                                        \vb 0 & D_{1,3,\vb p}^{(0), \bm\theta, \vb q} &                                                                      &                                                                         &        \\
\|\vb p_1^{(0), \bm\theta}\|^{-2}\vb P_{12}\tran\vb H^\prime &                                                                        \vb 0 & D_{2,2,\vb q}^{(0), \bm\theta, \vb p} &                                                                        \vb 0 &                \vb 0                                                      &                                                                         &        \\
                           &   \vb 0                                                                      &                 \vb 0                                                        & D_{2,3,\vb p}^{(0), \bm\theta, \vb q} &                                                                     \vb 0 & D_{2,4,\vb p}^{(0), \bm\theta, \vb q} &        \\
                            &                                                                         & D_{3,2,\vb q}^{(0), \bm\theta, \vb p} &                                                                        \vb 0 & D_{3,3,\vb q}^{(0),\bm\theta,\vb p} &                                                                        \vb 0 & \ddots \\
                            &                                                                         &                                                                         &                                                                        \vb 0 &                                                                     \vb 0 & D_{3,4,\vb p}^{(0), \bm\theta, \vb q} & \ddots \\
                            &                                                                         &                                                                         &                                                                         & D_{4,3,\vb q}^{(0),\bm\theta,\vb p} &                                                                        \vb 0 & \ddots \\
                            &                                                                         &                                                                         &                                                                         &                                                                      &                                                                         \vb 0 & \ddots \\
                            &                                                                         &                                                                         &                                                                         &                                                                      &                                                                         & \ddots \end{bmatrix}
\end{equation}
where the coefficients $D_{i, j, \vb p}^{(0), \bm\theta, \vb q}$ and $D_{i,j,\vb q}^{(0), \bm\theta, \vb p}$ are defined from \eqref{eq:diff_mat_arcp_relation}--\eqref{eq:diff_mat_arcq_relation} so that
\begin{equation}
D_{m,m+2, \vb p}^{(0), \bm\theta, \vb q} := \diag\left(d_{m,m+2, \vb p}^{(0, h_1), \vb q}, \ldots, d_{m,m+2, \vb p}^{(0, h_n), \vb q}\right) \in \mathbb R^{n \times n}
\end{equation}
and similarly for $D_{m,m+1,\vb p}^{(0), \bm\theta, \vb q}$, $D_{m+1,m+1,\vb q}^{(0), \bm\theta, \vb p}$, and $D_{m+1,m,\vb q}^{(0), \bm\theta, \vb p}$. The matrices $\vb P_{02}\tran\vb H^\prime$ and $\vb P_{12}\tran\vb H^\prime$ are given by 
\begin{equation}
\vb P_{02}\tran\vb H^\prime = \begin{bmatrix}
-1 & 1 \\
& -1 & 1 \\
& & -1 & \ddots \\
&&&\ddots& 1 \\
&&&\ddots&-1&1\\
1&&&&&-1
\end{bmatrix}, \quad \vb P_{12}\tran\vb H^\prime = \begin{bmatrix} 
\zeta_1 & -\zeta_1 \\
& \zeta_2 & -\zeta_2 \\
& & \zeta_3 & \ddots \\
&&&\ddots & -\zeta_{n-2} \\
&&&\ddots & \zeta_{n-1} & -\zeta_{n-1} \\
-\zeta_n &&&&&\zeta_n 
\end{bmatrix},
\end{equation}
where 
\[
\zeta_i := h_i\csc(\ell_i) \beta^{\tau_i, \left(-1/2,0,-1/2\right)}\left[\sin(\varphi_i)\left(2\alpha^{\tau_i, \left(-1/2,0,-1/2\right)} - 1\right) - \frac{\sin(\varphi_i) - \varphi_i}{1 - h_i}\right].
\]
Moreover, the weak Laplacian is a $CB^3(2,2; 1, 0)$ matrix given by $-\Delta^{(-1), \bm\theta} = [D_{(-1)}^{(0), \bm\theta}]\tran M^{(0), \bm\theta}D_{(-1)}^{(0), \bm\theta}$.
\end{proposition}

\begin{proof}
The final statement about the weak Laplacian follows from the properties of $D_{(-1)}^{(0), \bm\theta}$ and is proven using a similar argument to that in Proposition \ref{prop:cb3arrowhead_connection_mass}, and so we focus only on $D_{(-1)}^{(0), \bm\theta}$. Similar to the proof of Proposition \ref{prop:pwa_connmat}, the computation of $D_{(-1)}^{(0), \bm\theta}$ relies on the identity 
\begin{equation}\label{eq:piecewise_arc_differentiation_matrix_inp_form}
D_{(-1)}^{(0), \bm\theta} = \left[M^{(0),\bm\theta}\right]^{-1}\left[\vb P^{(0), \bm\theta}\right]\tran\left(\odv*{\vb P^{(-1), \bm\theta}}{\theta}\right).
\end{equation}
The remaining details for this computation are given in \ref{app:computing_differentiation_matrix_piecewise}.
\end{proof}

\subsection{Computing expansions}\label{sec:comp_exp_s}

In order to solve differential equations with our basis, we need to know how to expand functions in this basis, as we will need to be expanding \rev{right-hand sides and initial conditions for time-evolution problems}. In particular, consider expanding a function $f(x, y)$ in the $\vb P^{(0), \bm\theta}$ or $\vb P^{(-1), \bm\theta}$ bases, meaning $f = \vb P^{(0), \bm\theta}\vb f^{(0), \bm\theta}$ or $f = \vb P^{(-1), \bm\theta}\vb f^{(-1), \bm\theta}$. In the $\vb P^{(0), \bm\theta}$ case, we compute the expansion coefficients $\vb f^{(0, h_i)}$ for $f$ restricted to each element $E_i$, $i=1,\ldots,n$, meaning $\left.f\right|_{E_i} = \vb P^{(0, h_i)}\vb f^{(0, h_i)}$ for $i=1,\ldots,n$. Writing $\vb f^{(0, h_i)} = (f_0^{(0, h_i)}, f_1^{(0, h_i)}, f_2^{(0, h_i)}, \ldots)\tran$ and $\vb f_j^{(0), \bm\theta} = (f_j^{(0, h_1)}, \ldots, f_j^{(0, h_n)})\tran$ for $j=0,1,2,\ldots$, $\vb f^{(0), \bm\theta}$ is found by interlacing these coefficients:
\begin{equation}
\vb f^{(0), \bm\theta} = \begin{bmatrix} \left[\vb f_0^{(0), \bm\theta}\right]\tran & \left[\vb f_1^{(0), \bm\theta}\right]\tran & \left[\vb f_2^{(0), \bm\theta}\right]\tran & \cdots \end{bmatrix}\tran.
\end{equation}
The coefficients $\vb f^{(-1), \bm\theta}$ are computed using $\vb f^{(0), \bm\theta}$ and $R_{(-1)}^{(0), \bm\theta}$:
\begin{align}
f &= \vb P^{(0), \bm\theta}\vb f^{(0), \bm\theta} = \vb P^{(-1), \bm\theta}\left[R_{(-1)}^{(0), \bm\theta}\right]^{-1}\vb f^{(0), \bm\theta},
\end{align}
thus $\vb f^{(-1), \bm\theta}$ is the solution to $R_{(-1)}^{(0), \bm\theta}\vb f^{(-1), \bm\theta} = \vb f^{(0), \bm\theta}$. The $CB^3$-arrowhead structure of $R_{(-1)}^{(0), \bm\theta}$ allows this linear system to be solved efficiently.  These results are summarised below.

\begin{proposition}\label{prop:piecewise_arc_b0_trans}
Define $\vb f^{(0, h_i)} = (f_0^{(0, h_i)}, f_1^{(0, h_i)}, f_2^{(0, h_i)}, \ldots)\tran$ and $\vb f_j^{(0), \bm\theta} = (f_j^{(0, h_1)}, \ldots, f_j^{(0, h_n)})\tran$ for $j=0,1,2,\ldots$. The coefficients $\vb f^{(0), \bm\theta}$ such that $f = \vb P^{(0), \bm\theta}\vb f^{(0), \bm\theta}$ are given by 
\begin{equation}\label{eq:piecewise_arc_transform_formula_b0}
\vb f^{(0), \bm\theta} = \begin{bmatrix} \left[\vb f_0^{(0), \bm\theta}\right]\tran & \left[\vb f_1^{(0), \bm\theta}\right]\tran & \left[\vb f_2^{(0), \bm\theta}\right]\tran & \cdots \end{bmatrix}\tran.
\end{equation}
The coefficients $\vb f^{(-1), \bm\theta}$ such that $f = \vb P^{(-1), \bm\theta}\vb f^{(-1), \bm\theta}$ are given by the solution to $R_{(-1)}^{(0), \bm\theta}\vb f^{(-1), \bm\theta} = \vb f^{(0), \bm\theta}$.
\end{proposition}

\section{Analysis}\label{sec:analysis__}

In this section we consider some theoretical results for our bases. We start by rigorously demonstrating that we can expand any trigonometric polynomial $f(\theta) = a_0+\sum_{n=1}^N (a_n\cos n\theta + b_n\sin n\theta)$ with a finite expansion in $\vb P^{(b), \bm\theta}$ basis, $b \in \{-1, 0\}$, using $M\max\{2N,1\}$ coefficients where $M$ is the number of elements. We then finish with some convergence properties for our basis.

\subsection{Expanding trigonometric polynomials}\label{sec:fourier_expansion_arc}

A fundamental property of our $\vb P^{(-1)}$ basis is that it can be used to exactly represent trigonometric polynomials, meaning with a finite expansion. We prove this fact in this section, starting with a lemma demonstrating that we can represent $\cos n\theta $ and $\sin n\theta$ over an arc as finite expansions in the $\vb p^{(0)}$ and $\vb q^{(0)}$ bases, respectively.

\begin{lemma}\label{lem:expandfinthetaexparc}
Let $n \in \mathbb N$ and assume $\theta \in \Omega_h$, $|h| < 1$. We can expand $\cos n\theta = \sum_{j=0}^n \hat\mu_{jn}p_j^{(0, h)}(\theta)$ and $\sin n\theta = \sum_{j=1}^n \hat\eta_{jn}q_j^{(0, h)}(\theta)$ for some coefficients $\{\hat\mu_{jn}\}$ and $\{\hat\eta_{jn}\}$.
\end{lemma}

\begin{proof}
We give the proof of this lemma in \ref{app:proof_finite_arc_expansion}, where we also derive a recurrence relationship for computing the coefficients $\hat\mu_{jn}$ and $\hat\eta_{jn}$. In particular, the $\cos n\theta$ expansion is considered in Lemma \ref{lem:cosnthetaexparc}, and the $\sin(n\theta)$ expansion is considered in Lemma \ref{lem:sinnthetaexparc}.
\end{proof}

\begin{corollary}\label{cor:trigfinbas}
Given a trigonometric polynomial $f(\theta) = a_0 + \sum_{n=1}^N (a_n\cos n\theta + b_n \sin n \theta)$ for $\theta \in \Omega_h$, $|h| < 1$, we can write $f(\theta) = \vb P^{(b)}(\theta)\vb f$ for $b \geq -1$, where $\vb f = (\vb f_{2N+1}\tran, \vb 0\tran)\tran$ and $\vb f_{2N+1} \in \mathbb R^{2N+1}$. In particular, $f(\theta)$ can be represented as an expansion in the $\vb P^{(b)}$ basis with $2N+1$ terms.
\end{corollary}

\begin{proof}
For $b = 0$ this follows from simply expanding each of the terms $\cos(n\theta)$ and $\sin(n\theta)$ in $f(\theta)$ using Lemma \ref{lem:expandfinthetaexparc}. When $b \neq 0$, the result follows from first expanding in the $\vb P^{(0)}$ basis and then transforming using the connection matrix $R_{(0)}^{(b)}$, noting that this matrix is upper triangular.
\end{proof}

Lemma \ref{lem:expandfinthetaexparc} and Corollary \ref{cor:trigfinbas} lead us to the following theorem that extends these results to the entire circle. Note also in the following theorem that the $\vb P^{(-1),\bm\theta}$ basis requires $M$ fewer coefficients than the $\vb P^{(0),\bm\theta}$ basis when the trigonometric polynomial is not constant.

\begin{theorem}\label{thm:trigpolyexpans22}
Given a trigonometric polynomial $f(\theta) = a_0 + \sum_{n=1}^N (a_n \cos n \theta + b_n \sin n\theta)$ for $\theta \in \mathbb R$, we can write $f(\theta) = \vb P^{(b), \bm\theta}\vb f^{(b), \bm\theta}$ for $b \in \{-1, 0\}$, where $\vb f^{(b), \bm\theta} = ([\vb f_0^{(b), \bm\theta}]\tran, \vb 0\tran)\tran$ and $\vb f_0^{(0), \bm\theta} \in \mathbb R^{M(2N+1)}$ and $\vb f_0^{(-1), \bm\theta} \in \mathbb R^{M\max\{2N,1\}}$, using the grid $\bm\theta=(\theta_1,\ldots,\theta_{M+1})\tran$ with $M \geq 2$. 
\end{theorem}

\begin{proof}
The case for $N = 0$ is obvious since $\vb f_0^{(b), \bm\theta} = (a_0,\ldots,a_0) \in \mathbb R^M$ for both $b \in \{-1, 0\}$. The $b = 0$ case  follows from combining Corollary \ref{cor:trigfinbas} and Proposition \ref{prop:piecewise_arc_b0_trans} since each element contributes $2N+1$ terms. The $b = -1$ case can be derived from $b=0$ using $R_{(-1)}^{(0), \bm\theta}\vb f^{(-1), \bm\theta} = \vb f^{(0), \bm\theta}$, leading to the system $\vb f_i^{(0), \bm\theta} = C_i\vb f_{i-1}^{(-1), \bm\theta} + D_i\vb f_{i+1}^{(-1), \bm\theta}$, where $C_i$ and $D_i$ are invertible diagonal matrices. Sparing the details, we can show that we obtain a valid solution when $\vb f_i^{(-1), \bm\theta} = \vb 0$ for $i \geq 2N$ which leads to the result.
\end{proof}

\subsection{Convergence}

Now we consider the convergence of expansions in our basis to a given function $f$, which is no longer assumed to be a finite trigonometric polynomial. To start, we give some definitions followed by a result analogous to Theorem 3.2 of \cite{olver2021orthogonal}, stating that if we can expand the even and odd parts of a function in a semiclassical Jacobi expansion then the expansion also converges in the arc polynomial basis. This then implies that, provided a function $f$ has a convergent series over each element in a given grid, $f$ has a convergent series in the $\vb P^{(b), \bm\theta}$ basis for $b \in \{-1, 0\}$.

\begin{definition}\label{def:odd_even_part_arc_+}
Given a function $f(x, y)$ on an arc $\Omega_h$, $|h|<1$, define the \emph{even} and \emph{odd} parts of $f$ to be, respectively,
\begin{equation}
f_e(x) := \frac{f\left(x, \sqrt{1-x^2}\right) + f\left(x, -\sqrt{1-x^2}\right)}{2}, \qquad f_o(x) := \frac{f\left(x, \sqrt{1-x^2}\right) - f\left(x, -\sqrt{1-x^2}\right)}{2\sqrt{1-x^2}}.
\end{equation} 
We may also slightly misuse this notation to define $f_e(\theta) = (f(\cos\theta,\sin\theta)+f(\cos\theta,-\sin\theta))/2$ and $f_o(\theta) = (f(\cos\theta,\sin\theta)-f(\cos\theta,-\sin\theta))/\colr{(2\sin\theta)}$.
\end{definition}

\begin{definition}
Given a function $f(x, y)$ on an arc $\Omega_h$, $|h|<1$, define the partial sum operator $S_n^{(b, h)}$ for $b \geq -1$ by
\begin{equation}
S_n^{(b, h)}[f](x, y) = \hat f_0 p_0(x, y) + \sum_{j=1}^{n} \left[\hat f_{j1} q_j(x, y) + \hat f_{j2}p_j(x, y)\right],
\end{equation}
omitting the dependence of the functions and coefficients on $b$ and $h$, where
$\hat f_0 = \inp{f}{p_0}/\|p_0\|^2$ and 
\[
\hat f_{j1} = \dfrac{\inp{f}{q_j}}{\|q_j\|^2}, \quad \hat f_{j2} = \dfrac{\inp{f}{p_j}}{\|p_j\|^2}, \quad j=1,\ldots,n.
\]
If needed, we include superscripts to highlight the $(b, h)$ dependence. Similarly, for a function $g(\sigma)$ on $0 \leq \sigma \leq 1$ we let $s_n^{t, (a, b, c)}$ denote the partial sum operator
\begin{equation}
s_n^{t, (a, b, c)}[g](\sigma) = \sum_{j=0}^{n-1} \hat g_j P_j^{t, (a, b, c)}(\sigma), \qquad \hat g_j = \dfrac{\inp*{g}{P_j^{t, (a, b, c)}}^{t, (a, b, c)}}{\left\|P_j^{t, (a, b, c)}\right\|_{t, (a, b, c)}^2},~~j=0,1,\ldots,n.
\end{equation}
\end{definition}

\begin{proposition}\label{prop:conv_series_arcpol2q}
Let $f(x, y)$ be a function on an arc $\Omega_h$, $|h| < 1$ and assume $b \geq -1$. If 
\begin{equation}
s_n^{\tau,(-1/2,b,-1/2)}[f_e(\zeta)] \to f_e(\zeta) \quad \textnormal{and} \quad s_n^{\tau,(1/2, b, 1/2)}[f_o(\zeta)] \to f_o(\zeta)
\end{equation}
pointwise as $n \to \infty$, where $\zeta := 1 + (h-1)\sigma, \sigma := (x-1)/(h-1)$, and $\tau := 2/(1-h)$, then $S_n^{(b, h)}[f] \to f$ as $n \to \infty$. The same result holds when considering convergence in the norm \eqref{eq:def:inp} for $f_e$ and $f_o$, and \eqref{eq:inp_arc} for $f$, rather than pointwise convergence.
\end{proposition} 

\begin{proof}
We can assume that $b > -1$ in what follows since, for $b = -1$, the same result follows from $b=0$ after using $R_{(-1)}^{(0)}$. We omit superscripts showing the dependence on $b$ and $h$ in what follows. Using \eqref{eq:inp_sj}, we can show
\begin{align}
\inp{f}{p_j} &= 2(1-h)^b\inp*{P_j^{\tau, (-1/2, b, -1/2)}(\sigma)}{f_e(\zeta)}^{\tau, (-1/2, b, -1/2)}, \quad j=0,1,2,\ldots, \\
\inp{f}{q_j} &= 2(1-h)^{b+2}\inp*{P_{j-1}^{\tau, (1/2, b, 1/2)}(\sigma)}{f_o(\zeta)}^{\tau, (1/2, b, 1/2)}, \quad j = 1, 2, 3, \ldots, 
\end{align}
where $\zeta := 1 + (h-1)\sigma$. Now, note
\begin{equation}\label{eq:pjqj21mhbPj_form}
\|p_j\|^2 = 2(1-h)^b\left\|P_j^{\tau,(-1/2,b,-1/2)}\right\|_{\tau, (-1/2, b, -1/2)}^2, ~~\|q_j\|^2 = 2(1-h)^{b+2}\left\|P_j^{\tau,(1/2,b,1/2)}\right\|_{\tau,(1/2,b,1/2)}^2.
\end{equation}
Using these expressions, we find
\begin{equation}
S_n^{(b, h)}[f] = s_n^{\tau,(-1/2,b,-1/2)}[f_e(\zeta)] + ys_{n-1}^{\tau,(1/2,b,1/2)}[f_o(\zeta)].
\end{equation}
Assuming $s_n^{\tau,(-1/2,b,-1/2)}[f_e(\zeta)] \to f_e(\zeta)$ and $s_{n-1}^{\tau,(1/2,b,1/2)}[f_o(\zeta)]$ as $n \to \infty$, we have
\begin{align*}
S_n^{(b, h)}[f](x, y) &\to f_e(\zeta) + yf_o(\zeta) = f(x, y),
\end{align*}
as required. The equivalent statement about convergence in norm follows from noting that
\begin{equation}\label{eq:lem:fmsnfform}
\left\|f - S_n[f]\right\|^2 = 2(1-h)^b\left[\left\|f_e - s_n^{\tau,(-1/2,b,-1/2)}[f_e(\zeta)]\right\|_{\tau,(-1/2,b,-1/2)}^2 + (1-h)^2\left\|f_o - s_{n-1}^{\tau, (1/2, b, 1/2)}[f_o(\zeta)]\right\|_{\tau,(1/2,b,1/2)}^2\right]. \qedhere
\end{equation}
\end{proof}

\begin{theorem}
Let $f(\theta)$ be a $2\pi$-periodic function for $\theta \in \mathbb R$ and consider a grid $\bm\theta$ with $M$ elements. If, over each element $E_i$, $f$ has a convergent series in the $\vb P^{(b)}$ basis, then $f$ has a convergent series in the $\vb P^{(b), \bm\theta}$ basis, $b \in \{-1, 0\}$. In particular, if $f_e$ and $f_o$ have convergent series in the $\vb P^{\tau_i, (-1/2,b,-1/2)}$ and $\vb P^{\tau_i, (1/2,b,1/2)}$ bases, respectively, over each element, $f$ has a convergent series in the $\vb P^{(b)}$ basis.
\end{theorem}

\begin{proof}
This is an immediate consequence of Proposition \ref{prop:piecewise_arc_b0_trans}. Note that Proposition \ref{prop:conv_series_arcpol2q} gives the conditions for the convergence of $f$ over each element.
\end{proof}

\subsubsection{Rate of convergence}

Now let us examine the rate of convergence of our series. In particular, if
\begin{equation}
f(x, y) = \hat f_0 p_0(x, y) + \sum_{j=1}^\infty \left[\hat f_{j1}q_j(x, y) + \hat f_{j2}p(x, y)\right],
\end{equation}
we are interested in determining asymptotics for $\|f - S_n[f]\|$, $|\hat f_{j1}|$, and $|\hat f_{j2}|$\rev{, explicitly relating the convergence rates to the function's analyticity, the element size, and the polynomial degree}. To discuss this convergence, we first need the following definition.

\begin{definition}
We use $E_{\rho, c}$ to denote the open ellipse centred at the origin with foci at $(\pm c, 0)$ whose semimajor and semiminor axis lengths sum to $\rho$; note that $E_{\rho, 1}$ is the \emph{Bernstein ellipse} \cite{trefethen2019approximation}. When necessary, we will instead let $E_{\rho, [a, b]}^x$ denote such an ellipse whose foci are instead at $a$ and $b$ whose centre is at $(x, 0)$.
\end{definition}

The arguments we apply in what follows mirror those used for the half-range Chebyshev polynomials of \cite{huybrechs2010fourier}.

\begin{lemma}\label{lem:seqorthpol_lem_}
Let $\{r_n(x)\}$ be a sequence of orthogonal polynomials on $[a, b]$ with respect to the positive and integrable weight $w$. Given a function $g$, define $g_n(x) = \sum_{j=0}^n c_jr_j(x)$ where $c_j = \int_a^b g(x)r_j(x)w(x) \dx /\int_a^b r_j(x)^2 w(x) \dx$. If $g$ is analytic in $[a, b]$ and analytically continuable to $E_{\rho(b-a)/2, [a, b]}^{(a+b)/2}$ bounded there, where $\rho>1$, then $\|g - g_n\|_w = \mathcal O(\rho^{-n})$ as $n \to \infty$, where $\|g - g_n\|_w^2 = \int_a^b (g(x)-g_n(x))^2 w(x) \dx$.
\end{lemma}

\begin{proof}
This is a restatement of \cite[Lemma 3.9]{huybrechs2010fourier} on the interval $[a, b]$.
\end{proof}

\begin{theorem}\label{thm:rate_of_conv_main_rho}
Let $f(\cos\theta, \sin\theta)$ be a $2\pi$-periodic function \rev{restricted to} an arc $\Omega_h$, $|h| < 1$, and let $\cos\varphi = h$. If $f$ is analytic in $[-\varphi, \varphi]$ and analytically continuable to the region
\begin{equation}\label{eq:defDrho}
D_h(\rho) := \arccos\left(E_{\rho(1-h)/2, [h, 1]}^{(1 + h)/2}\right) := \left\{\theta : x = \cos \theta,\, x \in E_{\rho(1-h)/2, [h, 1]}^{(1+h)/2}\right\},
\end{equation}
where $\rho > 1$, and if $f$ is bounded in $D_h(\rho)$, then for $b \geq -1$
\begin{equation}
\left\|f - S_n^{(b, h)}[f]\right\|_{(b, h)} = \mathcal O(\rho^{-n})\quad \textnormal{as}~n\to\infty.
\end{equation}
Similarly, all the coefficients in the expansion are $\mathcal O(\rho^{-n})$ as $n \to \infty$.
\end{theorem}

\begin{proof}
We omit the dependence on $(b, h)$ unless necessary in what follows, and assume $b > -1$ for now. To start, consider $f_e(x) \equiv f_e(\cos\theta)$, recalling Definition \ref{def:odd_even_part_arc_+}. Since $f_e$ is even, its degree $n$ approximant takes the form
\begin{equation}
f_e^{(n)}(\cos\theta) = \sum_{j=0}^n \hat f_{j2}p_j(\cos\theta),
\end{equation}
writing $p_j(\cos\theta) \equiv p_j(\cos\theta, \sin\theta)$ since $p_j$ is independent of $\sin\theta$. Then
\begin{align*}
\left\|f_e(\cos\theta) - f_e^{(n)}(\cos\theta)\right\|^2 &= \int_{-\varphi}^\varphi \left(f_e(\cos\theta) - f_e^{(n)}(\cos\theta)\right)^2 w(\cos\theta) \dthe = 2\int_h^1 \left(f_e(x) - f_e^{(n)}(x)\right)^2 \frac{w(x)}{\sqrt{1 - x^2}} \dx.
\end{align*}
Since $f_e(x)$ is $2\pi$-periodic and analytic in $E_{\rho(1 - h)/2, [1, h]}^{(1+h)/2}$, the singularity from $1/\sqrt{1-x^2}$ squares away and thus, applying Lemma \ref{lem:seqorthpol_lem_}, we see $\|f_e(\cos\theta-f_e^{(n)}(\cos\theta)\|^2 = \mathcal O(\rho^{-2n})$ for $x \in E_{\rho(1-h)/2, [h, 1]}^{(1+h)/2}$, meaning for $\theta \in D_h(\rho)$. A similar argument applies to the odd part $f_o(\cos\theta)$. Combining these results using $f(x, y) = f_e(x) + yf_o(x)$ gives the result for $b > -1$, noting the expression for $\|f - S_n[f]\|$ from \eqref{eq:lem:fmsnfform}. The result for the coefficients follows immediately from the asymptotics on the error. 

Now we consider $b = -1$. Letting $\hat f_{02} := \hat f_0$ and using superscripts to indicate the dependence on $b$, $R_{(-1)}^{(0)}\hat{\vb f}^{(-1)} = \hat{\vb f}^{(0)}$ together with \eqref{eq:conn_mat_rbb+1_02bm_arc_prop4} gives
\begin{align}
a_{jj}^{(0)}\hat f_{j2}^{(-1)} + a_{j,j+1}^{(0)}\hat f_{j+1,2}^{(-1)} &= \hat f_{j2}^{(0)}, \qquad b_{j-1,j-1}^{(0)}\hat f_{j1}^{-1)} + b_{j-1,j}^{(0)}\hat f_{j+1,1}^{(-1)} = \hat f_{j1}^{(0)}, \qquad j=0,1,2,\ldots.
\end{align}
Thus, we have the bounds $|\hat f_{n1}^{(-1)}| \leq |\hat f_{n1}^{(0)}|/|b_{n-1,n-1}^{(0)}|$ and $|\hat f_{n2}^{(-1)}| \leq |\hat f_{n2}^{(0)}|/|a_{nn}^{(0)}|$ from which the result on the coefficients and the error when $b=-1$ follows; note that $1/|b_{n-1,n-1}^{(0)}|$ and $1/|a_{nn}^{(0)}|$ are asymptotically bounded \cite{kuijlaars2004riemann}.
\end{proof}

We can use Theorem \ref{thm:rate_of_conv_main_rho} to give an analogous result for the piecewise arc polynomial basis, stated below in Theorem \ref{thm:rate_of_conv_main_rho:all__}. In what follows, given a grid $\bm\theta = (\theta_1,\ldots,\theta_{n+1})\tran$, $n \geq 2$, we use the notation
\begin{equation*}
S_{m, i}^{(0), \bm\theta}[f](x, y) = \hat f_0^{(0, h_i)}p_0^{(0, h_i)}(a_i(\theta)) + \sum_{j=1}^m \left[\hat f_{j2}^{(0, h_i)}q_j^{(0, h_i)}(a_i(\theta)) + \hat f_{j1}^{(0, h_i)}p_j^{(0, h_i)}(a_i(\theta))\right], \quad S_m^{(0), \bm\theta}[f](x, y) = \sum_{i=1}^n S_{m, i}^{(0), \bm\theta}[f](x, y)
\end{equation*}
so that, given an expansion $f(x, y) = \vb P^{(0), \bm\theta}\vb f^{(0), \bm\theta}$, we have $f(x, y) = \lim_{m \to \infty} S_m^{(0), \bm\theta}[f](x, y)$. Similarly, we define
\begin{align*}
S_{m, i}^{(-1), \bm\theta}[f](x, y) &= \hat f_i^{(-1, h_i)}\phi_i(\theta) + \hat f_{12}^{(-1, h_i)}p_1^{(-1, h_i)}(a_i(\theta)) + \sum_{j=2}^m \left[\hat f_{j1}^{(-1, h_i)}q_j^{(-1, h_i)}(a_i(\theta)) + \hat f_{j2}^{(-1,h_i)}p_j^{(-1,h_i)}(a_i(\theta))\right], \\
S_m^{(-1), \bm\theta}[f](x, y) &= \sum_{i=1}^n S_{m,i}^{(-1),\bm\theta}[f](x, y).
\end{align*}
We define the norm
\begin{equation}
\|g\|_{(-1), \bm\theta}^2  = \sum_{i=1}^n \int_{\theta_i}^{\theta_{i+1}} g(\theta)^2 \frac{1}{\cos a_i(\theta) - h} \dthe,
\end{equation}
given a grid $\bm\theta=(\theta_1,\ldots,\theta_{n+1})\tran$, and $\|g\|_{(0), \bm\theta}^2 = \int_{-\pi}^\pi g(\theta)^2 \dthe$.

\begin{theorem}\label{thm:rate_of_conv_main_rho:all__}
Let $f(\cos\theta, \sin\theta)$ be a $2\pi$-periodic function and take $b \in \{-1, 0\}$ together with the grid $\bm\theta = (\theta_1, \ldots, \theta_{n+1})\tran$, $n \geq 2$. Let $F_i(\theta)$ be the restriction of $F(\theta) := f(\cos\theta,\sin\theta)$ to $[\theta_i,\theta_{i+1}]$. Suppose each $F_i$ is analytic in $[\theta_i, \theta_{i+1}]$ and analytically continuable to 
\[
D_{h_i}^i(\rho_i) := \left\{\theta + \frac{\theta_i+\theta_{i+1}}{2} : \theta \in D_{h_i}(\rho_i)\right\}
\] with $\rho_i > 1$, and assume $F_i$ is bounded in $D_{h_i}^i(\rho_i)$, $i=1,\ldots,n$. Define $\rho := \min_{i=1}^n \rho_i$. Then
\begin{equation}
\left\|f - S_m^{(b), \bm\theta}\right\|_{(b), \bm\theta} = \mathcal O(\rho^{-m}) \quad\text{as}~m \to \infty.
\end{equation}
Similarly, all the coefficients in the expansion are $\mathcal O(\rho^{-m})$ as $m \to \infty$.
\end{theorem}

\begin{proof}
This is a simple corollary of Theorem \ref{thm:rate_of_conv_main_rho} together with Proposition \ref{prop:piecewise_arc_b0_trans}.
\end{proof}

Now let us give an example of Theorem \ref{thm:rate_of_conv_main_rho:all__}. To give further comparisons, we will also compare against the Fourier basis $\vb F := (1, \sin\theta, \cos\theta, \sin2\theta, \cos2\theta, \ldots)$ and a periodic version of the \emph{piecewise integrated Legendre} basis introduced in \cite{knook2024quasi}, denoted $\vb W^{(b), \bm\theta}$ for $b \in \{-1, 0\}$. The precise definition of this basis is given in \ref{app:integ_legen_def}. We can directly compare the analogous regions to $D_h(\rho)$ for the convergence of the coefficients in the $\vb F$ basis and in the $\vb W^{(b), \bm\theta}$ basis. In particular:
\begin{enumerate}
\item If $f(\theta)$ is $2\pi$-periodic, analytic, and bounded in the open strip of half-width $\alpha_F$ around the real axis, then the $n$th Fourier coefficient is $\mathcal O(\e^{-\alpha_F n})$ \cite{wright2015extension}, or equivalently $\mathcal O(\rho_F^{-n})$ where $\rho_F = \log\alpha_F$. We denote this open strip by $H(\rho_F) := \{z \in \mathbb C : |\!\operatorname{Im} z| < \rho_F\}$. 
\item The analogous result to Theorem \ref{thm:rate_of_conv_main_rho:all__} for $\vb W^{(b), \bm\theta}$ assumes that each $F_i$ is analytic in $[\theta_i, \theta_{i+1}]$ and analytically continuable to $E_i(\rho_i) := E_{\rho_i\ell_i/2, [\theta_i, \theta_{i+1}]}^{(\theta_i + \theta_{i+1})/2}$ with $\rho_i > 1$. Using Lemma \ref{lem:seqorthpol_lem_} gives a convergence rate of $\mathcal O(\rho_W^{-m})$ as $m \to \infty$, $\rho_W = \min_{i=1}^n \rho_i$. \colr{Here, $m$ is the polynomial degree to which we take the  approximations.}
\end{enumerate}

\begin{figure}[h!]
\centering
\includegraphics[width=\textwidth]{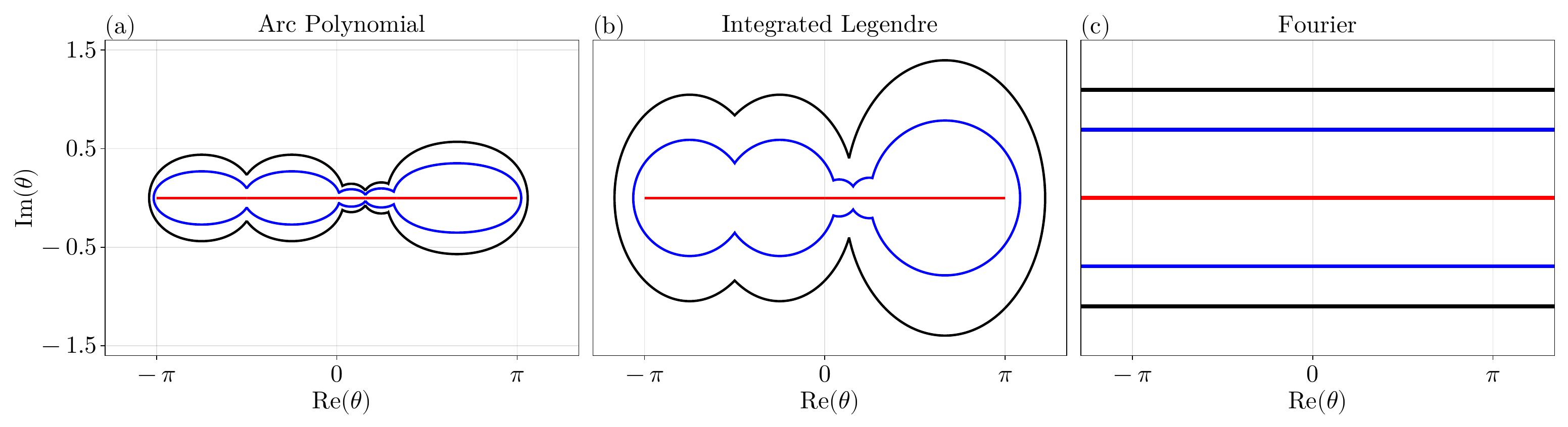}
\caption{Convergence regions for the three bases at $\rho=1$, $\rho=2$, and $\rho=3$. (a) Union of the convergence regions $D_{h_i}^i(\rho)$ for the piecewise arc polynomial basis $\vb P^{(b), \bm\theta}$, $b \in \{-1, 0\}$, for each $\rho \in \{1,2,3\}$ with $\bm\theta = (-\pi,-\pi/2,0,1/2,\pi/3,\pi)\tran$. (b) Union of the convergence regions $E_i(\rho)$ for the periodic piecewise integrated Legendre basis $\vb W^{(b), \bm\theta}$, $b \in \{-1, 0\}$, for each $\rho \in \{1, 2, 3\}$ with $\bm\theta$ as in (a). (c) Convergence regions $H(\log\rho)$ for the Fourier basis $\vb F$ for each $\rho \in \{1, 2, 3\}$; the logarithm is to compare convergence rates all with the same base $\rho$. The disparity in the sizes of the regions between the three bases highlights how functions need to be analytic in a much smaller region for expansions in the arc polynomial basis compared to those in the $\vb W^{(b), \bm\theta}$ and $\vb F$ bases for the same rate of convergence.}\label{fig:convergence_regions}
\end{figure}

We show in Figure \ref{fig:convergence_regions} a comparison of these three regions in the complex $\theta$-plane with the non-uniform grid $\bm\theta=(-\pi,-\pi/2,0,1/2,\pi/3,\pi)\tran$, showing contours for the convergence rates $\mathcal O(1)$, $\mathcal O(2^{-n})$, and $\mathcal O(3^{-n})$. The region in (a) is much smaller than those in (b)--(c), indicating that, to reach the same rate of convergence, expansions in the $\vb P^{(b), \bm\theta}$ basis require a much smaller region of analyticity than those in the $\vb W^{(b), \bm\theta}$ and $\vb F$ bases. To consider the convergence rates for an actual function, consider
\begin{equation}
f(\theta) = \frac{\e^{\sin x}}{\cos(2\theta) - \cosh(1/5)}
\end{equation}
which is singular at $\theta = \pm \i/10$. The results we obtain are shown in Figure \ref{fig:convergence_slopes}, where we use $\bm\theta=(-\pi,0,\pi)\tran$. To compare the bases, we consider the coefficients against both the polynomial degree and the total degrees of freedom. In Figure \ref{fig:convergence_slopes}(a) the points shown are those with the maximum magnitude at the associated polynomial degree, and we see that the $\vb P^{(0), \bm\theta}$ converges at a much faster rate, namely $\mathcal O(\rho^{-n})$ with $\rho \approx 1.5757$, than the $\vb W^{(0), \bm\theta}$ and $\vb F$ bases which converge at the approximate rates $\mathcal O(1.2887^{-n})$ and $\mathcal O(1.1052^{-n})$, respectively. Figure \ref{fig:convergence_slopes}(b) instead plots each individual coefficient's magnitude so that we compare against the degrees of freedom. We see that the Fourier basis converges much slower than the other two bases, but the convergence rate is similar between the $\vb P^{(0), \bm\theta}$ and $\vb W^{(0), \bm\theta}$ bases. Despite this similarity, as we will see in Section \ref{sec:solve_de_} the $\vb P^{(0), \bm\theta}$ \rev{basis} is preferred due to its ability to better represent periodic functions.

\begin{figure}[h!]
\centering
\includegraphics[width=\textwidth]{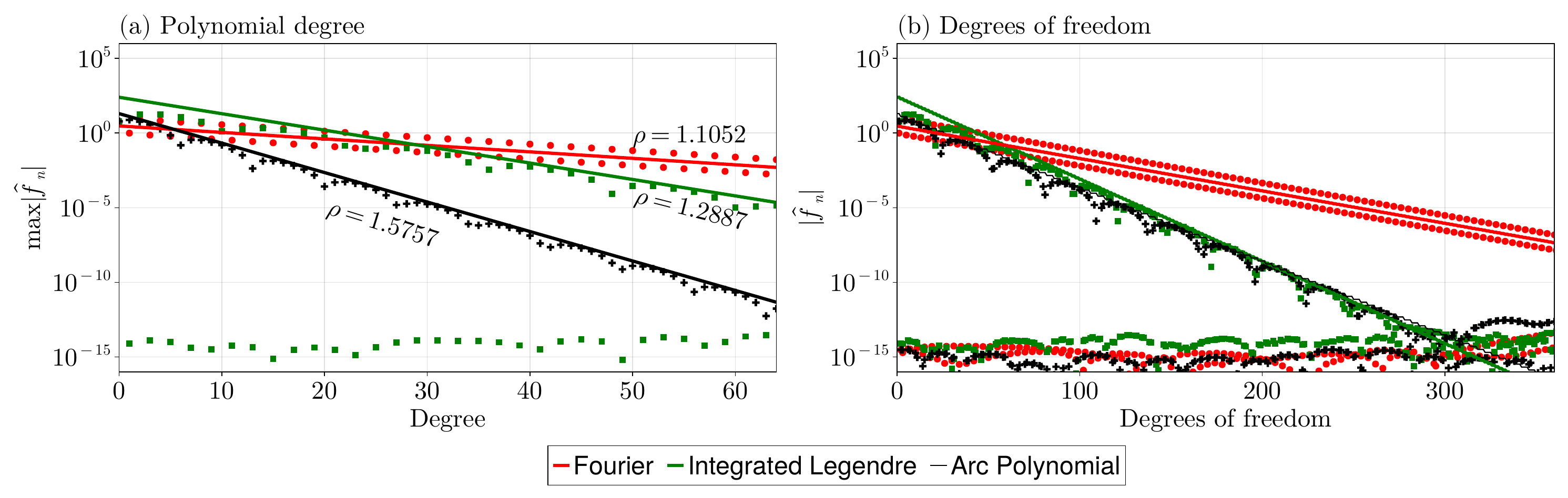}
\caption{Expansion coefficients for $f(\theta)=(\cos(2\theta)-\cosh(1/5))^{-1}$ in the piecewise arc polynomial basis $\vb P^{(0), \bm\theta}$, the periodic piecewise integrated Legendre basis $\vb W^{(0), \bm\theta}$, where $\bm\theta=(-\pi,0,\pi)\tran$, and the Fourier basis $\vb F$. (a) For $\vb P^{(0), \bm\theta}$ the coefficients shown are those with the highest magnitude for each degree across each element, and similarly for $\vb W^{(0), \bm\theta}$. The Fourier coefficients shown are the highest magnitude between the coefficient on $\sin n \theta$ and $\cos n \theta$ for each $n$. The solid lines show the convergence rates $\mathcal O(\rho^{-n})$ using the maximal $\rho$ for each basis; the maximal $\rho$ for the convergence regions for $\vb P^{(0), \bm\theta}$ is $\rho \approx 1.5757$, for $\vb W^{(0), \bm\theta}$ we have $\rho \approx 1.2887$, and for $\vb F$ we have $\rho \approx 1.1052$. We see that the coefficients in the $\vb P^{(0), \bm\theta}$ convergence  faster than those in the other two bases. (b) Similar to (a) except we plot each coefficient's magnitude instead of aggregating over each polynomial degree, thus comparing degrees of freedom. The Fourier basis converges much slower than the other two bases, and the $\vb P^{(0), \bm\theta}$ and $\vb W^{(0), \bm\theta}$ have a similar convergence rate.}\label{fig:convergence_slopes}
\end{figure} 

\section{Solving differential equations}\label{sec:solve_de_}

The primary motivation for this work is the solution of piecewise-smooth differential equations with periodic boundary conditions. The Fourier basis, as is typically used for problems with periodic boundary conditions \cite{boyd2001chebyshev}, is not appropriate for these problems. In particular, for solutions not analytic in a sufficiently large strip around the real axis, the number of coefficients needed to resolve the solution to machine precision is large. Piecewise bases such as our $\vb P^{(0),\bm\theta}$ are more appropriate for such equations since the elements can be defined to align with the discontinuities, giving a smooth function inside each element. For this reason, in what follows, we compare results when using the piecewise arc polynomial basis $\vb P^{(b),\bm\theta}$, the periodic piecewise integrated Legendre basis $\vb W^{(b),\bm\theta}$, and the Fourier basis $\vb F$.

\paragraph{Code availability} The code that implements our bases and produces the results in the examples that follow is provided in \texttt{ArcPolynomials.jl} \cite{arc}.

\subsection{Screened Poisson equation}\label{sec:screened_poisson}

Our first example considers the screened Poisson equation in one dimension  \cite{knook2024quasi}
\begin{equation}\label{eq:screened_poisson}
-\odv[2]{u(\theta)}{\theta} + \omega^2 u(\theta) = f(\theta)
\end{equation}
with periodic boundary conditions, where $f$ is $2\pi$-periodic and $\omega > 1$ is a frequency parameter. In this case, the Fourier basis is preferable \colr{when $f$ is analytic in a sufficiently large horizontal strip centred on the real axis}, otherwise the piecewise bases will be preferable. We consider general $f$ to start, but use a discontinuous initial condition $f(\theta) = 2 + \sgn(|\theta| - \pi/3)$ for a practical example afterwards. To solve \eqref{eq:screened_poisson}, we first put it into weak form. Taking $v \in H_{\mathrm{per}}^1(I)$, where $I := [-\pi, \pi]$, $H_{\mathrm{per}}^s(I) := \{v \in H^s(I) : v~\textnormal{is}~2\pi\textnormal{-periodic}\}$, $H^s(I) := W^{s, 2}(I)$ is the typical Sobolev space \cite{adams2003sobolev, knook2024quasi}, the weak form of \eqref{eq:screened_poisson} is: find $u \in H_{\mathrm{per}}^1(I)$ such that
\begin{equation}\label{eq:weak_screened_poisson}
\inp{v'}{u'} + \omega^2 \inp{v}{u} = \inp{v}{f} \quad \forall v \in H_{\mathrm{per}}^1(I),
\end{equation}
where $f \in H_{\mathrm{per}}^{-1}(I)$ and $\inp{f}{g} := \int_{-\pi}^\pi f(\theta)g(\theta) \dthe$. Focusing on the $\vb P^{(-1), \bm\theta}$ basis first for a given grid $\bm\theta$, we expand $v(\theta) = \vb P^{(-1), \bm\theta}\vb v$ and $u(\theta) = \vb P^{(-1), \bm\theta}\vb u$ . Since $f \in H_{\mathrm{per}}^{-1}(I)$, allowing for discontinuities in $f$, we expand $f$ in the $b=0$ basis rather than the $b=-1$ basis, giving $f(\theta) = \vb P^{(0), \bm\theta}\vb f$ so that \eqref{eq:weak_screened_poisson} becomes
\begin{equation}\label{eq:weak_screened_poisson_vec_v}
\vb v\tran\left(-\Delta^{(-1), \bm\theta} + \omega^2 M^{(-1), \bm\theta}\right)\vb u = \vb v\tran\left[R_{(-1)}^{(0), \bm\theta}\right]\tran M^{(0), \bm\theta}\vb f.
\end{equation}
Enforcing \eqref{eq:weak_screened_poisson_vec_v} for all $\vb v$ gives
\begin{equation}\label{eq:weak_screened_poisson_vec}
\left(-\Delta^{(-1), \bm\theta} + \omega^2 M^{(-1), \bm\theta}\right)\vb u = \left[R_{(-1)}^{(0), \bm\theta}\right]\tran M^{(0), \bm\theta}\vb f.
\end{equation}
A similar equation to \eqref{eq:weak_screened_poisson_vec} holds for the $\vb W^{(-1), \bm\theta}$ basis. In the Fourier basis we work with the strong form \eqref{eq:screened_poisson}, giving
\begin{equation}\label{eq:screened_poisson_fourier}
\left(-D^2 + \omega^2I\right)\vb u_{\vb F} = \vb f_{\vb F},
\end{equation}
where $\d\vb F/\d\theta = \vb FD$, $u(\theta) = \vb F(\theta)\vb u_{\vb F}$, and $f(\theta) = \vb F(\theta)\vb f_{\vb F}$. To solve these infinite dimensional systems, we need to project them into a finite dimensional space. To this end, define the $n \times \infty$ projection operator $\mathcal P_n = (I_n, \vb 0)$ so that, given a matrix $A \in \mathbb R^{\infty \times \infty}$, $\mathcal P_nA\mathcal P_n\tran \in \mathbb R^{n \times n}$ is the $n \times n$ principal finite section of $A$. We call $n$ the \emph{truncation size}. The solution $\vb u$ to \eqref{eq:weak_screened_poisson_vec} is thus approximated by the solution to the projected problem, namely $\vb u \approx (\tilde{\vb u}\tran, \vb 0)\tran$ where
\begin{equation}\label{eq:weak_screened_poisson_projection_sol}
\tilde{\vb u}\tran =  \left[\mathcal P_n\left(\colr{-}\Delta^{(-1), \bm\theta} + \omega^2M^{(-1), \bm\theta}\right)\mathcal P_n\tran\right]^{-1}\mathcal P_n\left(\left[R_{(-1)}^{(0), \bm\theta}\right]\tran M^{(0), \bm\theta}\vb f\right)\mathcal P_n\tran,
\end{equation}
and similarly in the $\vb W^{(-1), \bm\theta}$ basis. Note that, in practice, $\vb f$ is finite dimensional since we only compute as many coefficients as are needed to resolve $f$ to machine precision; the same holds for other function expansions in the examples that follow. The system \eqref{eq:screened_poisson_fourier} does not require such a projection since $D$ is block diagonal, allowing the system to be solved exactly as $\vb f_{\vb F}$ is finite dimensional. As discussed in Section \ref{sec:factor_cbbb}, we compute $\tilde{\vb u}\tran$ in \eqref{eq:weak_screened_poisson_projection_sol} by using the reverse Cholesky factorisation of $\mathcal P_n(-\Delta^{(-1), \bm\theta} + \omega^2 M^{(-1), \bm\theta})\mathcal P_n\tran$, and similarly for the truncated solution in the $\vb W^{(-1), \bm\theta}$ basis; \colr{note that, for problems where the resulting system is not symmetric positive definite, we can not apply the reverse Cholesky factorisation and so cannot solve the problem in optimal complexity.}

To now make this discussion concrete, we fix $\omega = 3/2$ and use the discontinuous function 
\begin{equation}\label{eq:discont_f_}
f(\theta) = 2 + \sgn\left(|\theta| - \frac{\pi}{3}\right), 
\end{equation}
and our grid is $10$ equally spaced points between $-\pi$ and $\pi$. To machine precision, $f$ requires $9$ coefficients in the $\vb P^{(-1), \bm\theta}$ and $\vb W^{(-1), \bm\theta}$ bases, and the discontinuity in $f$ means that we cannot expand $f$ in the $\vb F$ basis, hence we only consider $\vb P^{(-1),\bm\theta}$ and $\vb W^{(-1),\bm\theta}$ in what follows. Using these expansion lengths, we use the truncation size $n = 144$ for both bases, where this $n$ is the number of coefficients needed plus $15M$ where $M$ is the number of elements; this extra $15M$ comes from needing to account for the block bandwidths from the products in \eqref{eq:weak_screened_poisson_projection_sol}. 

\begin{remark}
The next examples in this section that do not involve a discontinuity will only need a correction of $2M$ due to Theorem \ref{thm:trigpolyexpans22} since, as a consequence, the $\vb P^{(-1), \bm\theta}$ solution should map trigonometric polynomials back into trigonometric polynomials, although the same cannot necessarily be said for the $\vb W^{(-1), \bm\theta}$ basis. 
\end{remark}

We show the solution in each basis in Figure \ref{fig:screened_poisson}, as well as the derivatives for each solution up to the second derivative, using 
\begin{equation}\label{eq:derivative_formulae}
\odv*[2]{\vb P^{(-1), \bm\theta}}{\theta} = \vb P^{(0), \bm\theta}D_{(-1)}^{(0), \bm\theta}\left[R_{(-1)}^{(0), \bm\theta}\right]^{-1}D_{(-1)}^{(0), \bm\theta},
\end{equation}
with a similar result for $\vb W^{(-1), \bm\theta}$. We truncate the matrices in \eqref{eq:derivative_formulae} using $\mathcal P_{15M}$ since, for larger $n$, the condition number of $R_{(-1)}^{(0), \bm\theta}$ becomes too large to get accurate derivatives when applying its inverse. \colr{We show the differences between the numerical solutions and the exact solution in Figure \ref{fig:screened_poisson}, where we find that the errors are comparable between each basis.}

\begin{figure}[h!]
\centering
\includegraphics[width=\textwidth]{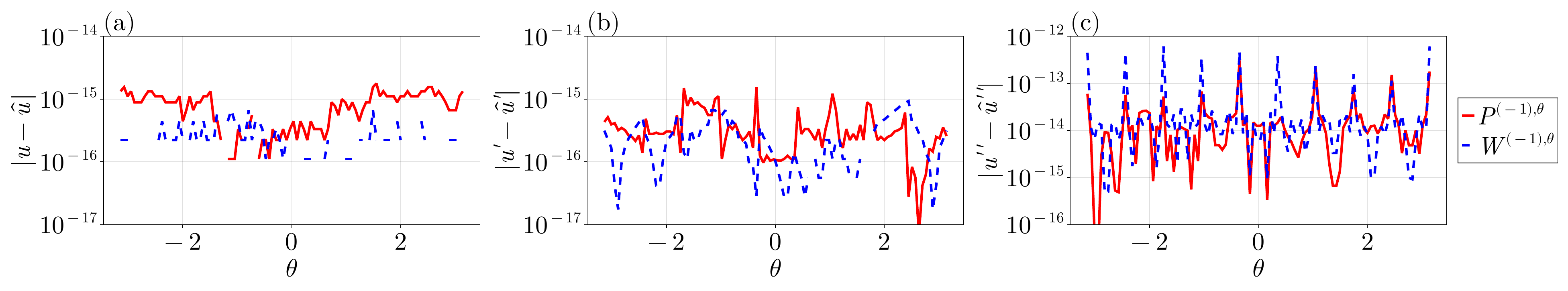}
\caption{\colr{Errors in the solutions and derivatives for the screened Poisson equation $-u'' + \omega^2 u = f$ \eqref{eq:screened_poisson} with $\omega = 3/2$, $f(\theta) = 2 +\sgn(|\theta|-\pi/3)$, and $\bm\theta$ is defined by $10$ equally spaced points between $-\pi$ and $\pi$, using the $\vb P^{(-1), \bm\theta}$ (red, solid) and $\vb W^{(-1), \bm\theta}$ (blue, dashed) bases, with errors computed using solutions and derivatives computed from the exact solution. The truncation sizes for $\vb P^{(-1), \bm\theta}$ and $\vb W^{(-1), \bm\theta}$ are $n = 144$. These results show that the two solutions do agree with minimal error.}}\label{fig:screened_poisson}
\end{figure}

\subsection{Heat Equation}

The second problem we consider is the heat equation,
\begin{equation}\label{eq:heat_eqn}
\pdv{u(\theta, t)}{t} = \pdv[2]{u(\theta, t)}{\theta}.
\end{equation}
The solution to \eqref{eq:heat_eqn} is governed by, using a similar argument to the screened Poisson example, 
\begin{equation}\label{eq:weak_heat_eqn_}
M^{(-1), \bm\theta} \odv{\vb u}{t} = \Delta^{(-1),\bm\theta}\vb u
\end{equation}
in the $\vb P^{(-1),\bm\theta}$ basis, with a similar result for the $\vb W^{(-1), \bm\theta}$ basis. To avoid discussing complications related to timestepping, we solve these systems using the matrix exponential, for example $\vb u(t) = \exp([M^{(-1), \bm\theta}]^{-1}\Delta^{(-1), \bm\theta}t)\vb u(0)$, computing $\vb u(0)$ from the initial condition. We compute the matrix exponential using ExponentialUtilities.jl \cite{ExponentialUtilities.jl2024}. The initial condition over $[-\pi,\pi]$ is given by 
\begin{equation}\label{eq:hat_fnc_approx}
u(\theta, 0) = \begin{cases} 0 & -\pi \leq \theta < -\pi/4-\varepsilon,\\
1 + (\theta-\varepsilon+\pi/4)/(2\varepsilon) & -\pi/4-\varepsilon\leq \theta<-\pi/4+\varepsilon, \\
2 + (\theta + \varepsilon - \pi/4)/\pi & -\pi/4+\varepsilon\leq \theta<\pi/4-\varepsilon, \\
2-(\theta+\varepsilon-\pi/4)/\varepsilon & \pi/4-\varepsilon\leq \theta<\pi/4+\varepsilon, \\
0 & \pi/4+\varepsilon\leq \theta \leq \pi,
\end{cases}
\end{equation}
with $\varepsilon = 0.005$, which is a continuous approximation to the function $\chi$ that is zero outside of $[-\pi/4, \pi/4]$ and linearly increases from $1$ to $2$ over $[-\pi/4,\pi/4]$. We do not use the Fourier basis in this example as \eqref{eq:hat_fnc_approx} requires too many coefficients to approximate. We use the grid $\bm\theta=(-\pi,-\pi/4-\varepsilon,-\pi/4+\varepsilon,\pi/4-\varepsilon,\pi/4+\varepsilon,\pi)\tran$, leading to truncation sizes of $n=55$ and $n=15$ for the $\vb P^{(-1), \bm\theta}$ and $\vb W^{(-1),\bm\theta}$ bases, respectively.

\begin{remark}
Instead of using an approximation to $\chi$, we could expand $\chi$ in the $b=0$ basis so that instead of \eqref{eq:weak_heat_eqn_} we have
\begin{equation}
M^{(0), \bm\theta}\odv{\vb u}{t} = -\left(\odv*{\vb P^{(0),\bm\theta}}{\theta}\right)\tran\left(\odv*{\vb P^{(0),\bm\theta}}{\theta}\right)\vb u.
\end{equation}
The complication with this approach is in the computation of the derivatives, since we have not defined a $b = 1$ basis to represent the derivatives in; this basis could be defined using $b=1$ polynomials together with delta functions, but we do not consider defining $\vb P^{(1), \bm\theta}$ here. We can use
\begin{equation}
-\left(\odv*{\vb P^{(0),\bm\theta}}{\theta}\right)\tran\left(\odv*{\vb P^{(0),\bm\theta}}{\theta}\right) = -\left(R_{(-1)}^{(0), \bm\theta}\right)^{-\mkern-1.5mu\mathsf T}\left(D_{(-1)}^{(0),\bm\theta}\right)\tran M^{(0), \bm\theta} D_{(-1)}^{(0),\bm\theta} \left(R_{(-1)}^{(0), \bm\theta}\right)^{-1},
\end{equation}
however we end up needing many coefficients to resolve the solution accurately due to the need to compute $[R_{(-1)}^{(0), \bm\theta}]^{-1}$ for the matrix exponential. One way around this would be to take a single backward Euler step so that the discontinuity is smoothed out, followed by the matrix exponential for the remaining steps, or to simply use timestepping for the entire integration. We do not consider these alternatives here.
\end{remark}

Rather than only look at the solution, we use this example to analyse the differences between the $\vb P^{(-1),\bm\theta}$ and $\vb W^{(-1),\bm\theta}$ bases. For a given numerical solution $\hat u(\theta, t)$, define the \emph{periodic drift} for its $d$th derivative with respect to $\theta$ to be $\mathcal E_{\mathrm{per}}^{(d)}(t) := |\hat u^{(d)}(\pi, t) - \hat u^{(d)}(-\pi, t)|$, where $\hat u^{(d)}(\theta, t)$ denotes an estimate for the $d$th derivative with respect to $\theta$ at $(\theta, t)$. The drift $\mathcal E_{\mathrm{per}}^{(d)}(t)$ measures how $\hat u^{(d)}(\theta, t)$ loses $2\pi$-periodicity as time increases. The results are shown in Figure \ref{fig:heat_equation_1}. We see that the drifts in the $\vb P^{(-1), \bm\theta}$ basis decay to machine precision over time, as expected since the solution goes to a constant for large time. In contrast, the drifts for the $\vb W^{(-1),\bm\theta}$ basis remain large as time increases, showing that this basis fails to preserve periodicity unlike in the $\vb P^{(-1), \bm\theta}$ basis.

\begin{figure}[h!]
\centering
\includegraphics[width=\textwidth]{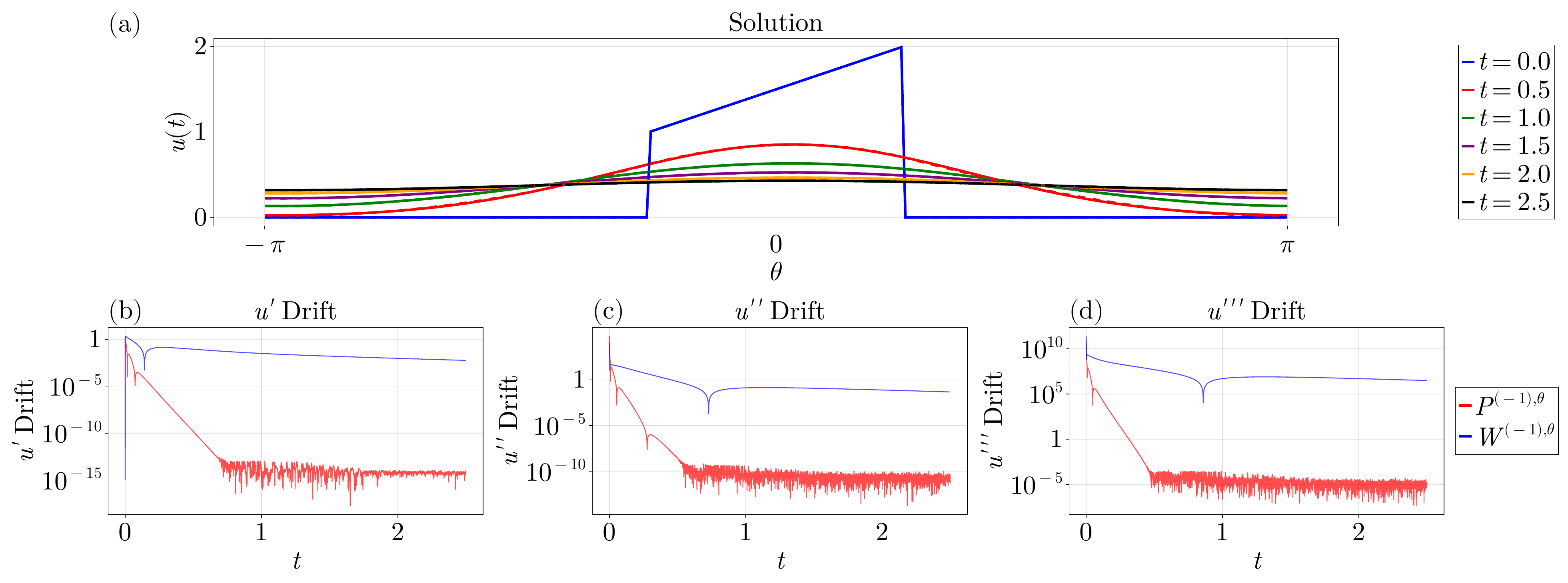}
\caption{Solutions and drifts for the heat equation $\partial_t u = \partial_{\theta\theta} u$ with $u(\theta, 0)$ given by \eqref{eq:hat_fnc_approx}. The grid is given by $\bm\theta=(-\pi,-\pi/4-\varepsilon,-\pi/4+\varepsilon,\pi/4-\varepsilon,\pi/4+\varepsilon,\pi)\tran$ where $\varepsilon=0.005$. The truncation sizes are $n=55$ and $n=15$ for the $\vb P^{(-1),\bm\theta}$ and $\vb W^{(-1),\bm\theta}$ bases, respectively. Primes indicate derivatives with respect to $\theta$. The results show that, while the drift in the solutions with the $\vb P^{(-1),\bm\theta}$ basis decay to zero over time as the solution spreads out to a constant, those in the $\vb W^{(-1),\bm\theta}$ basis remain large.}
\label{fig:heat_equation_1}
\end{figure}

The results in Figure \ref{fig:heat_equation_1} can be explained in terms of eigenfunctions. With the help of Theorem \ref{thm:trigpolyexpans22} and in contrast to the $\vb W^{(-1),\bm\theta}$ basis, it can be shown that the span of the $\vb P^{(-1),\bm\theta}$ basis   contains the low order eigenfunctions of constant coefficient differential operators. This fact implies that the $\vb P^{(-1),\bm\theta}$ basis leads to improved smoothness properties when integrating \eqref{eq:heat_eqn} over time, since the discontinuities in the solution are captured by the spurious high order eigenfunctions of the discretisation. These rapidly dissipate out, leaving only the true low order smooth eigenfunctions. On the other hand, $\vb W^{(-1),\bm\theta}$ only contains the constant eigenfunction, and the remaining eigenfunctions of the discretisation are not smooth. 
We show results with truncation sizes $n=20$ and $n=5$ for the $\vb P^{(-1),\bm\theta}$ and $\vb W^{(-1),\bm\theta}$ bases, respectively, which corresponds to roughly a third of the total number of coefficients needed to resolve the solution, for both bases in Figure \ref{fig:heat_equation_2}. We see in Figure \ref{fig:heat_equation_2} that, despite using fewer coefficients compared to the amount needed to accurately resolve the initial conditions, the drift still decays for the $\vb P^{(-1),\bm\theta}$ basis as expected, while the $\vb W^{(-1),\bm\theta}$ basis leads to more numerical error due to the eigenfunctions of the discretisation not being smooth.

\begin{figure}[h!]
\centering
\includegraphics[width=\textwidth]{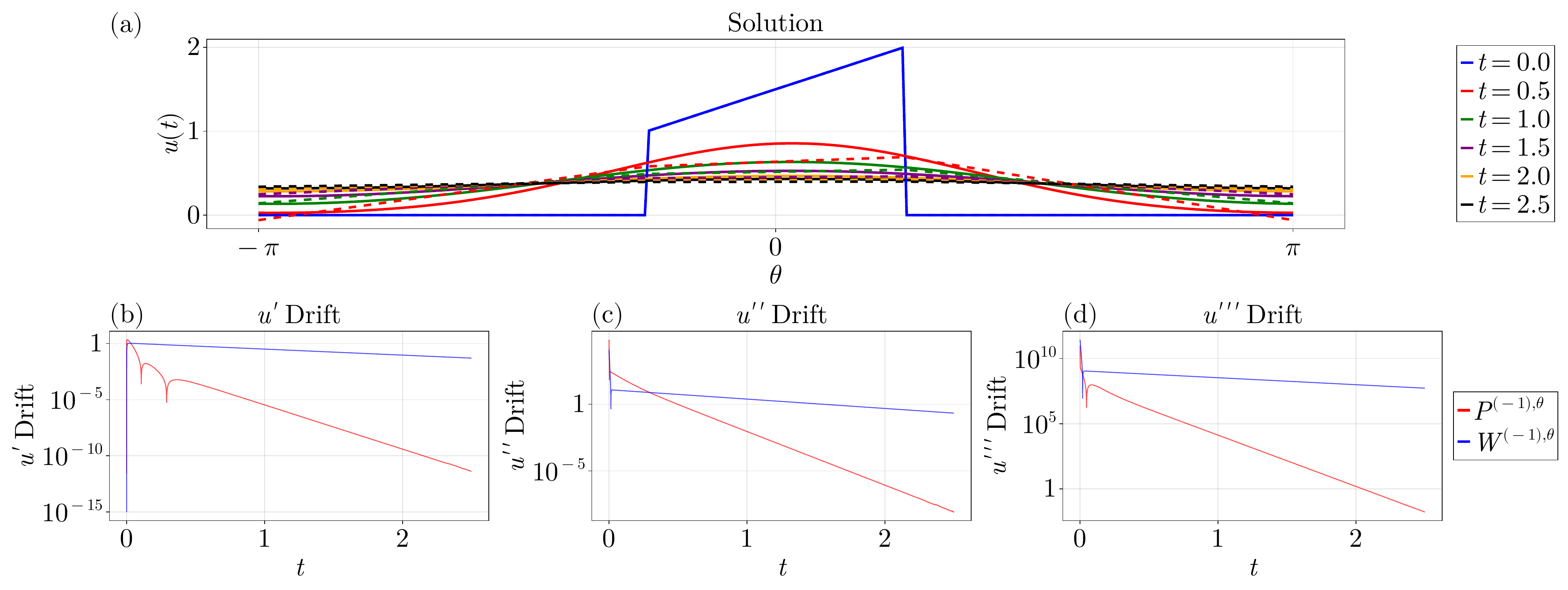}
\caption{Similar to Figure \ref{fig:heat_equation_1}, except now underresolving with truncation size $n=20$ and $n=5$ for the $\vb P^{(-1),\bm\theta}$ and $\vb W^{(-1),\bm\theta}$, respectively; these truncation sizes correspond to using, roughly, only a third of the number of coefficients needed to accurately resolve the solution. The results show that the $\vb P^{(-1),\bm\theta}$ basis still gives accurate results, while the solution with the $\vb W^{(-1),\bm\theta}$ basis has more numerical error in the solution, showing that the drift is causing a loss in accuracy.}
\label{fig:heat_equation_2}
\end{figure}

\rev{We can give a numerical demonstration of the differences in the eigenvalues of the piecewise arc and integrated Legendre bases. Consider the eigenvalue problem $u'' = \lambda u$. In the $\vb P^{(-1), \bm\theta}$ basis, putting this into weak form gives the generalised eigenvalue problem
\begin{align}
-\Delta^{(-1), \bm\theta} \vb u &= \lambda M^{(-1), \bm\theta} \vb u,
\end{align}
and similarly for the $\vb W^{(-1), \bm\theta}$ basis. Using $\bm\theta = (-\pi,-\pi/3,\pi/3,\pi)\tran$ and truncating the matrices to their $12 \times 12$ principal finite sections, we compute the eigenfunctions and plot the first four in Figure \ref{fig:eigval_ex}. We see from Figure \ref{fig:eigval_ex} (a) and (c) that the eigenfunctions computed in the $\vb P^{(-1), \bm\theta}$ basis are in $C^\infty$. In contrast, while the eigenfunctions in the $\vb W^{(-1), \bm\theta}$ basis are smooth as we see in Figure \ref{fig:eigval_ex}(b), the derivatives are not smooth, shown by the sharp jumps in the red curve in Figure \ref{fig:eigval_ex}(d). These differences between the curves show the key differences between the bases, and why the arc polynomial basis has better time integration properties as argued previously for Figures \ref{fig:heat_equation_1}--\ref{fig:heat_equation_2}.
}

\begin{figure}[h!]
\centering
\includegraphics[width=\textwidth]{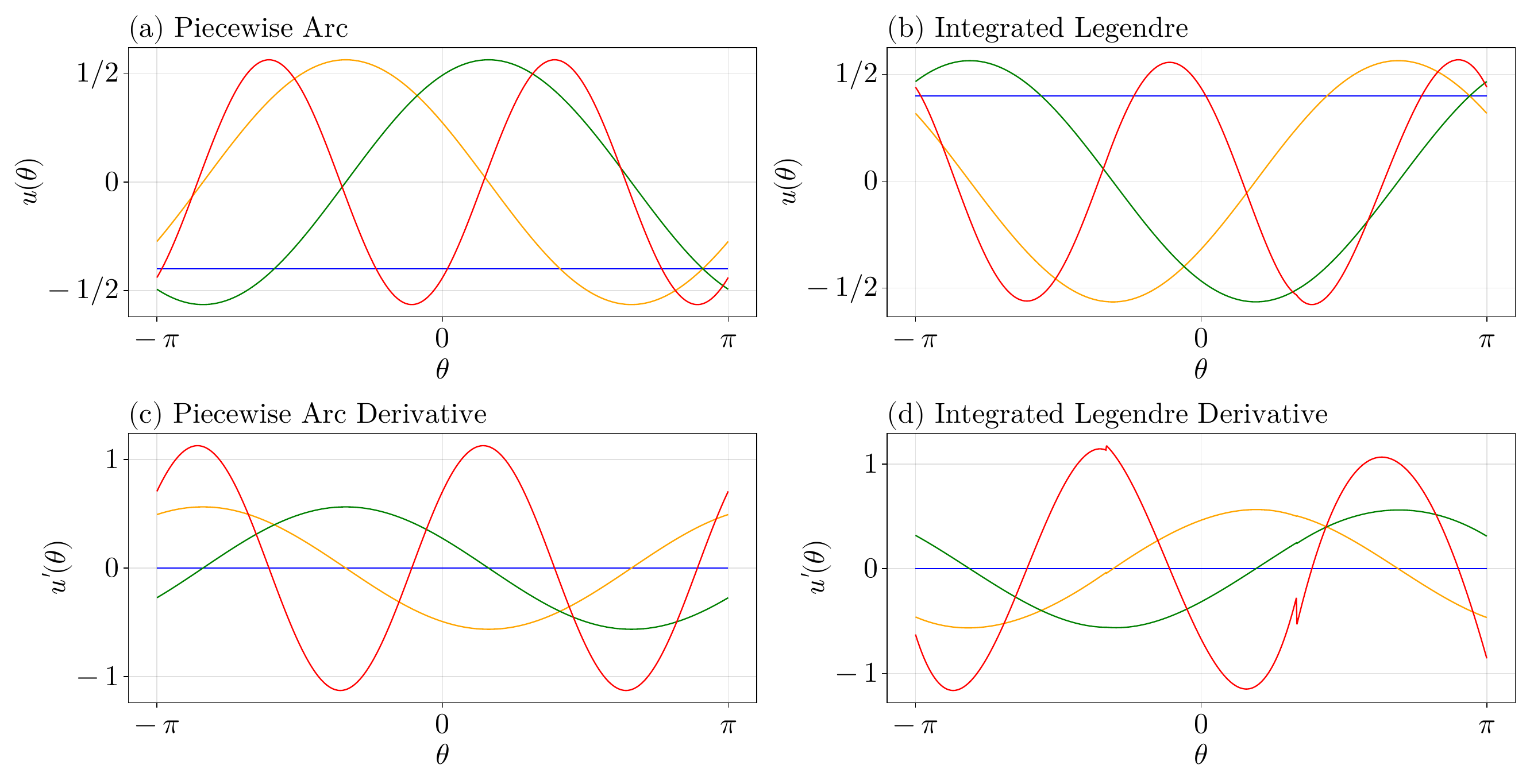}
\caption{\rev{The first four eigenfunctions for the eigenvalue problem $u'' = \lambda u$ in the $\vb P^{(-1), \bm\theta}$ and $\vb W^{(-1), \bm\theta}$ bases, computed with $\bm\theta=(-\pi,-\pi/3,\pi/3,\pi)\tran$ and with a $12 \times 12$ truncation. The blue curve is for $\lambda=0$, the orange and green curves are for $\lambda = 1$, and the red curve is for $\lambda = 4$. The curves plotted in (a) are for the eigenfunctions $u(\theta) = \vb P^{(-1), \bm\theta}\vb u$, and similarly for (b), while in (c)--(d) we plot $u'(\theta)$ instead. The plots show that while the eigenfunctions for the $\vb P^{(-1), \bm\theta}$ basis are in $C^\infty$, those in the $\vb W^{(-1), \bm\theta}$ basis are only continuous, explaining why we see the difference in time integration properties in Figures \ref{fig:heat_equation_1}--\ref{fig:heat_equation_2}.}}
\label{fig:eigval_ex}
\end{figure}

\subsection{Linear Schr\"odinger equation}
\label{sec:schrodinger}
Our third example is the linear Schr\"odinger equation,
\begin{equation}
\i\pdv{u(\theta, t)}{t} = \pdv[2]{u(\theta, t)}{\theta}.\label{eq:schrodinger}
\end{equation}
Using a similar argument to the screened Poisson example, the solution to \eqref{eq:schrodinger} is governed by 
\begin{equation}\label{eq:schrodinger_arc}
M^{(-1), \bm\theta}\odv{\vb u}{t} = -\i\Delta^{(-1), \bm\theta}\vb u
\end{equation}
in the $\vb P^{(-1), \bm\theta}$ basis, with a similar result for the $\vb W^{(-1), \bm\theta}$ basis. The Fourier basis leads to $\partial_t\vb u_{\vb F} = -\i D^2\vb u_{\vb F}$. Similarly to the heat equation example, we use the matrix exponential to solve these systems. We use the grid $\bm\theta = (-\pi,-\pi/3,\pi/3,\pi)\tran$. For this example, we consider the initial condition
\begin{equation}\label{eq:schrodinger_ic}
u(\theta, 0) = \sin(7\theta) + \e^{-\cos\theta}.
\end{equation}
The truncation sizes for $\vb P^{(-1), \bm\theta}$ and $\vb W^{(-1), \bm\theta}$ are $n = 60$ and $n = 93$, respectively. We can compute the matrix exponential for the Fourier solution exactly since $D$ is block diagonal.

The solutions and the drifts for this example are shown in Figure \ref{fig:schrodinger}.  We see that, unlike the $\vb P^{(-1), \bm\theta}$ and $\vb F$ bases, the drift in the $\vb W^{(-1), \bm\theta}$ solution and its derivatives increases over time, showing a disadvantage of the $\vb W^{(-1),\bm\theta}$ basis compared to our $\vb P^{(-1), \bm\theta}$ basis.

\begin{figure}[h!]
\centering
\includegraphics[width=\textwidth]{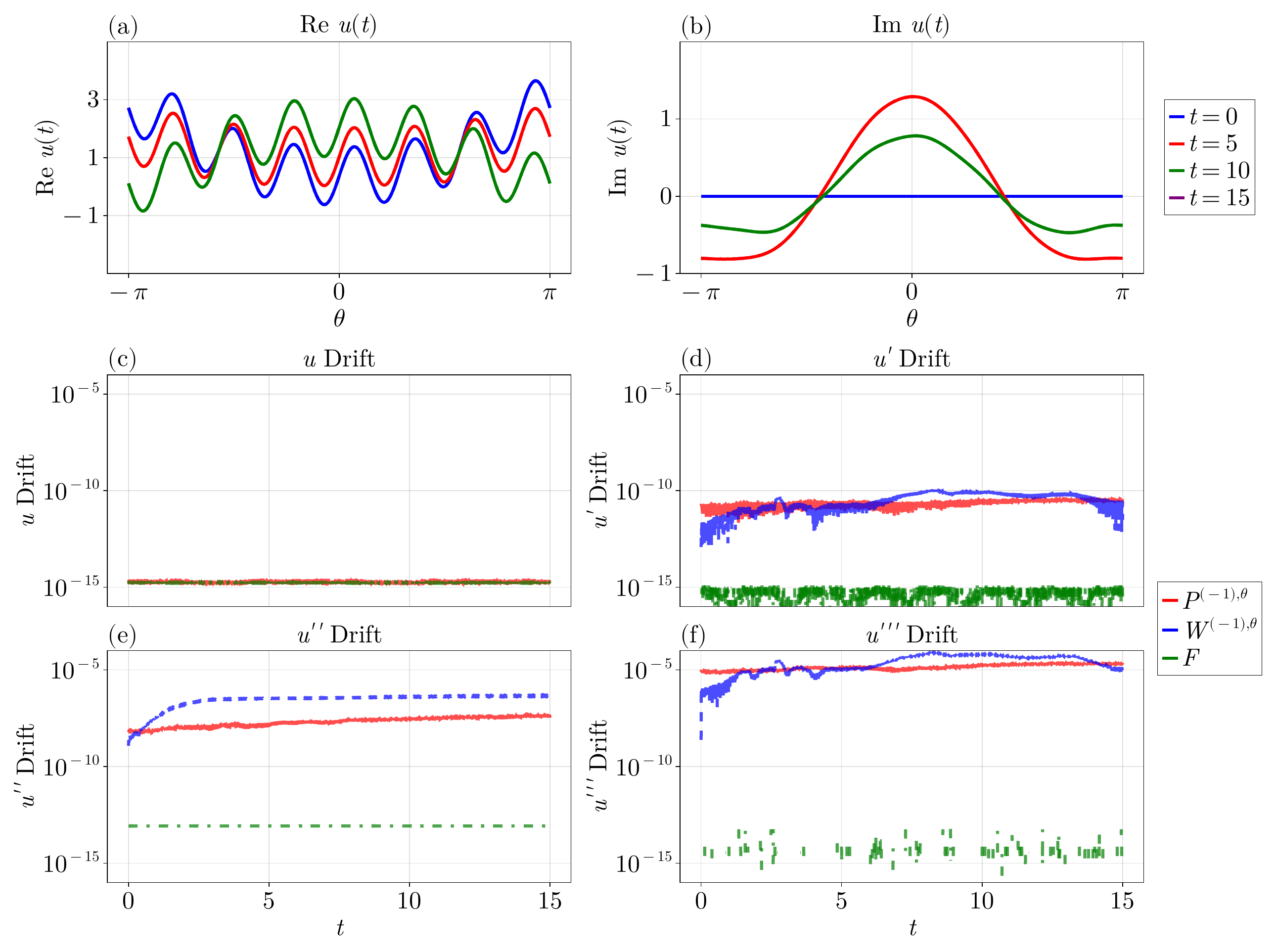}
\caption{Solutions and drifts for Schr\"odinger's equation $\i\partial_t u = \partial_{\theta\theta} u$ with $u(\theta, 0) =  \sin(7\theta) + \e^{-\cos\theta}$ and truncation sizes $n = 60$ and $n = 93$ for the $\vb P^{(-1), \bm\theta}$ and $\vb W^{(-1), \bm\theta}$ bases, respectively, where $\bm\theta = (-\pi,-\pi/3,\pi/3,\pi)\tran$. Primes indicate derivatives with respect to $\theta$. The results show that the drift in the $\vb W^{(-1), \bm\theta}$ solution worsens over time, unlike the $\vb P^{(-1), \bm\theta}$ and $\vb F$ solutions.}\label{fig:schrodinger}
\end{figure}

\subsection{Convection-diffusion equation}

Next we consider the convection-diffusion equation
\begin{equation}
\pdv{u(\theta, t)}{t} = \pdv[2]{u(\theta, t)}{\theta} - v(\theta)\pdv{u(\theta, t)}{\theta}\label{eq:convdiffeq}
\end{equation}
with periodic boundary conditions, where $v(\theta) = -\sin(\theta)/1000$ and the initial condition in $[-\pi, \pi]$ is given by 
\begin{equation}\label{eq:convection_diff_eq_ic}
u(\theta, 0) = \begin{cases} \e^{-\cos4\theta}\sin(3\theta) & |\theta| \leq \pi/4, \\ \e^{-\cos4\theta}\sin(\theta) & \textnormal {otherwise}, \end{cases}
\end{equation}
and defined outside of $[-\pi, \pi]$ through its periodic extension. This is a case which the Fourier basis cannot easily handle due to the derivative discontinuity at $|\theta| = \pi/4$. The piecewise bases can handle this easily by placing the element intersections at the discontinuity, in particular we use the grid $\bm\theta=(-\pi,-\pi/4,\pi/4,\pi)\tran$. Putting the equation into weak form and letting $J_{v}^{(0), \bm\theta}$ be the matrix such that $v(\theta)\vb P^{(0), \bm\theta} = \vb P^{(0), \bm\theta}J_{v}^{(0), \bm\theta}$, we obtain
\begin{equation}\label{eq:conviekqpdifok}
M^{(-1), \bm\theta}\odv{\vb u}{t} = \left[\Delta^{(-1)} - \left[R_{(-1)}^{(0), \bm\theta}\right]\tran M^{(0), \bm\theta}J_{v}^{(0), \bm\theta}D_{(-1)}^{(0),\bm\theta}\right]\vb u,
\end{equation}
with a similar result for the $\vb W^{(-1), \bm\theta}$ basis. We derive this multiplication matrix $J_{v}^{(0), \bm\theta}$ in \ref{app:multmat}. We use the matrix exponential to integrate \eqref{eq:conviekqpdifok}.

The Fourier basis in this case leads to an expansion with $179,956$ coefficients to achieve machine precision, and so it is \rev{prohibitively expensive to use} $\vb F$ for this example. The truncation sizes we use for $\vb P^{(-1), \bm\theta}$ and $\vb W^{(-1), \bm\theta}$ are $n=177$ and $n=216$, respectively. The solutions we obtain up to $t = 2.5$ are shown in Figure \ref{fig:conv_diff__}, where we again show the drift for the solutions. The drifts for $u$ and $\partial u/\partial \theta$ are not shown as they are all negligible or zero. We see that the drift for $\partial^2 u/\partial \theta^2$ is roughly machine precision for the $\vb P^{(-1), \bm\theta}$ basis, while it is much larger in the $\vb W^{(-1), \bm\theta}$ basis. The drifts are comparable for $\partial^3 u/\partial \theta^3$, although the $\vb P^{(-1), \bm\theta}$ curve has less drift. Thus, similar to Figure \ref{fig:schrodinger}, we see that the drift properties are much better in the $\vb P^{(-1), \bm\theta}$ than in the $\vb W^{(-1), \bm\theta}$ basis. \colr{Note that these curves for the drifts are reminiscent of those seen in Figure \ref{fig:heat_equation_1}, and can be similarly explained in terms of eigenfunctions due to the derivative discontinuities in the initial condition.}

\begin{figure}[h!]
\centering
\includegraphics[width=\textwidth]{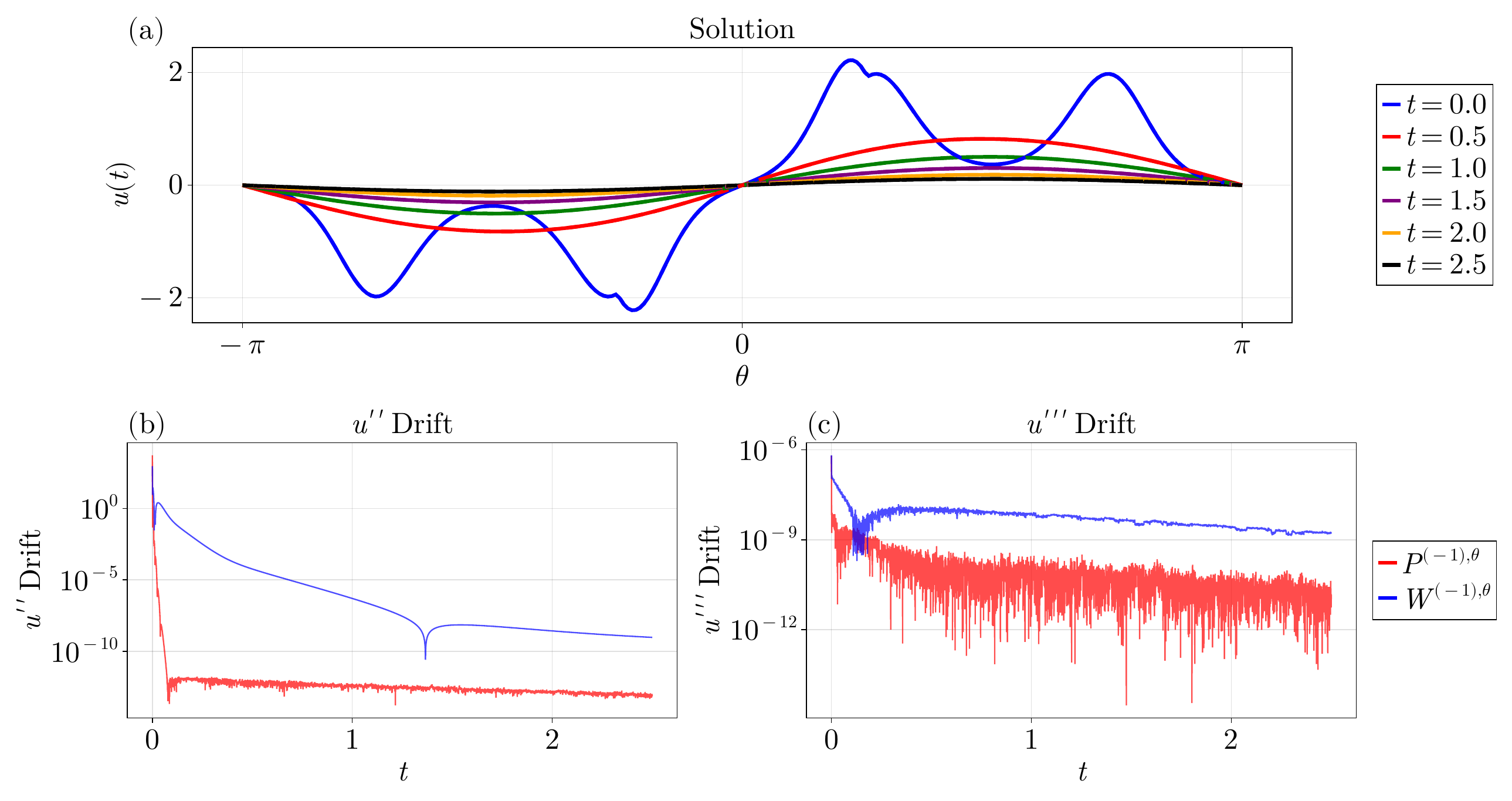}
\caption{Drifts for the solution of the convection-diffusion equation $\partial_t u = \partial_{\theta\theta} u - v\partial_\theta u$ \eqref{eq:convdiffeq} with periodic bundary conditions, where $v(\theta) = -\sin(\theta)/1000$ and $u(\theta, 0) = \e^{-\cos 4\theta} s(\theta)$, where $s(\theta) = \sin(3\theta)$ if $|\theta| \leq \pi/4$ and $s(\theta) = \sin\theta$ otherwise. The grid is $\bm\theta=(-\pi,-\pi/4,\pi/4,\pi)\tran$ and the truncation sizes are $n=60$ and $n=93$ for the $\vb P^{(-1), \bm\theta}$ and $\vb W^{(-1), \bm\theta}$ bases, respectively. Primes indicate derivatives with respect to $\theta$. The plots show that, similar to Figure \ref{fig:schrodinger}, the drift in the $\vb P^{(-1), \bm\theta}$ basis is much smaller than in the $\vb W^{(-1), \bm\theta}$ basis.}\label{fig:conv_diff__}
\end{figure}

\subsection{Complexity}\label{sec:complexity}

We close this section by considering the performance of our method, analysing the computational time required for (1) building and factorising the left-hand side of \eqref{eq:weak_screened_poisson_vec}, and (2) solving the system \eqref{eq:weak_screened_poisson_vec}. We analyse this performance using the same problem as in Section \ref{eq:screened_poisson}, except using $f(\theta) = \exp(\cos(\theta))$ for simplicity to avoid the need to include a grid point to accommodate the discontinuity in $f(\theta) = 2 + \sgn(|\theta|-\pi/3)$ in \eqref{eq:discont_f_}; we still expand $f$ in the $b=0$ basis to keep the examples similar. The results we obtain are shown in Figure \ref{fig:timing_results}, where we observe the optimal complexity in both parts of the solution process, with all lines showing a linear slope against the degrees of freedom used. The performance in building the discretisations with the $\vb W^{(-1), \bm\theta}$ basis is significantly better than in the $\vb P^{(-1), \bm\theta}$ basis. However, these can be reused for other differential equations on the same grid $\bm\theta$ and we emphasise that, for problems involving temporal integration, the build cost is only incurred once instead of at each step.
Furthermore, as argued in Section \ref{sec:analysis__}, the $\vb P^{(-1), \bm\theta}$ basis requires fewer degrees of freedom for an accurate solution than in the $\vb W^{(-1), \bm\theta}$ basis, making this difference less important.  There is still room to improve on the performance of our method significantly, for example improving some of the underlying code used for the semiclassical Jacobi polynomials \cite{SemiPoly.jl2024}, however we do not explore that here as the important part of Figure \ref{fig:timing_results} is the optimal complexity result. 

\begin{figure}[h!]
\centering
\includegraphics[width=\textwidth]{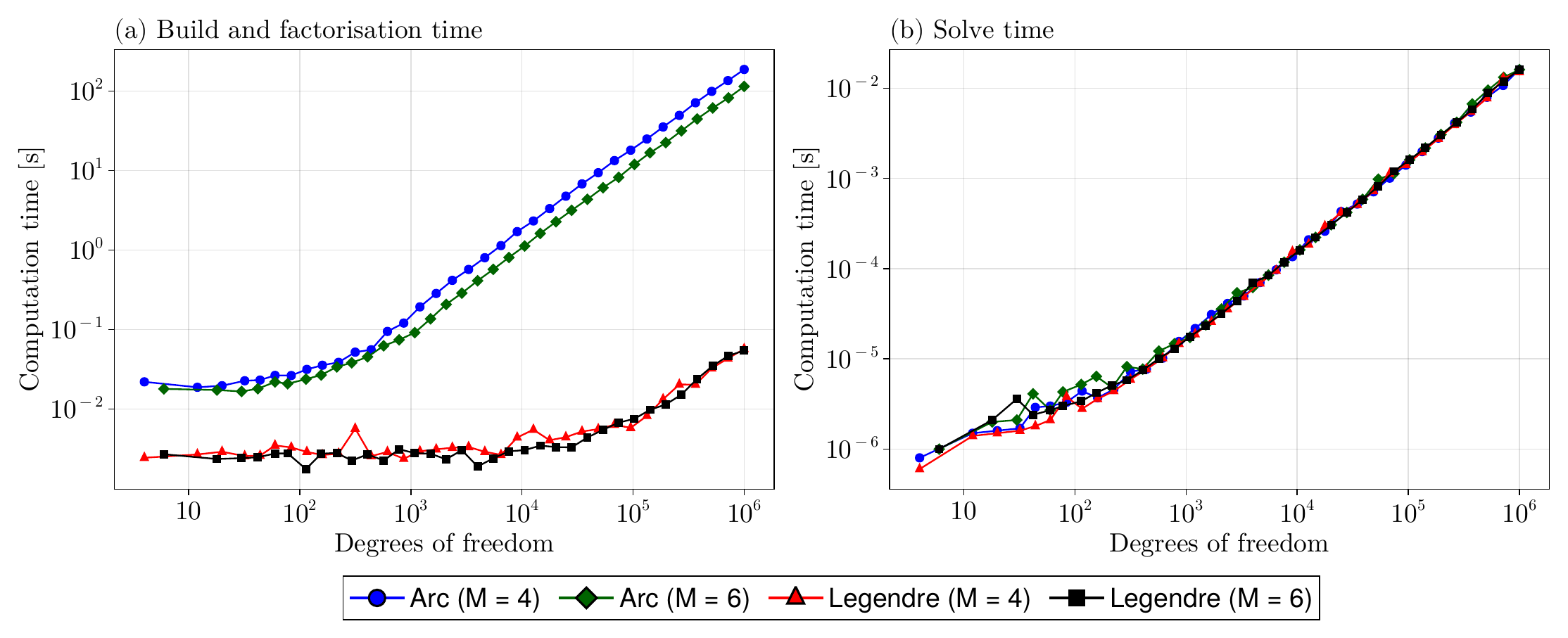}
\caption{\rev{Computational performance results for the screened Poisson equation \eqref{eq:screened_poisson} $-u''(\theta) + \omega^2 u(\theta) = f(\theta)$ with $\omega = 3/2$ and $f(\theta) = \exp(\cos(\theta))$, comparing the $\vb P^{(-1), \bm\theta}$ and $\vb W^{(-1), \bm\theta}$ bases using $M=4$ and $M=6$ equally spaced points for the grid $\bm\theta$. In (a) we show the time to build and factorise the left-hand side \eqref{eq:weak_screened_poisson_vec} $-\Delta^{(-1), \bm\theta} + \omega^2 M^{(-1), \bm\theta}$ (and similarly in the $\vb W^{(-1), \bm\theta}$ basis), and (b) shows the time to solve the system \eqref{eq:weak_screened_poisson_vec} $\left(-\Delta^{(-1), \bm\theta} + \omega^2 M^{(-1), \bm\theta}\right)\vb u = \left[R_{(-1)}^{(0), \bm\theta}\right]\tran M^{(0), \bm\theta}\vb f$ (and similarly in the $\vb W^{(-1), \bm\theta}$ basis), using a varying degrees of freedom. The timings in each plot show the optimal complexity we expect, with unit slopes for each curve. The performance in building the discretisations for the $\vb W^{(-1), \bm\theta}$ is better than that in the $\vb P^{(-1), \bm\theta}$ basis, however, this cost is only incurred once in problems involving time-stepping or for problems reusing the same grid $\bm\theta$. Furthermore, as discussed in Section \ref{sec:analysis__}, our $\vb P^{(-1), \bm\theta}$ basis requires fewer degrees of freedom making this different less important.}}
\label{fig:timing_results}
\end{figure}

\section{Conclusion}\label{sec:conc}

In this work we introduced a one-dimensional periodic basis for solving differential equations with periodic boundary conditions. The operators derived with this basis have a special banded structure that makes them efficient to work with. The basis can be used to solve problems \rev{in optimal complexity, even those with discontinuities} that a standard approach with a Fourier basis cannot handle, and has better numerical and convergence properties than other orthogonal polynomial approaches such as with a piecewise integrated Legendre basis from \cite{babuvska1983lecture, knook2024quasi}. Our numerical examples demonstrated these differences, and also highlighted that our solution is able to better preserve periodicity when integrating overtime.

Our arc polynomial basis, and its piecewise variant, \rev{could be used to solve more complicated problems than those considered here. Nonlinear equations such as Burgers' equation or the nonlinear Schr\"odinger equation could be considered, using a pseudospectral approach by efficiently transforming from values on a grid to polynomial coefficients, and from polynomial coefficients to values on a grid, similar to the method used in Section 6 of \cite{knook2024quasi} for solving Burgers' equation with the integrated Legendre basis.} The basis also generalises simply to problems other than one-dimensional differential equations. \colr{Using a tensor product basis, we could apply our basis to problems over domains such as tensor product domains, periodic strips, and a torus. \rev{Problems that periodic in both dimensions would also be simple, applying a tensor product basis with our arc polynomials.} Similar to \cite{knook2024quasi}, these higher dimensional problems could be combined with an alternating direction implicit iteration for solving the associated linear systems to achieve quasi-optimal complexity.} A subject of future work will be the consideration of problems in sectors of the form $\{0 < r < 1,\, |\theta| < \varphi\}$, where the arc polynomials can be used to generalise Zernike polynomials effectively to this domain. \rev{One complication with this approach is that, while one may form tensor-product bases such as $C_{ij}^{(h)}(r, \theta) = P_i^{(0, 1)}(2r-1)p_j^{(0, h)}(\theta)$ and $S_{ij}^{(h)}(r, \theta) = P_i^{(0,1)}(2r-1)q_j^{(0,h)}(\theta)$, these are not polynomials in Cartesian coordinates ($x$ and $y$). As a consequence, the associated operators may not be sparse, particularly when used in an $hp$-FEM framework, and there may be regularity issues at $r=0$ that require special care. Instead, an approach similar to that used in \cite{fasondini023orthogonal} is necessary if one wants polynomials in Cartesian coordinates}, using the Gram--Schmidt process on $C_{ij}^{(h)}$ and $S_{ij}^{(h)}$ to construct the orthogonal polynomials on the sector. Similar extensions could be done to annuli, with generalisations similar to those in \cite{papadopoulos2024building, papadopoulos2024sparse_}.  

\section*{Acknowledgements}
DV is grateful for the financial support of a President's PhD scholarship from Imperial College. This work was supported in part by the EPSRC grant EP/T022132/1 “Spectral element methods for fractional differential equations, with applications in applied analysis and medical imaging”.

% === BIBLIOGRAPHY INSERTED FROM .bbl ===

% === END BIBLIOGRAPHY ===

% \appendix

\clearpage

\appendix 
\setcounter{lemma}{0}
\renewcommand{\thelemma}{\Alph{section}\arabic{lemma}}
\setcounter{definition}{0}
\renewcommand{\thedefinition}{\Alph{section}\arabic{definition}}
\setcounter{theorem}{0}
\renewcommand{\thetheorem}{\Alph{section}\arabic{theorem}}
\setcounter{proposition}{0}
\renewcommand{\theproposition}{\Alph{section}\arabic{proposition}}

\section{Computing $\alpha^{t, (a, b, c)}$ and $\beta^{t, (a, b, c)}$}
\label{app:linear_coefficients}
In this appendix we consider the computation of the coefficients $\alpha^{t, (a, b, c)}$ and $\beta^{t, (a, b, c)}$ defining $P_1^{t, (a, b, c)}(x) = \beta^{t, (a, b, c)}(x-\alpha^{t, (a, b, c)}$. To start, we denote by $\Gamma^{t, (a, b, c)}$ the integral
\begin{equation}\label{eqapp:weight_integral}
\Gamma^{t, (a, b, c)} = \int_0^1 w^{t, (a, b, c)}(x) \dx = \mathrm B(a+1, b+1)t^c{}_2F_1\left(a+1, -c; a + b + 2, t^{-1}\right),
\end{equation}
where $\mathrm B$ is the beta function and ${}_2F_1$ is the Gaussian hypergeometric function. This identity \eqref{eqapp:weight_integral} could be derived by relating the integral to the generalised beta distribution \cite{mcdonald1995generalization}.

Using the Gram--Schmidt procedure with the inner product \eqref{eq:def:inp},
we obtain
\begin{align*}
\tilde{P}_1^{t, (a, b, c)}(x) &= x - \frac{\inp{x}{1}^{t, (a, b, c)}}{\|1\|_{t, (a, b, c)}^2} = x - \alpha^{t, (a, b, c)}, \quad \text{where} \quad \alpha^{t, (a, b, c)} = \frac{\Gamma^{t, (a+1, b, c)}}{\Gamma^{t, (a, b, c)}}.
\end{align*}
We can obtain $P_1^{t, (a, b, c)}$ from this by noting that all the polynomials have the same norm, meaning $\|P_n^{t, (a, b, c)}\|_{t, (a, b, c)}^2 = \|1\|_{t, (a, b, c)}^2 = \Gamma^{t, (a, b, c)}$ for all $n \geq 0$ \cite{papadopoulos2024building}. Thus,
\begin{align*}
P_1^{t, (a, b, c)}(x) &= \sqrt{\frac{\|1\|_{t,(a,b,c)}^2}{\left\|\tilde P_1^{t,(a,b,c)}\right\|_{t,(a,b,c)}^2}}\tilde P_1^{t, (a, b, c)} \\
&= \sqrt{\frac{\Gamma^{t, (a, b, c)}}{\int_0^1 \left(x - \alpha^{t, (a, b, c)}\right)^2w^{t, (a, b, c)}(x) \dx}}\left(x - \alpha^{t, (a, b, c)}\right) \\ 
&= \sqrt{\frac{\Gamma^{t, (a, b, c)}}{\int_0^1 \left[x^2 - 2\alpha^{t, (a, b, c)}x + \left(\alpha^{t, (a, b, c)}\right)^2\right]w^{t, (a, b, c)}(x)}}\left(x - \alpha^{t, (a, b, c)}\right) \\
&= \sqrt{\frac{\Gamma^{t, (a, b, c)}}{\Gamma^{t, (a+2, b, c)} - 2\alpha^{t, (a, b, c)}\Gamma^{t, (a+1, b, c)} + \left(\alpha^{t, (a, b, c)}\right)^2\Gamma^{t, (a, b, c)}}}\left(x - \alpha^{t, (a, b, c)}\right) \\
&= \beta^{t, (a, b, c)}\left(x - \alpha^{t, (a, b, c)}\right) ~~ \text{where} ~~ \beta^{t, (a, b, c)} = \sqrt{\frac{\Gamma^{t, (a, b, c)}}{\Gamma^{t, (a+2, b, c)} - 2\alpha^{t, (a, b, c)}\Gamma^{t, (a+1, b, c)} + \left(\alpha^{t, (a, b, c)}\right)^2\Gamma^{t, (a, b, c)}}}.
\end{align*}

\clearpage

\section[ $D_{(b), q}^{(b+1), q}$ is lower bidiagonal]{ $D_{(b), \vb q}^{(b+1), \vb p}$ is lower bidiagonal}\label{app:dbqbpqb_lowerbidiag}

In this appendix we complete the details of the proof of Proposition \ref{prop:diffmatpqqp}. We need to show that $D_{(b), \vb q}^{(b+1), \vb p}$ is a lower bidiagonal matrix, which is not immediately obvious since the individual operators for $x\partial_y$ and $y\partial_x$ are dense. There is likely a simpler way to express $D_{(b), \vb q}^{(b+1), \vb p}$ than \eqref{eq:diff_mat_arcq} that makes this bidiagonal structure clear, but here we prove it using inner products. In particular, consider
\[
\odv*{q_n^{(b)}(x, y)}{\theta} = \sum_{j=0}^\infty d_{nj}p_j^{(b+1)}(x, y).
\]
The numerator of $d_{nj}$, denoted $\gamma_{nj}$, is given by 
\begin{align}
\gamma_{nj} &= \inp*{\odv*{q_n^{(b)}}{\theta}}{p_j^{(b+1)}}^{(b+1, h)}\nonumber \\
&= 2q_n^{(b)}(\varphi)p_j^{(b+1)}(\varphi)(\cos\varphi - h)^{b+1} - \int_{-\varphi}^\varphi q_n^{(b)}(\theta)\odv*{\left[p_j^{(b+1)}(\theta)(\cos\theta-h)^{b+1}\right]}{\theta} \dthe.\label{eq:gammanj}
\end{align}
If $b \neq -1$, the first term in \eqref{eq:gammanj} is zero since $\cos\varphi - h = \cos\varphi - \cos\varphi = 0$. Otherwise, see that
\begin{align*}
q_n^{(-1)}(\varphi) &= \sin\varphi P_{n-1}^{\tau, (1/2, -1, 1/2)}(1) = \begin{cases} \sin\varphi & n = 1, \\ 0 & \text{otherwise}.
\end{cases}
\end{align*}
We will consider the $(b, n) = (-1, 1)$ case last and assume that this first term is zero in what follows. To evaluate the integral we first note that, using $\cos\theta - h = (1-h)(1-\sigma)$, we have the weighted derivative formula
\begin{equation}
\odv*{\left[\left(\cos\theta-h\right)^{b+1}\vb p^{(b+1)}\right]}{\theta} = \left(\cos\theta-h\right)^b \vb q^{(b)}D_{\mathrm b, (-1/2, b+1, -1/2)}^{\tau, (1/2, b, 1/2)},
\end{equation}
and this matrix $D_{\mathrm b, (-1/2, b+1, -1/2)}^{\tau, (1/2, b, 1/2)}$ is an upper bidiagonal matrix \cite{papadopoulos2024building}. Thus, 
\[
\odv*{\left[\left(\cos\theta-h\right)^{b+1}p_j^{(b+1)}(\theta)\right]}{\theta}  = \left(\cos\theta-h\right)^b\left[\xi_{j-1,j}q_j^{(b)}(\theta) + \xi_{jj}q_{j+1}^{(b)}\right],
\]
where $\xi_{j-1,j}$ and $\xi_{jj}$ are the non-zero entries in the $j$th column of $(1-h)^b D_{\mathrm b, (-1/2,b+1,-1/2)}^{\tau,(1/2,b,1/2)}$. Putting this relationship into \eqref{eq:gammanj}, assuming $(b, n) \neq (-1, 1)$,
\begin{equation}
\gamma_{nj} = -\xi_{j-1,j}\inp{q_n^{(b)}}{q_j^{(b)}}^{(b, h)} - \xi_{jj}\inp{q_n^{(b)}}{q_{j+1}^{(b)}}^{(b, h)}.
\end{equation}
Thus, $\gamma_{nj}$ is zero whenever $n \neq j$ or $n \neq j + 1$. In particular, $d_{nj} \neq 0$ only for $n \neq j$ or $n \neq j +1$, showing that $D_{(b), \vb q}^{(b+1), \vb p}$ has this lower bidiagonal structure.

We still need to consider the $b = -1$ and $n = 1$ case. Here, $\mathrm dq_1^{(-1)}/\mathrm d\theta = \cos\theta$, so
\begin{align*}
\gamma_{1j} &= \int_{-\varphi}^\varphi \cos\theta p_j^{(0)}(\theta) \dthe.
\end{align*}
Writing out $\cos(\theta)$ as a linear combination of $p_0^{(0)}(\theta) = 1$ and $p_1^{(0)}(\theta) = P_1^{\tau,(-1/2,0,-1/2)}[(\cos\theta-1)/(h-1)]$, it is easy to show that
\begin{equation}\label{app:cos_expansion_res}
\cos \theta = \left[1 + (h-1)\alpha^{\tau, (-1/2, 0, -1/2)}\right]p_0^{(0)}(\theta) + \frac{h-1}{\beta^{\tau, (-1/2, 0, -1/2)}}p_1^{(0)}(\theta),
\end{equation}
which implies that $\gamma_{1j} = 0$ whenever $j > 1$; note that $\beta^{\tau,(-1/2,0,-1/2)} \neq 0$ since that would otherwise imply that $p_1^{(0)}$ is identically zero. We therefore see that $d_{1j}$ is non-zero only for $j=1$ and $j=2$ when $b=-1$ and $n=1$. Thus, $D_{(b), \vb q}^{(b+1), \vb p}$ is also a lower bidiagonal matrix in this case, and this completes the proof. 

\clearpage 

\section[Computing $R_{(-1)}^{(0), \theta}$]{Computing $R_{(-1)}^{(0), \bm\theta}$}\label{app:computing_connection_matrix_piecewise}

We compute $R_{(-1)}^{(0), \bm\theta}$ using \eqref{eq:piecewise_arc_connection_matrix_inp_form}, deriving the formula in \eqref{eq:connection_matrix_formula_piecewise_arc}, in this appendix. For notational simplicity, we denote $\vb P = \vb P^{(0), \bm\theta}$ and $\vb Q = \vb P^{(-1), \bm\theta}$. To start, consider $\vb P_{mj}\tran\vb H$. Since the $\vb P_{mj}$ are orthogonal to all polynomials with degree less than $m$, and since the hat functions are of degree $1$,
\[
\vb P_{mj}\tran\vb H = 0, \quad m > 1,
\]
meaning in the first column of $\vb P\tran\vb Q$ we only need $\vb P_{02}\tran\vb H$, $\vb P_{11}\tran\vb H$, and $\vb P_{12}\tran\vb H$. We start with $\vb P_{02}\tran\vb H$:
\begin{align*}
\int_I P_{02;i}^{(0), \bm\theta}(\theta)\phi_j(\theta) \dthe &= \int_{E_i} \underbrace{p_0^{(0, h_i)}(a_i(\theta))}_1 \phi_j(\theta) \dthe = \int_{E_i} \phi_j(\theta) \dthe.
\end{align*}
This last integral is non-zero only for $i = j$ or $j = i + 1$. If $i = j$ then, using \eqref{eq:hat_function_psi2},
\begin{align}
\int_I P_{02;i}^{(0), \bm\theta}(\theta)\phi_i(\theta) \dthe &= \frac12 \int_{E_i} 1 + \csc(\ell_i)\left[\left(\sin\theta_i + \theta_{i+1}\right) \cos \theta - \left(\cos \theta_i + \cos\theta_{i+1}\right) \sin \theta\right] \dthe = \frac{\ell_i}{2}. \label{eq:app:mass_matrix_block11_i=j}
\end{align}
Similarly, $j = i + 1$ gives
\begin{align}
\int_I P_{02;i}^{(0), \bm\theta} \phi_{i+1}(\theta) \dthe &= \frac12\int_{E_i} 1 - \csc(\ell_i)\left[\left(\sin \theta_i + \sin\theta_{i+1}\right)\cos\theta - \left(\cos\theta_i + \cos\theta_{i+1}\right)\sin\theta\right] \dthe = \frac{\ell_i}{2}. \label{eq:app:mass_matrix_block11_j=i+1}
\end{align}
Note that these results are for $i = 1, \ldots, n$, so $\vb P_{02}\tran\vb H$ also has a non-zero value in the lower-left corner instead of being upper bidiagonal. Thus,
\begin{equation}
\vb P_{02}\tran\vb H = \frac12
\begin{bmatrix}
\ell_1 & \ell_1 &        &        &             &    \\
 & \ell_2 & \ell_2 &        &             &           \\
       &  & \ell_3 & \ddots &             &           \\
       &        & & \ddots & \ell_{n-2} &           \\
       &        &        &  & \ell_{n-1} & \ell_{n-1} \\
\ell_n &        &        &        &  & \ell_n   
\end{bmatrix}.
\end{equation}

Now consider $\vb P_{11}\tran\vb H$. Here,
\begin{align*}
\int_I P_{11;i}^{(0), \bm\theta}(\theta)\phi_j(\theta) \dthe &= \int_{E_i} \underbrace{q_1^{(0, h_i)}\left(a_i(\theta)\right)}_{\left.y\right|_{a_i(\theta)}} \phi_j(\theta) \dthe \\
&= \int_{E_i} \sin\left(\theta - \frac{\theta_i + \theta_{i+1}}{2}\right) \phi_j(\theta) \dthe.
\end{align*}
In the case $i = j$,
\begin{align}
\int_I P_{11;i}^{(0), \bm\theta}(\theta)\phi_i(\theta) \dthe &= \frac12\int_{E_i} \sin\left(\theta - \frac{\theta_i + \theta_{i+1}}{2}\right)\left\{1 + \csc(\ell_i)\left[\left(\sin\theta_i + \theta_{i+1}\right) \cos \theta - \left(\cos \theta_i + \cos\theta_{i+1}\right) \sin \theta\right]\right\} \dthe \nonumber \\
&= \frac{\sin(\ell_i/2) + \sin(3\ell_i/2) - 2\ell_i\cos\ell_i}{4\sin\ell_i} \nonumber \\
&= \frac{\sin \ell_i - \ell_i}{4\sin(\ell_i/2)}.\label{eq:app:mass_matrix_block21_i=j}
\end{align}
Next, $j = i + 1$ gives
\begin{align}
\int_I P_{11;i}^{(0), \bm\theta}(\theta)\phi_{i+1}(\theta) \dthe &= \frac12\int_{E_i} \sin\left(\theta - \frac{\theta_i + \theta_{i+1}}{2}\right) \left\{1 - \csc(\ell_i)\left[\left(\sin \theta_i + \sin\theta_{i+1}\right)\cos\theta - \left(\cos\theta_i + \cos\theta_{i+1}\right)\sin\theta\right] \right\} \dthe \nonumber \\
&= \frac{\ell_i - \sin\ell_i}{4\sin(\ell_i/2)}.\label{eq:app:mass_matrix_block21_j=i+1}
\end{align}
Thus, letting $\xi_i = (\sin\ell_i-\ell_i)/\sin(\ell_i/2)$,
\begin{equation}
\vb P_{11}\tran\vb H = \frac14\begin{bmatrix} 
\xi_1 & -\xi_1 \\
& \xi_2 & -\xi_2 \\
& & \xi_3 & \ddots \\
&&&\ddots & -\xi_{n-2} \\
&&& & \xi_{n-1} & -\xi_{n-1} \\
-\xi_n &&&&&\xi_n 
\end{bmatrix}.
\end{equation}

The final block to consider in the first column is $\vb P_{12}\tran\vb H$. We have
\begin{align*}
\int_I P_{12;i}^{(0), \bm\theta}(\theta)\phi_j(\theta) \dthe &= \int_{E_i} p_1^{(0, h_i)}(a_i(\theta))\phi_j(\theta) \dthe.
\end{align*}
Writing $P_1^{\tau_i, (-1/2, 0, -1/2)}(\sigma) = \beta(x - \alpha)$, where the coefficients $\beta$ and $\alpha$ can be computed from the results in \ref{app:linear_coefficients}, we have
\[
p_1^{(0, h_i)}(a_i(\theta)) = \beta\left[\frac{\cos\left(\theta - \frac{\theta_i + \theta_{i+1}}{2}\right) - 1}{h_i - 1} - \alpha\right].
\]
Working through the details, we can eventually show that $\vb P_{12}\tran\vb H = \vb 0$.

Now we need the remaining columns. All the blocks in these columns are diagonal since they involve integrals between functions supported only across a single element. Consider
\begin{equation}
\vb P_{m_1j}\tran\vb Q_{m_2k} = \left\{\int_{E_i} P_{m_1j;i}^{(0), \bm\theta}(a_i(\theta))P_{m_2k;\ell}^{(-1), \bm\theta}(a_\ell(\theta)) \dthe \right\}_{i,\ell=1,\ldots,n}.
\end{equation}
By orthogonality, the integrals with $j \neq k$ are all zero, and so we only need to consider the cases $j=k=1$ or $j=k=2$. Note that $m_1 = 0, 1, \ldots$ and $m_2 = 1, 2, \ldots$.  First taking $j=k=1$ and, since the blocks are all diagonal, $i=\ell$,
\begin{align}
\int_{E_i} P_{m_11;i}^{(0), \bm\theta}(a_i(\theta))P_{m_21;i}^{(-1), \bm\theta}(a_i(\theta)) \dthe &= \int_{-\varphi_i}^{\varphi_i} P_{m_11;i}^{(0), \bm\theta}(\theta)P_{m_21;i}^{(-1), \bm\theta} \dthe = \int_{-\varphi_i}^{\varphi_i} q_{m_1}^{(0, h_i)}(\theta)q_{m_2}^{(-1, h_i)}(\theta) \dthe.\nonumber
\end{align}
We can use \eqref{eq:arc_connection_p2} to write 
\[
q_{m_2}^{(-1, h_i)}(\theta) = b_{m_2-2,m_2-1}^{(0, h_i)} q_{m_2-1}^{(0,h_i)}(\theta) + b_{m_2-1,m_2-1}^{(0, h_i)}q_{m_2}^{(0, h_i)}(\theta),
\]
 Thus,
\begin{align}
\int_{E_i} P_{m_11;i}^{(0),\bm\theta}(a_i(\theta))P_{m_21;i}^{(-1), \bm\theta}(a_i(\theta)) \dthe &= \int_{-\varphi_i}^{\varphi_i} q_{m_1}^{(0, h_i)}\left[b_{m_2-2,m_2-1}^{(-1,h_i)}q_{m_2-1}^{(0,h_i)}(\theta) + b_{m_2-1,m_2-1}^{(-1,h_i)}q_{m_2}^{(0, h_i)}(\theta)\right] \dthe\nonumber \\
&= b_{m_2-2,m_2-1}^{(-1, h_i)}\inp{q_{m_1}^{(0, h_i)}}{q_{m_2-1}^{(0, h_i)}}^{(0, h_i)} + b_{m_2-1,m_2-1}^{(-1, h_i)}\inp{q_{m_1}^{(0, h_i)}}{q_{m_2}^{(0, h_i)}}^{(0, h_i)}.\nonumber
\end{align}
Using orthogonality,
\begin{align}\label{eq:app:arc_connection_m11m21}
\int_{E_i} P_{m_11;i}^{(0),\bm\theta}(a_i(\theta))P_{m_21;i}^{(-1), \bm\theta}(a_i(\theta)) \dthe &= \begin{cases} b_{m_1-1,m_1}^{(-1, h_i)}\inp{q_{m_1}^{(0, h_i)}}{q_{m_1}^{(0,h_i)}}^{(0, h_i)} & m_1 = m_2 - 1, \\ 
b_{m_1-1,m_1-1}^{(-1, h_i)}\inp{q_{m_1}^{(0, h_i)}}{q_{m_1}^{(0, h_i)}}^{(0,h_i)} & m_1=m_2, \\
0 & \text{otherwise}. \end{cases} 
\end{align}
The remaining norms can be computed using \eqref{eq:arc_polynomial_mass_matrix_formulae_pq}.

Now we consider $j = k = 2$.
\begin{align}
\int_{E_i} P_{m_12; i}^{(0), \bm\theta}(a_i(\theta))P_{m_22; i}^{(0), \bm\theta}(a_i(\theta)) \dthe &= \int_{-\varphi_i}^{\varphi_i} p_{m_1}^{(0, h_i)}(\theta)p_{m_2}^{(-1, h_i)}(\theta) \dthe \nonumber \\
&= \int_{-\varphi_i}^{\varphi_i} p_{m_1}^{(0, h_i)}(\theta)\left[a_{m_2-1,m_2}^{(-1, h_i)}p_{m_2-1}^{(0, h_i)}(\theta) + a_{m_2m_2}^{(-1,h_i)}p_{m_2}^{(0,h_i)}\right] \dthe \nonumber \\
&= a_{m_2-1,m_2}^{(-1, h_i)}\inp{p_{m_1}^{(0, h_i)}}{p_{m_2-1}^{(0, h_i)}}^{(0, h_i)} + a_{m_2m_2}^{(-1, h_i)}\inp{p_{m_1}^{(0, h_i)}}{p_{m_2}^{(0, h_i)}}^{(0, h_i)} \nonumber \\
&= \begin{cases} a_{m_1,m_1+1}^{(-1, h_i)}\inp{p_{m_1}^{(0, h_i)}}{p_{m_1}^{(0, h_i)}}^{(0, h_i)} & m_1 = m_2 - 1, \\ 
a_{m_2m_2}^{(-1, h_i)}\inp{p_{m_1}^{(0, h_i)}}{p_{m_1}^{(0, h_i)}}^{(0, h_i)} & m_1 = m_2, \\
0 & \text{otherwise}. \end{cases} \label{eq:app:arc_connection_m12m22}
\end{align}

We have now computed all the blocks needed in the matrix $\vb P\tran\vb Q$. Recall the definition for the matrices $\|\vb p_m^{(0), \bm\theta}\|^2$ and $\|\vb q_m^{(0), \bm\theta}\|^2$ from \eqref{eq:piecewise_arc_mass_matrix_formula_repr_pm_def},
\begin{equation}
\|\vb p_m^{(0), \bm\theta}\|^2 := \diag\left(\|p_m^{(0, h_1)}\|_{(0, h_1)}^2, \ldots, \|p_m^{(0, h_n)}\|_{(0, h_n)}^2\right) \in \mathbb R^{n\times n},
\end{equation}
and similarly for $\|\vb q_m^{(0), \bm\theta}\|^2$. Next, define
\begin{equation}
B_{m -1, m}^{(-1), \bm\theta} = \diag\left(b_{m-1,m}^{(-1,h_1)}, \ldots, b_{m-1,m}^{(-1,h_n)}\right) \in \mathbb R^{n \times n},
\end{equation}
with similar definitions for $B_{m-1,m-1}^{(-1), \bm\theta}$, $A_{m,m+1}^{(-1), \bm\theta}$, and $A_{mm}^{(-1), \bm\theta}$. With this notation, we have
\begin{align}
\vb P_{m1}\tran\vb Q_{m+1,1} &= B_{m-1,m}^{(-1), \bm\theta}\|\vb q_m^{(0), \bm\theta}\|^2, \\
\vb P_{m1}\tran\vb Q_{m1} &= B_{m-1,m-1}^{(-1), \bm\theta}\|\vb q_m^{(0), \bm\theta}\|^2, \\
\vb P_{m2}\tran\vb Q_{m+1,2} &= A_{m, m+1}^{(-1), \bm\theta}\|\vb p_m^{(0), \bm\theta}\|^2, \\
\vb P_{m2}\tran\vb Q_{m2} &= A_{mm}^{(-1), \bm\theta}\|\vb p_m^{(0), \bm\theta}\|^2.
\end{align}
We can use these expressions to write 
\begin{equation}
\vb P\tran\vb Q = \begin{bmatrix}
\vb P_{02}\tran\vb H & A_{01}^{(-1),\bm\theta}\|\vb p_0^{(0),\bm\theta}\|^2 &                           &                           &                           &        \\
\vb P_{11}\tran\vb H &  \vb 0                         & B_{01}^{(-1), \bm\theta}\|\vb q_1^{(0),\bm\theta}\|^2 &                           &                           &        \\
                     & A_{11}^{(-1),\bm\theta}\|\vb p_1^{(0),\bm\theta}\|^2 &                          \vb 0 & A_{12}^{(-1),\bm\theta}\|\vb p_1^{(0),\bm\theta}\|^2 &                           &        \\
                     &                           & B_{11}^{(-1),\bm\theta}\|\vb q_2^{(0),\bm\theta}\|^2 &   \vb 0                        & B_{12}^{(-1),\bm\theta}\|\vb q_2^{(0),\bm\theta}\|^2 &        \\
                     &                           &                           & A_{22}^{(-1),\bm\theta}\|\vb p_2^{(0),\bm\theta}\|^2 &    \vb 0                       & \ddots \\
                     &                           &                           &                           & B_{22}^{(-1),\bm\theta}\|\vb q_3^{(0),\bm\theta}\|^2 & \ddots \\
                     &                           &                           &                           &                           & \ddots
\end{bmatrix}
\end{equation}

Now consider $[M^{(0), \bm\theta}]^{-1}(\vb P\tran\vb Q)$. The matrix $M^{(0), \bm\theta}$ was already given in \eqref{eq:piecewise_arc_mass_matrix_formula_repr}. Taking the inverse of this matrix and multiplying on the left, we finally obtain
\begin{equation}\label{eq:app:connection_matrix_formula_piecewise_arc}
R_{(-1)}^{(0), \bm\theta} =\begin{bmatrix}
\|\vb p_0^{(0), \bm\theta}\|^{-2}\vb P_{02}\tran\vb H & A_{01}^{(-1),\bm\theta} &                           &                           &                           &        \\
\|\vb q_1^{(0), \bm\theta}\|^{-2}\vb P_{11}\tran\vb H & \vb 0                          & B_{01}^{(-1), \bm\theta} &                           &                           &        \\
                     & A_{11}^{(-1),\bm\theta} &     \vb 0                      & A_{12}^{(-1),\bm\theta} &                           &        \\
                     &                           & B_{11}^{(-1),\bm\theta} &                          \vb 0 & B_{12}^{(-1),\bm\theta} &        \\
                     &                           &                           & A_{22}^{(-1),\bm\theta} &      \vb 0                     & \ddots \\
                     &                           &                           &                           & B_{22}^{(-1),\bm\theta} & \ddots \\
                     &                           &                           &                           &                           & \ddots
\end{bmatrix}.
\end{equation}

\clearpage

\section[Computing $D_{(-1)}^{(0), \theta}$]{Computing $D_{(-1)}^{(0), \bm\theta}$}\label{app:computing_differentiation_matrix_piecewise}

In this appendix we compute $D_{(-1)}^{(0), \bm\theta}$ using \eqref{eq:piecewise_arc_differentiation_matrix_inp_form} to derive \eqref{eq:differentiation_matrix_formula_piecewise_arc}. As in \ref{app:computing_connection_matrix_piecewise}, we let $\vb P = \vb P^{(0), \bm\theta}$ and $\vb Q = \vb P^{(-1), \bm\theta}$ so that $D_{(-1)}^{(0), \bm\theta} = [M^{(0),\bm\theta}]^{-1}\vb P\tran\vb Q^\prime$, letting primes denote derivatives with respect to $\theta$. 

We start by considering $\vb P\tran\vb Q^\prime$. Similarly to $R_{(-1)}^{(0), \bm\theta}$, in the first column we only need to consider $\vb P_{02}\tran\vb H^\prime$, $\vb P_{11}\tran\vb H^\prime$, and $\vb P_{12}\tran\vb H^\prime$. Using \eqref{eq:hat_function_psi1}--\eqref{eq:hat_function_psi2}, we find, with some simplifying,
\begin{align}
\odv*{\psi_i^{(1)}(\theta)}{\theta} &= h_{i-1}\csc(\ell_{i-1})\cos\left(\theta - \frac{\theta_{i-1}+\theta_i}{2}\right) = h_{i-1}\csc(\ell_{i-1})\cos a_{i-1}(\theta), \\
\odv*{\psi_i^{(2)}(\theta)}{\theta} &= -h_i\csc(\ell_i)\cos\left(\theta-\frac{\theta_i+\theta_{i+1}}{2}\right) = -h_i\csc(\ell_{i})\cos a_{i}(\theta).
\end{align}
Now let us compute $\vb P_{02}\tran\vb H^\prime$.
\begin{align}
\int_I P_{02;i}^{(0), \bm\theta}(\theta)\phi_j^\prime(\theta) \dthe &= \int_{E_i} \underbrace{p_0^{(0, h_i)}(a_i(\theta))}_1\phi_j^\prime(\theta)  \dthe = \phi_j(\theta_{i+1}) - \phi_j(\theta_i) = \begin{cases} -1 & j = i, \\ 1 & j = i + 1, \\ 0 & \text{otherwise}. \end{cases}
\end{align}
Thus,
\begin{equation}
\vb P_{02}\tran\vb H^\prime = \begin{bmatrix}
-1 & 1 \\
& -1 & 1 \\
& & -1 & \ddots \\
&&&\ddots& 1 \\
&&&&-1&1\\
1&&&&&-1
\end{bmatrix}.
\end{equation}
Next, for $\vb P_{11}\tran\vb H^\prime$, 
\begin{align}
\int_I P_{11;i}^{(0), \bm\theta}(\theta)\phi_j^\prime(\theta) \dthe &= \int_{E_i} \underbrace{q_1^{(0, h_i)}(a_i(\theta))}_{\left.y\right|_{a_i(\theta)}}\phi_j^\prime(\theta) \dthe \nonumber\\&= \int_{E_i} \sin(a_i(\theta))\phi_j'(\theta) \dthe \nonumber\\&= \begin{cases} -\frac12\csc(\ell_i)\int_{-\varphi_i}^{\varphi_i} \sin(\theta)\cos(\theta) \dthe & i = j, \\ 
\frac12\csc(\ell_i)\int_{-\varphi_i}^{\varphi_i} \sin(\theta)\cos(\theta) \dthe & i = j + 1, \\ 0 & \text{otherwise} \end{cases} \nonumber\\& = 0.
\end{align}
Hence, $\vb P_{11}\tran\vb H^\prime = \vb 0$. The last block in this column to consider is $\vb P_{12}\tran\vb H^\prime$. For this computation, recall from \ref{app:linear_coefficients} that we have
\begin{align}
p_1^{(0, h_i)}(\theta) &= P_1^{\tau_i, (-1/2,0,-1/2)}\left(\frac{\cos\theta-1}{h_i-1}\right) \nonumber \\ 
&= \beta^{\tau_i, (-1/2, 0, -1/2)}\left(\frac{\cos \theta - 1}{h_i - 1} - \alpha^{\tau_i, (-1/2, 0, -1/2)}\right).
\end{align}
Then,
\begin{align}
\int_I P_{12}^{(0), \bm\theta}(\theta)\phi_j'(\theta) \dthe &= \int_{E_i} p_1^{(0, h_i)}(a_i(\theta)) \phi_j'(\theta) \dthe. \nonumber 
\end{align}
If $i = j$,
\begin{align}
\int_I P_{12}^{(0), \bm\theta}(\theta)\phi_i'(\theta) \dthe &= -h_i\csc(\ell_i)\beta^{\tau_i, \left(-\frac12,0,-\frac12\right)}\int_{-\varphi_i}^{\varphi_i} \left(\frac{\cos \theta - 1}{h_i - 1} - \alpha^{\tau_i, \left(-\frac12,0,-\frac12\right)}\right) \cos \theta \dthe \nonumber \\
&= h_i\csc(\ell_i) \beta^{\tau_i, \left(-\frac12,0,-\frac12\right)}\left[\sin(\varphi_i)\left(2\alpha^{\tau_i, \left(-\frac12,0,-\frac12\right)} - 1\right) - \frac{\sin(\varphi_i) - \varphi_i}{1 - h_i}\right].
\label{eq:zeta_i_defn}\end{align}
The case with $i = j + 1$ is similar, with $\int_I P_{12}^{(0), \bm\theta}(\theta)\phi_{i+1}'(\theta)\dthe = -\int_I P_{12}^{(0), \bm\theta}(\theta)\phi_i'(\theta) \dthe$. Defining $\zeta_i$ to be the quantity in \eqref{eq:zeta_i_defn}, we have
\begin{equation}
\vb P_{12}\tran\vb H^\prime = \begin{bmatrix} 
\zeta_1 & -\zeta_1 \\
& \zeta_2 & -\zeta_2 \\
& & \zeta_3 & \ddots \\
&&&\ddots & -\zeta_{n-2} \\
&&&\ddots & \zeta_{n-1} & -\zeta_{n-1} \\
-\zeta_n &&&&&\zeta_n 
\end{bmatrix}
\end{equation}

Now let us consider the blocks
\begin{equation}
\vb P_{m_1j}\tran\vb Q_{m_2k} = \left\{\int_{E_i} P_{m_1j;i}^{(0), \bm\theta}(a_i(\theta))\left(\odv*{P_{m_2k;\ell}^{(-1), \bm\theta}(a_\ell(\theta))}{\theta}\right) \dthe\right\}_{i,\ell=1,\ldots,n}.
\end{equation}
Writing $\mathrm d\vb p^{(b, h)}/\mathrm d\theta = \vb q^{(b+1, h)}D_{(b), \vb p}^{(b+1, h), \vb q}$ and $\mathrm d\vb q^{(b)}/\mathrm d\theta = \vb p^{(b+1, h)}D_{(b), \vb q}^{(b+1, h), \vb p}$, we see that the integrals above are zero whenever $j = k$ from orthogonality. Moreover, as with $R_{(-1)}^{(0), \bm\theta}$, the non-zero blocks are all diagonal. To start, consider $j = 1$ and $k = 2$.
\begin{align*}
\int_{E_i} P_{m_11;i}^{(0), \bm\theta}(a_i(\theta))\left(\odv*{P_{m_22;i}^{(-1), \bm\theta}(a_i(\theta))}{\theta}\right) \dthe &= \int_{-\varphi_i}^{\varphi_i} q_{m_1}^{(0, h_i)}(\theta)\odv*{p_{m_2}^{(-1, h_i)}(\theta)}{\theta} \dthe.
\end{align*}
Using \eqref{eq:diff_mat_arcp_relation}, we can write
\begin{equation}
\odv*{p_m^{(-1, h_i)}(\theta)}{\theta} = d_{m-1, m+1, \vb p}^{(0, h_i), \vb q}q_{m-1}^{(0, h_i)} + d_{m, m+1, \vb p}^{(0, h_i), \vb q}q_m^{(0, h_i)}.
\end{equation}
Thus,
\begin{align}
\int_{E_i} P_{m_11;i}^{(0), \bm\theta}(a_i(\theta))\left(\odv*{P_{m_22;i}^{(-1),\bm\theta}(a_i(\theta))}{\theta}\right)\dthe &= \int_{-\varphi_i}^{\varphi_i} q_{m_1}^{(0, h_i)}(\theta)\left[d_{m_2-1, m_2+1, \vb p}^{(0, h_i), \vb q}q_{m_2-1}^{(0, h_i)} + d_{m_2, m_2+1, \vb p}^{(0, h_i), \vb q}q_{m_2}^{(0, h_i)}\right] \dthe \nonumber \\
&= d_{m_2-1,m_2+1, \vb p}^{(0,h_i), \vb q}\inp{q_{m_1}^{(0, h_i)}}{q_{m_2-1}^{(0, h_i)}}^{(0, h_i)} + d_{m_2,m_2+1,\vb p}^{(0, h_i), \vb q}\inp{q_{m_1}^{(0, h_i)}}{q_{m_2}^{(0, h_i)}}^{(0, h_i)} \nonumber \\
&= \begin{cases} d_{m_1,m_1+2,\vb p}^{(0, h_i), \vb q}\inp{q_{m_1}^{(0, h_i)}}{q_{m_1}^{(0, h_i)}}^{(0, h_i)} & m_1 = m_2 - 1, \\
d_{m_1,m_1+1,\vb p}^{(0,h_i),\vb q}\inp{q_{m_1}^{(0,h_i)}}{q_{m_1}^{(0,h_i)}}^{(0,h_i)} & m_1 = m_2, \\
0 & \text{otherwise}. \end{cases}
\end{align}
Thus,
\begin{align}
\vb P_{m1}\tran\vb Q_{m+1,2}^\prime &= D_{m,m+2, \vb p}^{(0), \bm\theta, \vb q}\|\vb q_m^{(0), \bm\theta}\|^2, \label{eq:app:diffmatform1} \\
\vb P_{m1}\tran\vb Q_{m2}^\prime &= D_{m,m+1, \vb p}^{(0), \bm\theta, \vb q}\|\vb q_m^{(0), \bm\theta}\|^2, \label{eq:app:diffmatform2}
\end{align}
where
\begin{align}\label{eq:app:diffmatdvecdefn}
D_{m,m+2, \vb p}^{(0), \bm\theta, \vb q} = \diag\left(d_{m,m+2, \vb p}^{(0, h_1), \vb q}, \ldots, d_{m,m+2, \vb p}^{(0, h_n), \vb q}\right) \in \mathbb R^{n \times n},
\end{align}
and similarly for $D_{m,m+1,\vb p}^{(0), \bm\theta, \vb q}$. 

Next, consider $j = 2$ and $k = 1$. Using \eqref{eq:diff_mat_arcq_relation},
\begin{equation}
\odv*{q_{m}^{(-1,h_i)}(\theta)}{\theta} = d_{m,m,\vb q}^{(0, h_i), \vb p}p_{m-1}^{(0, h_i)}(\theta) + d_{m+1,m,\vb q}^{(0, h_i), \vb p}p_m^{(0, h_i)}(\theta).
\end{equation}
The integrals are then
\begin{align}
\int_{E_i} P_{m_12;i}^{(0),\bm\theta}(a_i(\theta))\left(\odv*{P_{m_21;i}^{(-1),\bm\theta}(a_i(\theta))}{\theta}\right) \dthe &= \int_{-\varphi_i}^{\varphi_i} p_{m_1}^{(0,h_i)}(\theta)\odv*{q_{m_2}^{(-1,h_i)}(\theta)}{\theta} \dthe \nonumber \\
&= \int_{-\varphi_i}^{\varphi_i} p_{m_1}^{(0, h_i)}(\theta)\left[d_{m_2,m_2,\vb q}^{(0, h_i), \vb p}p_{m_2-1}^{(0, h_i)}(\theta) + d_{m_2+1,m_2,\vb q}^{(0, h_i), \vb p}p_{m_2}^{(0, h_i)}(\theta)\right] \dthe \nonumber \\
&= \begin{cases} 
d_{m_1+1,m_1+1,\vb q}^{(0,h_i), \vb p}\inp{p_{m_1}^{(0, h_i)}}{p_{m_1}^{(0, h_i)}}^{(0,h_i)} & m_1 = m_2 - 1, \\ 
d_{m_1+1,m_1,\vb q}^{(0, h_i), \vb p}\inp{p_{m_1}^{(0, h_i)}}{p_{m_1}^{(0, h_i)}}^{(0, h_i)} & m_1 = m_2, \\
0 & \text{otherwise}. \end{cases}
\end{align}
This gives
\begin{align}
\vb P_{m2}\tran\vb Q_{m+1,1}^\prime &= D_{m+1,m+1,\vb q}^{(0), \bm\theta, \vb p}\|\vb p_m^{(0), \bm\theta}\|^2, \label{eq:app:diffmatform3} \\
\vb P_{m2}\tran\vb Q_{m1}^\prime &= D_{m+1,m,\vb q}^{(0), \bm\theta, \vb p}\|\vb p_m^{(0), \bm\theta}\|^2, \label{eq:app:diffmatform4}
\end{align}
defining $D_{m+1,m+1,\vb q}^{(0), \bm\theta, \vb p}$ and $D_{m+1,m,\vb q}^{(0), \bm\theta, \vb p}$ similarly to \eqref{eq:app:diffmatdvecdefn}. Putting \eqref{eq:app:diffmatform1}--\eqref{eq:app:diffmatform4} together gives
\begin{equation}
\vb P\tran\vb Q^\prime = \begin{bmatrix} \vb P_{02}\tran\vb H^\prime &                                                                        \vb 0 &  \vb 0                                                                        &                                                                         &                                                                      &                                                                         &        \\ \vb 0
                            & D_{1,2,\vb p}^{(0), \bm\theta, \vb q}\|\vb q_1^{(0), \bm\theta}\|^2 &      \vb 0                                                                   & D_{1,3,\vb p}^{(0), \bm\theta, \vb q}\|\vb q_1^{(0), \bm\theta}\|^2 &                                                                      &                                                                         &        \\
\vb P_{12}\tran\vb H^\prime &                                                                        \vb 0 & D_{2,2,\vb q}^{(0), \bm\theta, \vb p}\|\vb p_1^{(0), \bm\theta}\|^2 &                                                                       \vb 0 &       \vb 0                                                               &                                                                         &        \\
                            &                                                                        \vb 0 &                                          \vb 0                               & D_{2,3,\vb p}^{(0), \bm\theta, \vb q}\|\vb q_2^{(0), \bm\theta}\|^2 &                                                                     \vb 0 & D_{2,4,\vb p}^{(0), \bm\theta, \vb q}\|\vb q_2^{(0), \bm\theta}\|^2 &        \\
                            &                                                                         & D_{3,2,\vb q}^{(0), \bm\theta, \vb p}\|\vb p_2^{(0), \bm\theta}\|^2 &                                                                        \vb 0 & D_{3,3,\vb q}^{(0),\bm\theta,\vb p}\|\vb p_2^{(0),\bm\theta}\|^2 &                                                                        \vb 0 & \ddots \\
                            &                                                                         &                                                                         &                                                                         \vb 0&                                                                     \vb 0 & D_{3,4,\vb p}^{(0), \bm\theta, \vb q}\|\vb q_3^{(0), \bm\theta}\|^2 & \ddots \\
                            &                                                                         &                                                                         &                                                                         & D_{4,3,\vb q}^{(0),\bm\theta,\vb p}\|\vb p_3^{(0),\bm\theta}\|^2 &                                                                        \vb 0 & \ddots \\
                            &                                                                         &                                                                         &                                                                         &                                                                      &                                                                         \vb 0& \ddots \\
                            &                                                                         &                                                                         &                                                                         &                                                                      &                                                                         & \ddots \end{bmatrix}.
\end{equation}
Multiplying on the left by $[M^{(0), \bm\theta}]^{-1}$, we obtain
\begin{equation}
D_{(-1)}^{(0), \bm\theta} = \begin{bmatrix} \|\vb p_0^{(0), \bm\theta}\|^{-2}\vb P_{02}\tran\vb H^\prime &                                                                        \vb 0 &           \vb 0                                                              &                                                                         &                                                                      &                                                                         &        \\
                      \vb 0      & D_{1,2,\vb p}^{(0), \bm\theta, \vb q} &                                                                        \vb 0 & D_{1,3,\vb p}^{(0), \bm\theta, \vb q} &                                                                      &                                                                         &        \\
\|\vb p_1^{(0), \bm\theta}\|^{-2}\vb P_{12}\tran\vb H^\prime &                                                                        \vb 0 & D_{2,2,\vb q}^{(0), \bm\theta, \vb p} &                                                                        \vb 0 & \vb 0                                                                      &                                                                         &        \\
                            &                                                                         \vb 0&                                                                        \vb 0 & D_{2,3,\vb p}^{(0), \bm\theta, \vb q} &                                                                     \vb 0 & D_{2,4,\vb p}^{(0), \bm\theta, \vb q} &        \\
                            &                                                                         & D_{3,2,\vb q}^{(0), \bm\theta, \vb p} &                                                                        \vb 0 & D_{3,3,\vb q}^{(0),\bm\theta,\vb p} &                                                                        \vb 0 & \ddots \\
                            &                                                                         &                                                                         &                                                                         \vb 0&                                                                    \vb 0  & D_{3,4,\vb p}^{(0), \bm\theta, \vb q} & \ddots \\
                            &                                                                         &                                                                         &                                                                         & D_{4,3,\vb q}^{(0),\bm\theta,\vb p} &                                                                        \vb 0 & \ddots \\
                            &                                                                         &                                                                         &                                                                         &                                                                      &                                                                         \vb 0& \ddots \\
                            &                                                                         &                                                                         &                                                                         &                                                                      &                                                                         & \ddots \end{bmatrix}.
\end{equation}

\clearpage

\section{Proof of Lemma \ref{lem:expandfinthetaexparc}}\label{app:proof_finite_arc_expansion}

In this appendix, we prove Lemma \ref{lem:expandfinthetaexparc} and give expressions for the coefficients in the $\cos(n\theta)$ and $\sin(n\theta)$ expansions.

\begin{lemma}\label{lem:cosnthetaexparc}
Let $n \in \mathbb N$ and $h \in (-1, 1)$. We can write
\begin{equation}\label{eq:cosntheta_expansion_arc}
\cos(n\theta) = \frac{1}{\mu_{00}}\sum_{j=0}^n \mu_{jn}p_j^{(0, h)}(\theta), \quad |\theta| < \varphi,
\end{equation}
where $\cos\varphi = h$ and the coefficients $\mu_{jn}$ satisfy the recurrence
\begin{equation}\label{eq:cos_gamma_recur}
\mu_{jn} = 2\left[1 + (h-1)a_j^{\tau, (-1/2, 0, -1/2)}\right]\mu_{j,n-1} - \mu_{j, n-2} + 2(h-1)\left[c_j^{\tau,(-1/2, 0, -1/2)} \mu_{j-1,n-1} + b_j^{\tau,(-1/2,0,-1/2)}\mu_{j+1,n-1}\right],
\end{equation}
 where $a_j^{\tau, (-1/2, 0, -1/2)}, b_j^{\tau,(-1/2,0,-1/2)}$, and $c_j^{\tau,(-1/2,0,-1/2)}$ are the coefficients defined in Proposition \ref{prop1}, and $\tau = 2/(1-h)$. For $j > n$ or $j < 0$, $\mu_{jn} = 0$. The initial values for the recurrence \eqref{eq:cos_gamma_recur} are $\mu_{00} = 4\arccsc(\!\!\sqrt\tau)$, $\mu_{01} = \mu_{00}[1+(h-1)\alpha^{\tau,(-1/2,0,-1/2)}]$, and $\mu_{11}=\mu_{00}(h-1)/\beta^{\tau,(-1/2,0,-1/2)}$. 
\end{lemma}

\begin{proof}
In what follows, we will denote $p_j(\theta) := p_j^{(0, h)}(\theta), P_j(\sigma) := P_j^{\tau, (-1/2, 0, -1/2)}(\sigma)$, and $\inp{\cdot}{\cdot} := \inp{\cdot}{\cdot}^{\tau,(-1/2,0,-1/2)}$. Now, let us start by considering the $n =0$ and $n = 1$ cases. Since $p_0(\theta) = 1$, we immediately get $\mu_{00} = 4\arccsc(\!\!\sqrt\tau)$. This $\mu_{00}$ is the denominator of all the coefficients in the expansion \eqref{eq:cosntheta_expansion_arc} since it is given by $\|p_j\|_{(0,h)}^2$ from \eqref{eq:arc_polynomial_mass_matrix_formulae_pq}. The $n=1$ case can be derived using the expansion we already found in \eqref{app:cos_expansion_res}, giving the stated expressions for $\mu_{01}$ and $\mu_{11}$.

Now consider $n > 1$. For this case, note that from orthogonality we have $\mu_{jn} = \inp{p_j(\theta)}{\cos(n\theta)}^{(0, h)}$. Letting $T_n(x) := \cos(n \arccos x)$ denote the $n$th degree Chebyshev polynomial of the first kind, we obtain 
\begin{align}
\mu_{jn} &= \int_{-\varphi}^\varphi p_j(\theta)\cos(n\theta) \dthe = 2\int_0^1 P_j(\sigma)T_n(\zeta)w^{\tau,(-1/2,0,-1/2)}(\sigma) \d\sigma = 2\inp{P_j(\sigma)}{T_n(\zeta)},
\end{align}
where $\zeta := 1 + (h-1)\sigma$. Using $T_n(x) = 2xT_{n-1}(x) - T_{n-2}(x)$ \cite{gautschi2004orthogonal},
\begin{align*}
\mu_{jn} &= 2\mu_{j,n-1} - \mu_{j,n-2} + 2(h-1)\inp{P_j(\sigma)}{\sigma T_{n-1}(\zeta)}.
\end{align*}
Then, writing $\inp{P_j(\sigma)}{\sigma T_{n-1}(\zeta)} = \inp{\sigma P_j(\sigma)}{T_{n-1}(\zeta)}$ together with the three-term recurrence for $P_j$ and simplifying gives \eqref{eq:cos_gamma_recur}, from which we can also see that $\mu_{jn} = 0$ whenever $j > n$ so that the expansion \eqref{eq:cosntheta_expansion_arc} is finite.
\end{proof}

\begin{lemma}\label{lem:sinnthetaexparc}
Using similar notation as in Lemma \ref{lem:cosnthetaexparc} and taking $n \in \mathbb Z^+$, we can write
\begin{equation}
\sin(n\theta) = \frac{1}{\eta_{00}}\sum_{j=1}^n \eta_{jn}q_j^{(0, h)}(\theta), \quad |\theta| < \varphi,
\end{equation}
where the coefficients $\eta_{jn}$ satisfy the recurrence
\begin{equation}
\eta_{jn} = 2\left[1 + (h-1)a_{j-1}^{\tau,(1/2,0,1/2)}\right]\eta_{j,n-1} - \eta_{j,n-2} + 2(h-1)\left[c_{j-1}^{\tau,(1/2,0,1/2)}\eta_{j-1,n-1} + b_{j-1}^{\tau,(1/2,0,1/2)}\eta_{j+1,n-1}\right].
\end{equation}
When $j > n$ or $j < 1$, $\eta_{jn} = 0$. The initial values for the recurrence \eqref{eq:cos_gamma_recur} are $\eta_{11} = (1-h)^2[\tau^2\arccsc(\!\!\sqrt\tau) + (2-\tau)\sqrt{\tau-1}]/2$, $\eta_{12} = 2\eta_{11}[1+(h-1)\alpha^{\tau,(1/2,0,1/2)}]$, and $\eta_{22} = 2\eta_{11}(h-1)/\beta^{\tau,(1/2,0,1/2)}$.
\end{lemma}

\begin{proof}
The proof is essentially the same as in Lemma \ref{lem:cosnthetaexparc}, except the $\eta_{jn}$ are given by
\begin{equation}
\eta_{jn} = \int_{-\varphi}^\varphi q_j^{(0, h)}(\theta)\sin(n\theta) \dthe = 2(1-h)^2\inp*{P_{j-1}^{\tau,(1/2,0,1/2)}}{U_{n-1}(\zeta)}^{\tau,(1/2,0,1/2)},
\end{equation}
where $\zeta := 1 + (h-1)\sigma$ and $U_{n-1}(x) := \sin(n\arccos x)/\sqrt{1-x^2}$ is the Chebyshev polynomial of the second kind that also satisfies $U_n(x) = 2xU_{n-1}(x) - U_{n-2}(x)$ \cite{gautschi2004orthogonal}. The initial values for the recurrence are obtained by noting that $\sin(\theta) = q_1^{(0, h)}(\theta)$ and $\sin(2\theta) = 2\left(1+(h-1)\alpha^{\tau,(1/2,0,1/2)}\right)q_1^{(0, h)}(\theta) + 2(h-1)q_2^{(0, h)}(\theta)/\beta^{\tau,(1/2,0,1/2)}$. 
\end{proof}

\clearpage

\section{Definition of the periodic piecewise integrated Legendre basis}\label{app:integ_legen_def}
The definition we use for the $\vb W^{(-1),\bm\theta}$ basis, based on \cite{knook2024quasi}, is given below.

\begin{definition}[\cite{knook2024quasi}]
Let $C_k^{(\lambda)}(\theta)$ denote the $k$th degree ultraspherical polynomial orthogonal with respect to $(1-\theta^2)^{\lambda-1/2}$ on $[-1, 1]$ for $\lambda > -1/2$, $\lambda \neq 0$, normalised such that $C_k^{(\lambda)}(\theta) = 2^k(\lambda)_k\theta^k/k! + \mathcal O(\theta^{k-1})$, where $(\lambda)_k$ is the Pochhammer symbol. Define the $k$th degree \emph{integrated Legendre polynomial} $W_k^{(-1)}(\theta)$ by
\begin{equation}
W_k^{(-1)}(\theta) := \frac{(1 - \theta^2)C_k^{(3/2)}(\theta)}{(k+1)(k+2)}.
\end{equation}
Let $W_k^{(0)}(\theta) \equiv C_k^{(1/2)}(\theta)$ denote the $k$th degree Legendre polynomial. Given a grid $\bm\theta = (\theta_1,\ldots,\theta_{n+1})$, define the polynomials for $b \in \{-1, 0\}$
\begin{equation}
W_{ki}^{(b), \bm\theta}(\theta) = \begin{cases} W_k^{(b)}(b_i(\theta)) & \theta \in E_i, \\
0 & \text{otherwise} \end{cases} 
\end{equation}
over each element, where $b_i(\theta) := (2\theta - \theta_i - \theta_{i+1})/\ell_i$ maps $E_i$ into $[-1, 1]$. Define the hat functions $\vb H_L^{\bm\theta}$ by 
\begin{align*}
\vb H_L^{\bm\theta} = \begin{bmatrix} h_1^{\bm\theta}(\theta) & \cdots & h_n^{\bm\theta}(\theta) \end{bmatrix},
\end{align*}
where
\begin{equation}
h_1^{\bm\theta}(\theta) = \begin{cases}
(\theta_2 - \theta)/\ell_1 & \theta \in E_1, \\
(\theta -\theta_n)/\ell_n & \theta \in E_n, \\
0 & \text{otherwise}, \end{cases} \quad h_j^{\bm\theta}(\theta) = \begin{cases} (\theta -\theta_{j-1})/\ell_{j-1} & \theta \in E_{j-1}, \\ 
(\theta_{j+1} - \theta)/\ell_j & \theta \in E_j, \\ 0 & \text{otherwise} \end{cases} \quad j=2,\ldots,n.
\end{equation}
With these definitions, we define the bases
\begin{align*}
\vb W^{(0), \bm\theta} &:= \begin{bmatrix} \vb W_0^{(0), \bm\theta} & \vb W_1^{(0), \bm\theta} & \cdots \end{bmatrix}, \\
\vb W^{(-1), \bm\theta} &:= \begin{bmatrix} \vb H_L^{\bm\theta} & \vb W_0^{(-1), \bm\theta} & \vb W_1^{(-1), \bm\theta} & \cdots \end{bmatrix},
\end{align*}
where $\vb W_k^{(b), \bm\theta} := (W_{k1}^{(b),\bm\theta},\ldots,W_{kn}^{(b),\bm\theta})\tran$.
\end{definition}

\clearpage

\section{Multiplication matrices}\label{app:multmat}

In this appendix we consider the problem of expanding $a(x, y)\vb P^{(b,h)}(x, y)$ in the $\vb P^{(b,h)}(x, y)$ basis. In particular, the computation of the matrix $J_a$ such that $a(x, y)\vb P^{(b,h)}(x, y) = \vb P^{(b,h)}(x, y)J_a$. For what follows, we will need the definition of the Jacobi matrix given below (in general, this is the non-symmetric transposed version of the typical definition of a Jacobi matrix \cite{olver2020fast}, although $J^{t, (a, b, c)}$ defined below is symmetric).

\begin{definition}[\cite{olver2020fast}]\label{def:app:jacmat}
Let $\vb P$ be a basis of orthogonal polynomials. The \emph{Jacobi matrix} is the matrix $J$ such that $x\vb P(x) = \vb P(x)J$. This matrix corresponds to the three-term recurrence for the polynomials, denoted by $p_n$, given by
\begin{align*}
xp_n(x) &= c_np_{n-1}(x) + a_np_n(x) + b_np_{n+1}(x), 
\end{align*}
where $c_0 = 0$. We will use superscripts to refer to Jacobi matrices for the semiclassical Jacobi polynomials, i.e. $x\vb P^{t, (a, b, c)} = \vb P^{t, (a, b, c)}J^{t, (a, b, c)}$.
\end{definition}

Let's now work on computing $J_a$. First, the following lemma concerning weighted conversion between semiclassical Jacobi polynomials both with $b = -1$ will be necessary.

\begin{lemma}\label{lem:weight_conv_jac:app}
The matrix $L_{\mathrm a\mathrm c, (a, -1, c)}^{t, (a-1, -1, c-1)}$ is a $(2, 0)$ banded matrix given by 
\begin{equation}
L_{\mathrm a\mathrm c, (a, -1, c)}^{t, (a-1, -1, c-1)} = \begin{bmatrix}
t - 1 & \vb 0\tran \\
\ell_1 \vb e_1 + \ell_2\vb e_2 & L_{\mathrm a\mathrm c, (a, 1, c)}^{t, (a-1,1,c-1)}
\end{bmatrix},
\end{equation}
where $\ell_1 = 1 + \alpha^{t, (a-1, 1, c-1)} - t$, $\ell_3 = 1/\beta^{t, (a-1, 1, c-1)}$, $\vb e_1 = (1, 0, \ldots)\tran$, and $\vb e_2 = (0,1,0,\ldots)\tran$.
\end{lemma}

\begin{proof}
Writing 
\[
L_{\mathrm a\mathrm c, (a, -1, c)}^{t, (a-1, -1, c-1)} = \begin{bmatrix} \ell & \vb 0\tran \\ \bm\ell & L \end{bmatrix},
\]
we obtain the two equations
\begin{align}
w^{t, (a, 0, c)} &=  w^{t, (a-1, 0, c-1)}\ell + w^{t, (a-1, 1, c-1)}\vb P^{t, (a-1, 1, c-1)}\bm\ell, \\
w^{t, (a, 1, c)}\vb P^{t, (a, 1, c)} &= w^{t, (a-1, 1, c-1)}\vb P^{t, (a-1, 1, c-1)}L.
\end{align}
The second equation shows that $L = L_{\mathrm a\mathrm c, (a, 1, c)}^{t, (a-1, 1, c-1)}$. For the first equation, the form of $\bm\ell$ follows from substituting $x = 0$, $x = 1$, and $x = \alpha^{t, (a-1, 1, c-1)}$. The fact that $L_{\mathrm a\mathrm c, (a, -1, c)}^{t, (a-1, -1, c-1)}$ is a $(2, 0)$ banded matrix follows from the fact that $\bm\ell$ is of the form $\bm\ell = (\ell_1, \ell_2, 0, \ldots)\tran$ and that $L_{\mathrm a\mathrm c, (a, 1, c)}^{t, (a-1, 1, c-1)}$ is a $(2, 0)$ banded matrix also \cite{papadopoulos2024building}.
\end{proof}

Now we may continue. Omitting the superscript $(b, h)$ in what follows, the first step for computing $J_a$ is to compute the products $p_n(x, y)\vb p(x, y)$, $p_n(x, y)\vb q(x, y)$, $q_n(x, y)\vb p(x, y)$, and $q_n(x, y)\vb q(x, y)$. We use the notation $\vb P^- := \vb P^{\tau, (-1/2, b, -1/2)}$ and $\vb P^+ := \vb P^{\tau, (1/2, b, 1/2)}$, and similarly define $J^{\pm}$ to be the Jacobi matrix of $\vb P^{\pm}$.
The first product is
\begin{align}
p_n(x, y)\vb p(x, y) = P_n^-(\sigma)\vb P^-(\sigma) = \vb P^-(\sigma)P_n^-(J^-) = \vb p(x, y)p_n(J^-) = \vb p(x, y)M_{n, (2, 2)},\label{eq:app:g:eM22}
\end{align}
where $M_{n, (2, 2)} := p_n(J^-)$, abusing notation slightly to write $p_n(\sigma) \equiv p_n(x, y)$ instead of the usual $p_n(\theta) \equiv p_n(\cos\theta,\sin\theta)$. Note that we have made use of the fact that $h(\sigma)\vb P^- = \vb P^-h(J^-)$ for polynomial $h$ which follows directly from the definition of the Jacobi matrix. Next, 
\begin{align}\label{eq:app:g:pnnqxy_eM21}
p_n(x, y)\vb q(x, y) = yP_n^-(\sigma)\vb P^+(\sigma) = y\vb P^+(\sigma)P_n^-(J^+) = \vb q(x, y)P_n^-(J^+) = \vb q(x, y)M_{n, (2, 1)},
\end{align}
where $M_{n, (2, 1)} := p_n(J^+)$. Continuing,
\begin{equation}\label{eq:app:g:eM12}
q_n(x, y)\vb p(x, y) = yP_{n-1}^+(\sigma)\vb P^-(\sigma) = y\vb P^-(\sigma)P_{n-1}^+(J^-) = \vb q(x, y)R_{(-1/2,b,-1/2)}^{\tau,(1/2,b,1/2)}P_{n-1}^+(J^-) = \vb q(x, y)M_{n, (1, 2)},
\end{equation}
where $M_{n, (1, 2)} := R_{(-1/2, b, -1/2)}^{\tau, (1/2, b, 1/2)}P_{n-1}^+(J^-)$. Lastly,
\begin{align*}
q_n(x, y)\vb q(x, y) &= y^2 P_{n-1}^+(\sigma)\vb P^+(\sigma) = y^2 \vb P^+(\sigma)P_{n-1}^+(J^+).
\end{align*}
Consider $y^2\vb P^+(\sigma)$. Using $y^2 = \sigma(1-h)^2(\tau-\sigma)$, we can write
\[
y^2\vb P^+ = (1-h)^2 w^{\tau,(1/2,-b,1/2)}(\sigma)\left(w^{\tau,(1/2,b,1/2)}\vb P^+\right).
\]
We know that we can write $w^{\tau,(1/2,b,1/2)}\vb P^+ = w^{\tau,(-1/2,b,-1/2)}\vb P^-L_{\mathrm a\mathrm c, (1/2,b,1/2)}^{\tau,(-1/2,b,-1/2)}$ where $L_{\mathrm a\mathrm c, (1/2,b,1/2)}^{\tau,(-1/2,b,-1/2)}$ is a $(2, 0)$ banded matrix \cite{papadopoulos2024building}, including when $b = -1$ thanks to Lemma \ref{lem:weight_conv_jac:app}. Thus, $y^2\vb P^+ = (1-h)^2\vb P^-L_{\mathrm a\mathrm c, (1/2,b,1/2)}^{\tau,(-1/2,b,-1/2)}$. This gives
\begin{equation}
q_n(x, y)\vb q(x, y) = y^2\vb P^+(\sigma)P_{n-1}^+(J^+) = (1-h)^2\vb P^-L_{\mathrm a\mathrm c, (1/2, b, 1/2)}^{\tau, (-1/2, b, -1/2)}P_{n-1}^+(J^+) = \vb pM_{n, (1, 1)},
\label{eq:app:g:e11}
\end{equation}
where $M_{n, (1, 1)} := (1-h)^2L_{\mathrm a\mathrm c, (1/2, b, 1/2)}^{\tau, (-1/2, b, -1/2)}P_{n-1}^+(J^+)$.  We summarise these results in the following lemma.

\begin{lemma}\label{lem:app:pnpprod_forms_Mij}
Given $b \geq -1$ and $|h| < 1$, we can write
\begin{align}
p_n^{(b, h)}(x, y)\vb p^{(b, h)}(x, y) &= \vb p^{(b, h)}(x, y)M^{(b, h)}_{n, (2, 2)}, \qquad\qquad M^{(b, h)}_{n, (2, 2)} := P_n^{\tau, (-1/2, b, -1/2)}\left(J^{\tau, (-1/2, b, -1/2)}\right), \label{eq:app:g:__EM22} \\
p_n^{(b, h)}(x, y)\vb q^{(b, h)}(x, y) &= \vb q^{(b, h)}(x, y)M^{(b, h)}_{n, (2, 1)}, \qquad\qquad M^{(b, h)}_{n, (2, 1)} := P_n^{\tau, (-1/2, b, -1/2)}\left(J^{\tau, (1/2, b, 1/2)}\right), \label{eq:app:g:__EM21}\\
q_n^{(b, h)}(x, y)\vb p^{(b, h)}(x, y) &= \vb q^{(b, h)}(x, y)M^{(b, h)}_{n, (1, 2)}, \qquad\qquad M^{(b, h)}_{n, (1, 2)} := R_{(-1/2,b,-1/2)}^{\tau,(1/2,b,1/2)}P_{n-1}^{\tau, (1/2, b, 1/2)}\left(J^{\tau, (-1/2, b, -1/2)}\right),  \label{eq:app:g:__EM12} \\
q_n^{(b, h)}(x, y)\vb q^{(b, h)}(x, y) &= \vb p^{(b, h)}(x, y)M^{(b, h)}_{n, (1, 1)}, \qquad\qquad M^{(b, h)}_{n, (1, 1)} := (1 - h)^2L_{\mathrm a\mathrm c, (1/2, b, 1/2)}^{\tau, (-1/2, b, -1/2)}P_{n-1}^{\tau, (1/2, b, 1/2)}\left(J^{\tau, (1/2, b, 1/2)}\right).\label{eq:app:g:__EM11}
\end{align}
\end{lemma}

Now let's consider $a(x, y)\vb P(x, y)$. Given $\vb a := (a_{02}, a_{11}, a_{12}, \ldots)\tran$ such that $a(x, y) = \vb P(x, y)\vb a$, we can write
\begin{align}\label{eq:axy_sum_an1qn_an2pn}
a(x, y) &= \sum_{n=1}^\infty a_{n1}q_n(x, y) + \sum_{n=0}^\infty a_{n2}p_n(x, y).
\end{align}
Then, using Lemma \ref{lem:app:pnpprod_forms_Mij},
\begin{align}
a(x, y)\vb p(x, y) &= \vb q(x, y)\sum_{n=1}^\infty a_{n1}M_{n, (1, 2)} + \vb p(x, y)\sum_{n=0}^\infty a_{n2}M_{n, (2, 2)}. \label{eq:app:g:ap_expa_mul_mat}\\
a(x, y)\vb q(x, y) &= \vb p(x, y)\sum_{n=1}^\infty a_{n1}M_{n, (1, 1)} + \vb q(x, y)\sum_{n=0}^\infty a_{n2}M_{n, (2, 1)}. \label{eq:app:g:aq_expq_mul_matpq}
\end{align}
Next, define $M_{(1, 2)} := \sum_{n=1}^\infty a_{n1}M_{n, (1, 2)}$ and similarly for $M_{(2, 2)}$, $M_{(1, 1)}$, and $M_{(2, 1)}$. To evaluate these matrices requires some more work. First,
\begin{align}\label{eq:app:g:form_of_pgm22}
M_{(2, 2)} &= \sum_{n=0}^\infty a_{n2}M_{n, (2, 2)} = \sum_{n=0}^\infty a_{n2}P_n^-(J^-) = b_2(J^-),
\end{align}
where $b_2(\sigma) := \sum_{n=0}^\infty a_{n2} P_n^-(\sigma)$. Next,
\begin{align}\label{eq:app:g:form_of_pgm21}
M_{(2, 1)} &= \sum_{n=0}^\infty a_{n2}M_{n, (2, 1)} = \sum_{n=0}^\infty a_{n2}P_n^-(J^+) = b_2(J^+).
\end{align}
For $M_{(1, 2)}$ we have
\begin{align}\label{eq:app:g:form_of_pgm12}
M_{(1, 2)} &= \sum_{n=1}^\infty a_{n1}M_{n, (1, 2)} = \sum_{n=1}^\infty a_{n1}R_{(-1/2, b, -1/2)}^{\tau, (1/2, b, 1/2)}P_{n-1}^+(J^-) = R_{(-1/2, b, -1/2)}^{\tau, (1/2, b, 1/2)}b_1(J^-),
\end{align}
where $b_1(\sigma) := \sum_{n=1}^\infty a_{n1}P_{n-1}^+(\sigma)$. Finally,
\begin{align}\label{eq:app:g:form_of_pgm11}
M_{(1, 1)} &= \sum_{n=1}^\infty a_{n1} M_{n, (1, 1)} = \sum_{n=1}^\infty a_{n1}(1-h)^2L_{\mathrm a\mathrm c, (1/2,b,1/2)}^{\tau,(-1/2,b,-1/2)}P_{n-1}^+(J^+) = (1-h)^2L_{\mathrm a\mathrm c, (1/2, b, 1/2)}^{\tau, (-1/2, b, -1/2)}b_1(J^+).
\end{align}
The only remaining complication involved in \eqref{eq:app:g:form_of_pgm22}--\eqref{eq:app:g:form_of_pgm11} is in the computation of $b_1(J^{\pm})$ and $b_2(J^{\pm})$. This computation is done using Clenshaw's algorithm \cite{olver2020fast} together with the three-term recurrence relationships defining $\vb P^{\pm}$, making use of the \texttt{Clenshaw} function in RecurrenceRelationshipArrays.jl \cite{RecurrenceRelationshipArrays.jl2024}. We summarise these results and give further properties of these $M$ matrices below.

\begin{lemma}\label{lem:app:g:pnpprod_form_JA_MAT_mulmat}
Given $b \geq -1$ and $|h| < 1$, suppose $a(x, y)$ has a convergent $\vb P^{(b, h)}$-series so that
\begin{equation}
a(x, y) = b_1(\sigma) + yb_1(\sigma) \equiv \vb P^{(b, h)}\vb a, 
\end{equation}
where $\sigma := (x-1)/(h-1)$, $\tau := 2/(1-h)$, $b_2(\sigma) := \sum_{n=0}^\infty a_{n2}P_n^{\tau, (-1/2, b, -1/2)}(\sigma)$, $b_1(\sigma) := \sum_{n=1}^\infty a_{n1}P_{n-1}^{\tau, (1/2, b, 1/2)}(\sigma)$, and $\vb a = (a_{02}, a_{11}, a_{12}, \ldots)\tran$. Then, using the same notation as in Lemma \ref{lem:app:pnpprod_forms_Mij},
\begin{align}
\sum_{n=0}^\infty a_{n2}M^{(b, h)}_{n, (2, 2)} &= b_2(J^{\tau, (-1/2, b, -1/2)}) =: M^{(b, h)}_{(2, 2)}, \\
\sum_{n=0}^\infty a_{n2}M^{(b, h)}_{n, (2, 1)} &= b_2(J^{\tau, (1/2, b, 1/2)}) =: M^{(b, h)}_{(2, 1)}, \\
\sum_{n=1}^\infty a_{n1}M^{(b, h)}_{n, (1, 2)} &= R_{(-1/2, b, -1/2)}^{\tau, (1/2, b, 1/2)}b_1(J^{\tau, (-1/2, b, -1/2)}) =: M_{(1, 2)}^{(b, h)}, \\
\sum_{n=1}^\infty a_{n1}M^{(b, h)}_{(1, 1)} &= (1-h)^2L_{\mathrm a\mathrm c, (1/2, b, 1/2)}^{\tau, (-1/2, b, -1/2)}b_1(J^{\tau, (1/2, b, 1/2)}) =: M_{(1, 1)}^{(b, h)},
\end{align}
and
\begin{align}
a(x, y)\vb p^{(b, h)}(x, y) &= \vb q^{(b, h)}(x, y)M^{(b, h)}_{(1, 2)} + \vb p^{(b, h)}(x, y)M^{(b, h)}_{(2, 2)}, \\
a(x, y)\vb q^{(b, h)}(x, y) &= \vb p^{(b, h)}(x, y)M_{(1, 1)}^{(b, h)} + \vb q^{(b, h)}(x, y)M_{(2, 1)}^{(b, h)}.
\end{align}
Let $\alpha_1$ and $\alpha_2$ be such that $a_{n1} = 0$ for $n > \alpha_1$ and $a_{n2} = 0$ for $n > \alpha_2$, in particular $b_1(\sigma) \equiv \sum_{n=1}^{\alpha_1} a_{n1}P_{n-1}^{\tau, (1/2, b, 1/2)}(\sigma)$ and $b_2(\sigma) \equiv \sum_{n=0}^{\alpha_2} a_{n2}P_n^{\tau, (-1/2, b, -1/2)}(\sigma)$. Then $M_{(2, 2)}^{(b, h)}$ and $M_{(2, 1)}^{(b, h)}$ are $(\alpha_2 , \alpha_2)$ banded matrices, $M_{(1, 2)}^{(b, h)}$ is a $(\alpha_1-1, \alpha_1 + 1)$ banded matrix, and $M_{(1, 1)}^{(b, h)}$ is a $(\alpha_1 + 1, \alpha_1-1)$ banded matrix.
\end{lemma}

\begin{proof}
The only unproven part of this result is the last statement concerning the bandwidths of these matrices. To prove these bandwidths, suppose we have a polynomial $h(\sigma) = \sum_{n=0}^d h_n \sigma^n$ and consider $h(J)$ where $J$ is a tridiagonal matrix. Since multiplying matrices adds their bandwidths together, the largest bandwidth matrix in the sum associated with $h(J)$ comes from $J^d$ which is a $(d, d)$ banded matrix (since tridiagonal matrices are $(1, 1)$ banded matrices). Since $b_1(\sigma) = \sum_{n=1}^{\alpha_1} a_{n1}P_{n-1}^{\tau, (1/2, b, 1/2)}$, we see that $b_1(J^{\tau, (\pm 1/2, b, \pm 1/2)})$ must be a $(\alpha_1-1, \alpha_1-1)$ banded matrix and similarly for $b_2(\sigma) = \sum_{n=0}^{\alpha_2} a_{n2}P_n^{\tau, (-1/2, b, -1/2)}(\sigma)$. The result for $M_{(1, 2)}^{(b, h)}$ and $M_{(1, 1)}^{(b, h)}$ follows from this argument together with noting that $R_{(-1/2, b, -1/2)}^{\tau, (1/2, b, 1/2)}$ and $L_{\mathrm a\mathrm c, (1/2,b,1/2)}^{\tau,(-1/2,b,-1/2)}$ are $(0, 2)$ and $(2, 0)$ banded matrices, respectively.
\end{proof}

Using this result, we can finally derive $J_a$. We see that
\begin{align*}
a(x, y)p_n(x, y) &= \vb q(x, y)M_{(1, 2)}\vb e_n + \vb p(x, y)M_{(2, 2)}\vb e_n, \quad n=0,1,2,\ldots, \\
a(x, y)q_n(x, y) &= \vb p(x, y)M_{(1, 1)}\vb e_{n-1} + \vb q(x, y)M_{(2, 1)}\vb e_{n-1}, \quad n=1,2,\ldots.
\end{align*}
Thus, letting $\vb e_0 := (1,0,\ldots)\tran$, $\vb e_1=(0,1,0,\ldots)\tran$, and so on, and letting $M_{(u, v)} = \{m_{i, j, (u, v)}\}_{i,j=0,1,2,\ldots}$, we can write, omitting the $(x, y)$ arguments,
\begin{align*}
a\vb P &= a\begin{bmatrix} p_0 & q_1 & p_1 & q_2 & p_2 & \cdots \end{bmatrix} = a\begin{bmatrix} 
\vb qM_{(1,2)}\vb e_0 + \vb pM_{(2,2)}\vb e_0 \\
\vb pM_{(1, 1)}\vb e_0 + \vb qM_{(2, 1)}\vb e_0 \\
\vb qM_{(1,2)}\vb e_1 + \vb pM_{(2,2)}\vb e_1 \\
\vb pM_{(1,1)}\vb e_1 + \vb qM_{(2, 1)}\vb e_1 \\
\vb qM_{(1, 2)}\vb e_2 + \vb pM_{(2, 2)}\vb e_2\\ 
\vdots
\end{bmatrix}\tran = \begin{bmatrix} 
\sum_{n=1}^\infty q_nm_{n-1, 0, (1, 2)} + \sum_{n=0}^\infty p_nm_{n, 0, (2, 2)} \\
\sum_{n=0}^\infty p_nm_{n, 0, (1, 1)} + \sum_{n=1}^\infty q_nm_{n-1, 0, (2, 1)} \\
\sum_{n=1}^\infty q_nm_{n-1, 1, (1, 2)} + \sum_{n=0}^\infty p_nm_{n, 1, (2, 2)} \\
\sum_{n=0}^\infty p_nm_{n, 1, (1, 1)} + \sum_{n=1}^\infty q_nm_{n-1, 1, (2, 1)} \\
\sum_{n=1}^\infty q_nm_{n-1, 2, (1, 2)} + \sum_{n=0}^\infty p_nm_{n, 2, (2, 2)} \\
\vdots
\end{bmatrix}\tran.
\end{align*}
We can write this first entry and second entry as
\begin{align*}
\sum_{n=1}^\infty q_nm_{n-1,0,(1,2)} + \sum_{n=0}^\infty p_nm_{n,0,(2,2)} &= \begin{bmatrix} p_0 & q_1 & p_1 & q_2 & p_2 & \cdots \end{bmatrix}\begin{bmatrix} m_{0,0,(2,2)} \\ m_{0,0,(1,2)} \\ m_{1,0,(2,2)} \\ m_{1,0,(1,2)} \\ m_{2,0,(2,2)} \\ \vdots \end{bmatrix}, \\
\sum_{n=0}^\infty p_nm_{n, 0, (1, 1)} + \sum_{n=1}^\infty q_nm_{n-1, 0, (2, 1)} &= \begin{bmatrix} p_0 & q_1 & p_1 & q_2 & p_2 & \cdots \end{bmatrix} \begin{bmatrix} m_{0,0,(1,1)} \\ m_{0,0,(2,1)} \\ m_{1,0,(1,1)} \\ m_{1,0,(2,1)} \\ m_{2,0,(1,1)} \end{bmatrix},
\end{align*}
respectively, and this pattern continues for the remaining entries. Thus, we finally obtain the following result. Note that the bandwidth of $J_a$ in the result below is simply the number of non-zero entries in $\vb a = (a_{02}, a_{11}, a_{12}, a_{21}, a_{22}, \ldots)\tran$, with the exception that if the last non-zero is an $a_{n1}$ term instead of an $a_{n2}$ term we need to add $1$.

\begin{proposition}\label{app:prop:mulmat_r1+}
Given $b \geq -1$ and $|h| < 1$, suppose $a(x, y)$ has a convergent $\vb P^{(b, h)}$-series so that $a(x, y) = b_1(\sigma) + yb_1(\sigma) \equiv \vb P^{(b, h)}A$, using the same notation as in Lemma \ref{lem:app:g:pnpprod_form_JA_MAT_mulmat}. Then the matrix $J_a^{(b, h)}$ defined by $a(x, y)\vb P^{(b, h)}(x, y) = \vb P^{(b, h)}(x, y)J_a^{(b, h)}$ is given by
\begin{equation}\label{eq:ja_mat_res}
J_a^{(b, h)} = \left[\begin{array}{c|cc|cc|c}
m_{0, 0, (2, 2)} & m_{0, 0, (1, 1)} & m_{0, 1, (2, 2)} & m_{0, 1, (1, 1)} & m_{0, 2, (2, 2)} &\cdots \\ \hline
m_{0, 0, (1, 2)} & m_{0, 0, (2, 1)}
 & m_{0, 1, (1, 2)} & m_{0, 1, (2, 1)} & m_{0, 2, (1, 2)} &  \cdots \\
m_{1, 0, (2, 2)} & m_{1, 0, (1, 1)} & m_{1, 1, (2, 2)} & m_{1, 1, (1, 1)} & m_{1, 2, (2, 2)} &   \cdots \\ \hline
m_{1, 0, (1, 2)} & m_{1, 0, (2, 1)} & m_{1, 1, (1, 2)} & m_{1, 1, (2, 1)} & m_{1, 2, (1, 2)} & \cdots \\
m_{2, 0, (2, 2)} & m_{2, 0, (1, 1)} & m_{2, 1, (2, 2)} & m_{2, 1, (1, 1)} & m_{2, 2, (2, 2)} &  \cdots \\ \hline
m_{2, 0, (1, 2)} & m_{2, 0, (2, 1)} & m_{2, 1, (1, 2)} & m_{2, 1, (2, 1)} & m_{2, 2, (1, 2)} & \cdots \\
m_{3, 0, (2, 2)} & m_{3, 0, (1, 1)} & m_{3, 1, (2, 2)} & m_{3, 1, (1, 1)} & m_{3, 2, (2, 2)} &  \cdots \\ \hline
m_{3, 0, (1, 2)} & m_{3, 0, (2, 1)} & m_{3, 1, (1, 2)} & m_{3, 1, (2, 1)} & m_{3, 2, (1, 2)} &  \cdots \\
m_{4, 0, (2, 2)} & m_{4, 0, (1, 1)} & m_{4, 1, (2, 2)} & m_{4, 1, (1, 1)} & m_{4, 2, (2, 2)} &  \cdots \\ \hline
\vdots           & \vdots           & \vdots           & \vdots           & \vdots                      & \ddots
\end{array}\right]
\end{equation}
using the notation $M_{(u, v)}^{(b, h)} = \{m_{i,j,(u,v)}\}_{i,j=0,1,2,\ldots}$. This matrix $J_a^{(b, h)}$ is an $(\alpha, \alpha)$ banded matrix, where $\alpha = 2\max\{\alpha_1, \alpha_2\} + \mathbf{1}[\alpha_1 \geq \alpha_2]$, $\alpha_1$ and $\alpha_2$ are as in Lemma \ref{lem:app:g:pnpprod_form_JA_MAT_mulmat}, and $\mathbf{1}[\alpha_1 \geq \alpha_2] = 1$ if $\alpha_1 \geq \alpha_2$ and $\mathbf{1}[\alpha_1 \geq \alpha_2] = 0$ otherwise.
\end{proposition}

The analogous result for $\vb P^{(b), \bm\theta}$ can be derived by applying Proposition \ref{app:prop:mulmat_r1+} over each element.


\begin{thebibliography}{10}
\expandafter\ifx\csname url\endcsname\relax
  \def\url#1{\texttt{#1}}\fi
\expandafter\ifx\csname urlprefix\endcsname\relax\def\urlprefix{URL }\fi
\expandafter\ifx\csname href\endcsname\relax
  \def\href#1#2{#2} \def\path#1{#1}\fi

\bibitem{boyd2001chebyshev}
J.~P. Boyd, {C}hebyshev and {F}ourier {S}pectral {M}ethods, 2nd Edition, Dover
  Publications, New York, USA, 2001.

\bibitem{wright2015extension}
G.~B. Wright, M.~Javed, H.~Montanelli, L.~N. Trefethen, Extension of {C}hebfun
  to periodic functions, SIAM Journal on Scientific Computing 37 (2015)
  C554--C573.

\bibitem{frenkel2002understanding}
D.~Frenkel, B.~Smit, Understanding {M}olecular {S}imulation, Academic Press,
  San Diego, 2002.

\bibitem{fausset2007periodic}
F.~Sausset, G.~Tarjus, Periodic boundary conditions on the pseudosphere,
  Journal of Physics A: Mathematical and Theoretical 40 (2007) 12873.

\bibitem{combined2006dong}
S.~Dong, G.~E. Karniadakis, A.~Ekmekci, D.~Rockwell, A combined direct
  numerical simulation-particle image velocimetry study of the turbulent near
  wake, Journal of Fluid Mechanics 569 (2006) 185--207.

\bibitem{babuvska1983lecture}
I.~Babu{\v{s}}ka, B.~A. Szab{\'o}, {L}ecture \rev{{N}otes} on {F}inite
  {E}lement {A}nalysis, 1983--1985.

\bibitem{knook2024quasi}
K.~Knook, S.~Olver, I.~P.~A. Papadopoulos, Quasi-optimal complexity $hp$-{FEM}
  for the {P}oisson equation on a rectangle, IMA Journal of Numerical Analysis
  (2025) draf102.

\bibitem{raju2014spectral}
G.~N. Raju, P.~Dutt, N.~K. Kumar, C.~S. Upadhyay, Spectral element method for
  elliptic equations with periodic boundary conditions, Applied Mathematics and
  Computation 246 (2014) 426--439.

\bibitem{los2015dealing}
M.~{\L}o{\'s}, M.~Paszy{\'n}ski, L.~Dalcin, V.~Calo, Dealing with periodic
  boundary conditions for {1D}, {2D} and {3D} isogeometric finite element
  method, Computer Methods in Materials Science 15 (2015) 213--218.

\bibitem{huybrechs2010fourier}
D.~Huybrechs, On the {F}ourier extension of nonperiodic functions, SIAM Journal
  on Numerical Analysis 47 (2010) 4326--4355.

\bibitem{olver2021orthogonal}
S.~Olver, Y.~Xu, Orthogonal structure on a quadratic curve, IMA Journal of
  Numerical Analysis 41 (2021) 206--246.

\bibitem{Driscoll2014}
T.~A. Driscoll, N.~Hale, L.~N. Trefethen,
  \href{http://www.chebfun.org/docs/guide/}{{C}hebfun {G}uide}, Pafnuty
  Publications, 2014.
\newline\urlprefix\url{http://www.chebfun.org/docs/guide/}

\bibitem{papadopoulos2024building}
I.~P.~A. Papadopoulos, T.~S. Gutleb, R.~M. Slevinsky, S.~Olver, Building
  hierarchies of semiclassical {J}acobi polynomials for spectral methods in
  annuli, SIAM Journal on Scientific Computing 46 (2024) A3448--A3476.

\bibitem{magnus1995painleve}
A.~P. Magnus, Painlev{\'e}-type differential equations for the recurrence
  coefficients of semi-classical orthogonal polynomials, Journal of
  Computational and Applied Mathematics 57 (1995) 215--237.

\bibitem{townsend2015continuous}
A.~Townsend, L.~N. Trefethen, Continuous analogues of matrix factorizations,
  Proceedings of the Royal Society A 471 (2015) 20140585.

\bibitem{guo2009generalized}
B.-Y. Guo, J.~Shen, L.-L. Wang, Generalized {J}acobi polynomials/functions and
  their applications, Applied Numerical Mathematics 59 (2009) 1011--1028.

\bibitem{SemiPoly.jl2024}
\href{https://github.com/JuliaApproximation/SemiclassicalOrthogonalPolynomials.jl}{{SemiclassicalOrthogonalPolynomials.jl}},
  v0.7.2 (2025).
\newline\urlprefix\url{https://github.com/JuliaApproximation/SemiclassicalOrthogonalPolynomials.jl}

\bibitem{olver2020fast}
S.~Olver, R.~M. Slevinsky, A.~Townsend, Fast algorithms using orthogonal
  polynomials, Acta Numerica 29 (2020) 573--699.

\bibitem{gautschi2004orthogonal}
W.~Gautschi, {O}rthogonal {P}olynomials: {C}omputation and {A}pproximation,
  Oxford University Press, Oxford, 2004.

\bibitem{ClassicalOrthogonalPolynomials.jl2024}
\href{https://github.com/JuliaApproximation/ClassicalOrthogonalPolynomials.jl}{{ClassicalOrthogonalPolynomials.jl}},
  v0.15.11 (2025).
\newline\urlprefix\url{https://github.com/JuliaApproximation/ClassicalOrthogonalPolynomials.jl}

\bibitem{aurentz2017chopping}
J.~L. Aurentz, L.~N. Trefethen, Chopping a {C}hebyshev series, ACM Transactions
  on Mathematical Software 43 (2017) 1--21.

\bibitem{schwab1998theory}
C.~Schwab, $p$-and $hp$-finite element methods: {T}heory and {A}pplications in
  {S}olid and {F}luid {M}echanics, Clarendon Press, 1998.

\bibitem{strang2011groups}
G.~Strang, Groups of banded matrices with banded inverses, Proceedings of the
  American Mathematical Society 139~(12) (2011) 4255--4264.

\bibitem{Golub2013}
G.~H. Golub, C.~F. Van~Loan, {M}atrix {C}omputations, 4th Edition, Johns
  Hopkins University Press, Baltimore, 2013.

\bibitem{trefethen2019approximation}
L.~N. Trefethen, {A}pproximation {T}heory and {A}pproximation {P}ractice,
  {Extended} Edition, SIAM, Philadelphia, 2019.

\bibitem{kuijlaars2004riemann}
A.~Kuijlaars, K.-R. Mc{L}aughlin, W.~{Van Assche}, M.~Vanlesen, The
  {R}iemann-{H}ilbert approach to strong asymptotics for orthogonal polynomials
  on $[-1, 1]$, Advances in Mathematics 188 (2004) 337--398.

\bibitem{arc}
\href{https://github.com/DanielVandH/ArcPolynomials.jl}{{ArcPolynomials.jl}},
  v0.1.0 (2025).
\newline\urlprefix\url{https://github.com/DanielVandH/ArcPolynomials.jl}

\bibitem{adams2003sobolev}
R.~A. Adams, J.~J. Fournier, Sobolev {S}paces, Elsevier, 2003.

\bibitem{ExponentialUtilities.jl2024}
\href{https://github.com/SciML/ExponentialUtilities.jl}{{ExponentialUtilities.jl}},
  v1.28.0 (2024).
\newline\urlprefix\url{https://github.com/SciML/ExponentialUtilities.jl}

\bibitem{fasondini023orthogonal}
M.~Fasondini, S.~Olver, Y.~Xu, Orthogonal polynomials on a class of planar
  algebraic curves, Studies in Applied Mathematics 151 (2023) 369--405.

\bibitem{papadopoulos2024sparse_}
I.~P.~A. Papadopoulos, S.~Olver, A sparse hierarchical $hp$-finite element
  method on disks and annuli, Journal of Scientific Computing 104 (2025).

\bibitem{mcdonald1995generalization}
J.~B. McDonald, Y.~J. Xu, A generalization of the beta distribution with
  applications, Journal of Econometrics 66 (1995) 133--152.

\bibitem{RecurrenceRelationshipArrays.jl2024}
\href{https://github.com/JuliaApproximation/RecurrenceRelationshipArrays.jl}{{RecurrenceRelationshipArrays.jl}},
  v0.1.3 (2025).
\newline\urlprefix\url{https://github.com/JuliaApproximation/RecurrenceRelationshipArrays.jl}

\end{thebibliography}
\end{document}